\documentclass{amsart}

\usepackage{amssymb,mathrsfs,amscd,graphicx, array, tikz-cd}
\usepackage{comment}
\allowdisplaybreaks

\usepackage[nocompress, noadjust]{cite}  

\usepackage{enumitem}

\usepackage{mathtools}

\usepackage[
  pdfencoding=auto, 
  psdextra,
]{hyperref}
\usepackage{letltxmacro}
\LetLtxMacro{\oldsqrt}{\sqrt}
\renewcommand{\sqrt}[2][]{\,\oldsqrt[#1]{#2}\,}

\newcommand{\bsh}{\backslash}

\def\Zp{{\bbZ}_p}

\newcommand{\wcO}{\widehat{\mathcal{O}}}
\newcommand{\whA}{\widehat{A}}

\newcommand{\Fqbar}{\overline{\mathbb{F}}_q}
\newcommand{\Fq}{\mathbb{F}_q}
\newcommand{\Fp}{\mathbb{F}_p}
\newcommand{\cc}{\mathbb{C}}
\newcommand{\dieu}{Dieudonn\'{e} }

\DeclareMathOperator{\SCl}{SCl}
\DeclareMathOperator{\Emb}{Emb}

\DeclareMathOperator{\SG}{SG}
\DeclareMathOperator{\Mass}{Mass}

\DeclareMathOperator{\Cl}{Cl}

\DeclareMathOperator{\Tp}{Tp}
\DeclareMathOperator{\Nr}{Nr}
\newcommand{\<}{\langle}   
\renewcommand{\>}{\rangle} 

\def\sfF{{\sf F}}
\def\sfV{{\sf V}}

\newcommand{\isoto}{\stackrel{\sim}{\longrightarrow}}

\newcommand{\pp}{\mathrm{pp}}
\newcommand{\un}{{\mathrm{un}}}
\newcommand{\pmm}{\mathrm{pm}}
\newcommand{\scc}{\mathrm{sc}} 
\newcommand{\sg}{\mathrm{sg}} 

\def\ul{\underline}

\DeclareMathOperator{\Isog}{Isog}
\DeclareMathOperator{\Qisog}{Qisog}
\DeclareMathOperator{\ppav}{PPAV}
\DeclareMathOperator{\ppsp}{PPSP}
\DeclareMathOperator{\Stab}{Stab}

\DeclareMathOperator{\PolSp}{PolSp}

\makeatletter
\def\@tocline#1#2#3#4#5#6#7{\relax
  \ifnum #1>\c@tocdepth 
  \else
    \par \addpenalty\@secpenalty\addvspace{#2}%
    \begingroup \hyphenpenalty\@M
    \@ifempty{#4}{%
      \@tempdima\csname r@tocindent\number#1\endcsname\relax
    }{%
      \@tempdima#4\relax
    }%
    \parindent\z@ \leftskip#3\relax \advance\leftskip\@tempdima\relax
    \rightskip\@pnumwidth plus4em \parfillskip-\@pnumwidth
    #5\leavevmode\hskip-\@tempdima
      \ifcase #1
       \or\or \hskip 1em \or \hskip 2em \else \hskip 3em \fi%
      #6\nobreak\relax
    \dotfill\hbox to\@pnumwidth{\@tocpagenum{#7}}\par
    \nobreak
    \endgroup
  \fi}
\makeatother

\usepackage{bbm}
\def\bmu{\boldsymbol \mu}

\usepackage{stmaryrd}
\newcommand{\dbr}[1]{\left\llbracket #1 \right\rrbracket}
\makeatletter

\def\greekbolds#1{%
 \@for\next:=#1\do{%
    \def\X##1;{%
     \expandafter\def\csname V##1\endcsname{\boldsymbol{\csname##1\endcsname}}
     }
   \expandafter\X\next;
  }
}

\greekbolds{alpha,beta,iota,gamma,lambda,nu,eta,Gamma,varsigma}

\def\make@bb#1{\expandafter\def
  \csname bb#1\endcsname{{\mathbb{#1}}}\ignorespaces}

\def\make@bbm#1{\expandafter\def
  \csname bb#1\endcsname{{\mathbbm{#1}}}\ignorespaces}

\def\make@bf#1{\expandafter\def\csname bf#1\endcsname{{\bf
      #1}}\ignorespaces} 

\def\make@gr#1{\expandafter\def
  \csname gr#1\endcsname{{\mathfrak{#1}}}\ignorespaces}

\def\make@scr#1{\expandafter\def
  \csname scr#1\endcsname{{\mathscr{#1}}}\ignorespaces}

\def\make@cal#1{\expandafter\def\csname cal#1\endcsname{{\mathcal
      #1}}\ignorespaces} 

\def\do@Letters#1{#1A #1B #1C #1D #1E #1F #1G #1H #1I #1J #1K #1L #1M
                 #1N #1O #1P #1Q #1R #1S #1T #1U #1V #1W #1X #1Y #1Z}
\def\do@letters#1{#1a #1b #1c #1d #1e #1f #1g #1h #1i #1j #1k #1l #1m
                 #1n #1o #1p #1q #1r #1s #1t #1u #1v #1w #1x #1y #1z}
\do@Letters\make@bb   \do@letters\make@bbm
\do@Letters\make@cal  
\do@Letters\make@scr 
\do@Letters\make@bf \do@letters\make@bf   
\do@Letters\make@gr   \do@letters\make@gr
\makeatother

\newcommand{\abs}[1]{\lvert #1 \rvert}
\newcommand{\zmod}[1]{\mathbb{Z}/ #1 \mathbb{Z}}

\newcommand{\dangle}[1]{\left\langle #1 \right\rangle}
\newcommand{\wh}{\widehat}
\newcommand{\wt}{\widetilde}
\newcommand{\Lsymb}[2]{\genfrac{(}{)}{}{}{#1}{#2}}  
\newcommand{\qalg}[3]{\left(\frac{#1, #2}{#3}\right)}

\newcommand{\ddiv}{\mid}

\DeclareMathSymbol{\twoheadrightarrow} {\mathrel}{AMSa}{"10}

\DeclareMathOperator{\id}{id}
\DeclareMathOperator{\img}{img}

\DeclareMathOperator{\ord}{ord}
\DeclareMathOperator{\Frac}{Frac}

\DeclareMathOperator{\Pic}{Pic}

\DeclareMathOperator{\Alt}{Alt}

\DeclareMathOperator{\Res}{Res}

\DeclareMathOperator{\End}{End}
\DeclareMathOperator{\Hom}{Hom}

\DeclareMathOperator{\Gal}{Gal}
\DeclareMathOperator{\Mat}{Mat}

\DeclareMathOperator{\Tr}{Tr}
\DeclareMathOperator{\Nm}{N}  

\DeclareMathOperator{\GL}{GL}
\DeclareMathOperator{\SL}{SL}

\DeclareMathOperator{\Sp}{Sp}
\DeclareMathOperator{\PGL}{PGL}
\DeclareMathOperator{\PSL}{PSL}

\newcommand{\Qbar}{\bar{\mathbb{Q}}}

\newcommand{\Z}{\mathbb Z}
\newcommand{\Q}{\mathbb Q}
\newcommand{\R}{\mathbb R}
\newcommand{\F}{\mathbb F}

\newcommand{\whD}{\widehat{D}}
\newcommand{\whF}{\widehat{F}}
\newcommand{\whK}{\widehat{K}}
\newcommand{\whO}{\widehat{O}}
\newcommand{\whI}{\widehat{I}}

\newcommand{\whZ}{\widehat{\mathbb{Z}}}
\newcommand{\whQ}{\widehat{\mathbb{Q}}}

\newcounter{thmcounter} 
\numberwithin{thmcounter}{section}  
\newtheorem{thm}[thmcounter]{Theorem}
\newtheorem{lem}[thmcounter]{Lemma}
\newtheorem{cor}[thmcounter]{Corollary}
\newtheorem{prop}[thmcounter]{Proposition}
\newtheorem*{que}{Question}

\theoremstyle{definition}
\newtheorem{defn}[thmcounter]{Definition}
\newtheorem{ex}[thmcounter]{Example}

\newtheorem{rem}[thmcounter]{Remark}

\newtheorem{sect}[thmcounter]{}

\numberwithin{equation}{section}
\numberwithin{figure}{section}
\numberwithin{table}{section}

\newtheoremstyle{notitle}  
  {}
  {}
  {\itshape}
  {}
  {}
  {\ }
  {.5em}
  {}
\theoremstyle{notitle}

\title[Superspecial abelian surfaces]{Polarized 
  superspecial simple abelian surfaces with real Weil numbers}
\author{Jiangwei Xue}

\address{(Xue) Collaborative Innovation Center of Mathematics, School of
  Mathematics and Statistics, Wuhan University, Luojiashan, 430072,
  Wuhan, Hubei, P.R. China}   

\address{(Xue) Hubei Key Laboratory of Computational Science (Wuhan
  University), Wuhan, Hubei,  430072, P.R. China.}

\email{xue\_j@whu.edu.cn}

\author{Chia-Fu Yu}
\address{(Yu) Institute of Mathematics,
  Academia Sinica, Astronomy-Mathematics
  Building, No. 1, Sec. 4, Roosevelt Road, Taipei 10617, TAIWAN.}

\address{(Yu) National Center for Theoretical Sciences, 
  Astronomy-Mathematics
  Building, No. 1, Sec. 4, Roosevelt Road, Taipei 10617, TAIWAN.}
\email{chiafu@math.sinica.edu.tw}

\begin{document}
\date{\today} 


\begin{abstract}
  Let $q$ be an odd power of a prime $p\in \mathbb{N}$, and
  $\mathrm{PPSP}(\sqrt{q})$ be the finite set of isomorphism classes
  of principally polarized superspecial abelian surfaces in the simple
  isogeny class over $\mathbb{F}_q$ corresponding to the real Weil
  $q$-numbers $\pm \sqrt{q}$.  We produce explicit formulas for
  $\mathrm{PPSP}(\sqrt{q})$ of the following  kinds: (i) the
  class number formula, i.e.~the cardinality of
  $\mathrm{PPSP}(\sqrt{q})$; (ii) the type number formula, i.e.~the
  number of endomorphism rings up to isomorphism of the underlying
  abelian surfaces of
  $\mathrm{PPSP}(\sqrt{q})$.  Similar formulas are obtained for
  other collections of  polarized superspecial members of this isogeny
  class grouped together according to their
  polarization modules. We observe several surprising identities
  involving the arithmetic genus of certain Hilbert modular surface on
  one side and the class number or type number of $(P, P_+)$-polarized
  superspecial abelian surfaces in this isogeny class on the other
  side.
\end{abstract}

\maketitle



\section{Introduction}

Let $p\in \bbN$ be a prime number, $q=p^n$ a power of $p$, and $\Fq$
the finite field with $q$ elements. Let $\Qbar$ be the algebraic
closure of $\Q$ in $\cc$.  An algebraic integer
$\pi\in \Qbar\subset \cc$ is called a \emph{Weil $q$-number} if
$\abs{\sigma(\pi)}=\sqrt{q}$ for every embedding
$\sigma: \Q(\pi)\hookrightarrow \cc$. By the Honda-Tate Theorem
\cite[Theorem~1]{tate:ht}, there is a bijection between the isogeny
classes of simple abelian varieties over $\F_q$ and the
$\Gal(\Qbar/\Q)$-conjugacy classes of Weil $q$-numbers.
Throughout this paper,  we reserve terms such as
 ``isogeny, isomorphism, endomorphism,  or 
polarization'' of abelian varieties over $\Fq$ for those defined over the base
field $\Fq$.
Let
$X_\pi/\F_q$ be a simple abelian variety in the isogeny class
corresponding to the Weil $q$-number $\pi$.  Both the dimension
$g(\pi)\coloneqq\dim(X_\pi)$ and the endomorphism algebra
$\End^0(X_\pi)\coloneqq\End(X_\pi)\otimes_\Z\Q$ are invariants of the
isogeny class and can be determined explicitly from $\pi$ (ibid.).
Recall that $\End^0(X_\pi)$ is a finite-dimensional central division
$\Q(\pi)$-algebra.

It is well known \cite[4.1]{Zarhin:finiteness-AV-1977} that for each
fixed $g\geq 1$, there are only finitely many $g$-dimensional abelian
varieties over $\F_q$ up to isomorphism.  Let $\Isog(\pi)$ be the finite
set of $\F_q$-isomorphism classes of simple abelian varieties  in the isogeny
class corresponding to $\pi$. Similarly, let $\ppav(\pi)$ be the finite
set of isomorphism classes of principally polarized abelian
varieties $(X, \lambda)/\F_q$ with $X/\F_q$ in the simple isogeny
class corresponding to $\pi$. 
The finiteness of $\ppav(\pi)$ is again well known and follows from  \cite[Appendix I,
Lemma~1]{mumford:av} (see also~\cite{Lipnowski-Tsimerman}). 
Therefore, it is natural to ask: 
\begin{que}
  What are the cardinalities of $\Isog(\pi)$ and $\ppav(\pi)$? 
\end{que}

In contrast to the algebraically closed base field case
\cite[\S23, Corollary~1]{mumford:av}, the set $\ppav(\pi)$ could be
empty. Consequently, an $\Fq$-isogeny class $\calI$ of abelian varieties is said to be {\it
principally polarizable} if there exists an abelian variety $X\in\calI$ that admits a
principal polarization over $\Fq$. The question  whether a given
$\Fq$-isogeny class $\calI$ is principally polarizable  has been investigated  
by E.~Howe in a series of papers (see \cite{Howe-Maisner-et-al2008} 
for the precise references). 
Howe, Maisner, Nart and Ritzenthaler \cite[Theorem
1]{Howe-Maisner-et-al2008} determined all isogeny classes of abelian
surfaces that are not principally polarizable.  Based on the works of
R\"uck~\cite{Ruck1990}, Maisner-Nart~\cite{maisner-nart}, and the
aforementioned result, Howe, Nart and
Ritzenthaler~\cite{Howe-Maisner-et-al2008} gave a complete solution
to the problem of characterizing the isogeny classes of abelian
surfaces over finite fields containing a Jacobian.

Let $q=p^n$ be an arbitrary prime power. The Weil $q$-numbers
$\pm \sqrt{q}$ are exceptional in several ways.  Suppose that $\pi$ is
a Weil $q$-number different from $\pm \sqrt{q}$. Then the number field
$\Q(\pi)$ is always a CM-field (i.e.~a totally imaginary quadratic
extension of a totally real number field).  From
\cite[Proposition~2.2]{xue-yu:counting-av},
\begin{equation}
  \label{eq:intro-3}
  \abs{\Isog(\pi)}=N_\pi\cdot h(\Q(\pi)), 
\end{equation}
where $N_\pi$ is a positive integer, and $h(\Q(\pi))$ is the class
number of $\Q(\pi)$. It should be mentioned that $N_\pi$ is
highly dependent on $\pi$ and can be challenging to calculate
explicitly in general. See the discussions in
\cite[\S3.2]{Lipnowski-Tsimerman} and
\cite[\S2.4]{xue-yu:counting-av}.   The proof of (\ref{eq:intro-3}) relies
on a strong approximation argument, which fails for the Weil $q$-numbers $\pm\sqrt{q}$.   The
distinction is further amplified in the case $q=p$. If $\pi$ is a Weil
$p$-number distinct from $\pm
\sqrt{p}$, then  by \cite[Theorem~6.1]{waterhouse:thesis},
\begin{equation}
  \label{eq:intro-4}
  \End^0(X_\pi)=\Q(\pi)
\end{equation}
for every abelian variety $X_\pi$ in the simple $\F_p$-isogeny class corresponding
to $\pi$, while \eqref{eq:intro-4} does not hold for the Weil $p$-numbers
$\pm \sqrt{p}$. Consequently, many theories for abelian varieties over
$\F_p$ have to make an exception for the isogeny class corresponding
to $\pm\sqrt{p}$. See \cite[\S1.3]{centeleghe-stix} and
\cite[Theorem~0.3]{Lipnowski-Tsimerman}. Recently, results for
counting $\ppav(\pi)$ with $\pi$ being a Weil $p$-number distinct from
$\pm \sqrt{p}$ are obtained in \cite{Achter-Rud-ANT2023}.

Now suppose that $\pi=\pm \sqrt{q}$ with $q=p^n$. There are two cases to consider. 
First suppose that $n$ is even. Then $X_\pi$ is a supersingular elliptic
   curve with
   \begin{equation}\label{eq:1.3}
\End^0(X_\pi)\simeq D_{p, \infty},     
   \end{equation}
where $D_{p, \infty}$ denotes the unique
   quaternion $\Q$-algebra ramified exactly at $p$ and $\infty$.
   Since each elliptic curve is equipped with the unique canonical principal
   polarization, we have a canonical bijection $\ppav(\pi)\cong
   \Isog(\pi)$. By \cite[Theorem~4.2]{waterhouse:thesis},  the endomorphism ring
   $\End(X_\pi)$ is a maximal order in $\End^0(X_\pi)$
   for every $X_\pi$, and conversely by \cite[Theorem~3.13]{waterhouse:thesis}, every maximal order in $D_{p,
     \infty}$ occurs as an endomorphism ring. 
   A classical result of Deuring \cite{Deuring-1941}
   and later re-interpreted by Waterhouse
   \cite[Theorem~4.5]{waterhouse:thesis} shows that for $\pi\in \{\pm p^{n/2}\}$
   with  $n\in 2\bbZ_{>0}$, we have
   \begin{equation} 
     \label{eq:5}
\begin{split}
\abs{\ppav(\pi)}&=\abs{\Isog(\pi)}=h(D_{p, \infty})\\
    &=\frac{p-1}{12}+\frac{1}{4} \left(1-\left(\frac{-4}{p}\right)
\right)+ \frac{1}{3} \left(1-\left(\frac{-3}{p}\right) \right).
\end{split}     
   \end{equation}
   Here $h(D_{p, \infty})$ is the class number of $D_{p, \infty}$, which is
   first computed by Eichler
   \cite{eichler-CNF-1938}. Igusa \cite{igusa} also computed
   (\ref{eq:5}) using another method. 
   In fact, we know more about $\ppav(\pi)$ than just its
   cardinality.  Let\footnote{The superscript $^\pp$ stands for ``principal polarization''.}  $t^\pp(\pi)$ be the \emph{type number} of $\ppav(\pi)$, that
is, the number of isomorphism classes of 
 endomorphism rings $\End(X)$ as the isomorphism class $[X, \lambda]$ ranges over $\ppav(\pi)$:
\begin{equation}
  \label{eq:12}
t^\pp(\pi)\coloneqq\#\big(\{\End(X)\mid [X, \lambda]\in \ppav(\pi)\}/\!{\simeq}\big).
\end{equation}
   By the above discussion, for $\pi\in \{\pm p^{n/2}\}$
   with  $n$ even,  $t^\pp(\pi)$ coincides with the \emph{type number}
   $t(D_{p, \infty})$, which  was computed by Eichler
   \cite{eichler-CNF-1938} and Deuring
   \cite{Deuring1950} using different methods.  If $p=2, 3$, then $t(D_{p, \infty})=1$, and 
\begin{equation}
\label{eq:x6}
2t(D_{p, \infty})=h(D_{p, \infty})+
\left[\frac{1}{2}+\frac{1}{4}\left(1-\Lsymb{-4}{p}\right)\left(2-\Lsymb{2}{p}\right)\right]h(-p)\quad
\text{if } p\geq 5. 
\end{equation}
See also \cite[(1.10) and (1.11) and (2.5)]{Gross-Height-L-series},
\cite[Remark~3, p.~42]{Ibukiyama-Katsura-1994} and
\cite[(1.5) and (1.6)]{yu:sp-prime}.

Next, suppose that $\pi =\pm\sqrt{q}=\pm p^{n/2}$ with $n$ odd.
It is well known that the isogeny class corresponding to
$\pi$ consists of $\F_q$-simple abelian surfaces \cite[\S1,
Examples]{tate:ht}. In particular, the two sets
$\ppav(\pi)$ and $\Isog(\pi)$ can no longer be identified with each other. Thanks to \cite[Theorem
1]{Howe-Maisner-et-al2008}, $\ppav(\pi)\neq \emptyset$.
 Similar to
the previous case as in \eqref{eq:1.3},  the endomorphism algebra of $X_\pi$ is isomorphic to the unique
quaternion $\Q(\sqrt{p})$-algebra $D_{\infty_1, \infty_2}$ ramified exactly at the two infinite
places of $\Q(\sqrt{p})$ and splits at all finite places, so we write
\begin{equation}
  \label{eq:23}
\End^0(X_\pi)\simeq D_{\infty_1, \infty_2}.
\end{equation} However, 
in the present case, $\End(X_\pi)$ is no longer
necessarily a maximal order in $\End^0(X_\pi)$  by \cite[Theorem~6.2]{waterhouse:thesis}, which causes new
difficulties.  The  explicit formula for $\abs{\Isog(\sqrt{p})}$ was previously computed by the
present authors together with Tse-Chung Yang in 
\cite{xue-yang-yu:ECNF, xue-yang-yu:num_inv}, and later 
extended to a 
formula for $\abs{\Isog(\sqrt{q})}$ by the present authors in
\cite[Theorem~4.4]{xue-yu:counting-av}.
The type number (i.e.~the number of endomorphism rings up to isomorphism) of
 $\Isog(\sqrt{p})$ is calculated in
\cite{xue-yu:type_no}.  The primary goal of the current paper is to produce the counterparts of \eqref{eq:5}
and  (\ref{eq:x6})
for the set $\ppav(\sqrt{p})$.






For every
square-free integer $d\in \bbZ$,  let $h(d)$ be the class number of the quadratic field
$\Q(\sqrt{d})$.  As usual, 
$\Lsymb{\cdot}{p}$ denotes the Legendre symbol, and $\zeta_F(s)$
denotes the
Dedekind zeta function of the quadratic real field
$F\coloneqq\Q(\sqrt{p})$. See Remark~\ref{rem:Bernoulli} for methods of
computing the special value $\zeta_F(-1)$.

\begin{thm}\label{thm:main}
  \begin{enumerate}[label=(\arabic*)]
\item  $|\ppav(\sqrt{p})|=1, 1, 2$ for $p=2,3,5$, respectively. 
  
\item For $p\geq 13$ and $p\equiv 1 \pmod 4$, 
\begin{equation}\label{eq:1}
  |\ppav(\sqrt{p})|=\left ( 9-2\left(\frac{2}{p}\right )\right )
   \frac{\zeta_F(-1)}{2}+ \frac{3h(-p)}{8}+\left (
   3+\left(\frac{2}{p}\right )\right ) \frac{h(-3p)}{6}. 
\end{equation}

\item  For $p\geq 7$ and $p\equiv 3 \pmod 4$, 
 \begin{equation}
   |\ppav(\sqrt{p})|=\frac{\zeta_F(-1)}{2}+\left(11-3\Lsymb{2}{p}\right)\frac{h(-p)}{8}+\frac{h(-3p)}{6}.
\end{equation}
  \end{enumerate}
\end{thm}
When $p\equiv 1\pmod{4}$, the set $\ppav(\sqrt{p})$ naturally
partitions into two subsets $\Lambda_1^\pp$ and $\Lambda_{16}^\pp$
according to the endomorphism ring of the underlying abelian surface by
Proposition~\ref{P.4}, and we can discuss the class number formulas
for $\Lambda_1^\pp$ and $\Lambda_{16}^\pp$ individually (see
\eqref{eq:99}--\eqref{eq:92}). This applies to the type number
formulas in Theorem~\ref{thm:type-num} below  as well.

\begin{thm}\label{thm:type-num}
The type number of $\ppav(\sqrt{p})$ is given as follows: 
  \begin{enumerate}[label=(\arabic*)]
\item  $t^\pp(\sqrt{p})=1, 1, 2$ for $p=2,3,5$, respectively. 

\item If $p\equiv 1 \pmod 4$ and $p\geq 13$, then 
\begin{equation}
t^\pp(\sqrt{p})=8 \zeta_F(-1) +
     \frac{h(-p)}{2}+ \frac{2h(-3p)}{3} . 
\end{equation}  
\item If $p\equiv 3 \pmod 4$ and $p\geq 7$, then 
\begin{equation}
t^\pp(\sqrt{p})=\frac{\zeta_F(-1)}{4}+\left(17-\Lsymb{2}{p}\right)\frac{h(-p)}{16}+\frac{h(-2p)}{8}+\frac{h(-3p)}{12}. 
\end{equation}  
\end{enumerate}
\end{thm}

Naturally, we would like to generalize Theorems~\ref{thm:main} and
\ref{thm:type-num} to Weil numbers $\pi=\pm \sqrt{q}$, where $q=p^n$
is an arbitrary \emph{odd} power of $p$.  Instead of obtaining a
complete formulas for $\ppav(\sqrt{q})$, we shall restrict ourselves
to its subset of superspecial members in this paper. 
Recall that an abelian variety over a field $k$ of characteristic $p$
is 
{\it supersingular} if it is $\bar k$-isogenous to a product of supersingular
elliptic curves over an algebraic closure $\bar k$ of $k$; it is {\it
  superspecial} \cite[\S1.7]{li-oort} if it is $\bar k$-isomorphic to a product of supersingular
elliptic curves over $\bar k$.  In the previous papers
\cite{xue-yang-yu:sp_as, xue-yang-yu:sp_as2, yu:sp-prime, xue-yu-zheng:spIII}, the authors
calculated the number of isomorphism classes of (unpolarized)
superspecial abelian surfaces over $\Fq$ for every prime power $q$.
By the Manin-Oort Theorem \cite[Theorem~2.9]{Yu-End-QM-2013}, each
$X/\F_q$ in the isogeny class corresponding to $\pi=\pm \sqrt{q}$ is
supersingular.  Moreover, if $q=p$, then each $X$ is superspecial by the proof
of \cite[Theorem~6.2]{waterhouse:thesis}. In general, for  an 
\emph{odd} power $q$ of $p$, let $\ppsp(\sqrt{q})$ be the following  subset
of $\ppav(\sqrt{q})$:
\[\ppsp(\sqrt{q})\coloneqq\{[X, \lambda]\in
\ppav(\sqrt{q})\mid X \text{ is  superspecial}\}.\] 
Similarly, let $\Sp(\sqrt{q})\subseteq \Isog(\sqrt{q})$ be the subset consisting of the superspecial abelian surfaces.
We show in  Corollary~\ref{cor:equiv-cat} that the base change functor
$-\!\otimes_{\F_p}\F_q$ induces canonical identifications 
\begin{equation} \label{eq:PPSP=PPAV}
   \Isog(\sqrt{p}) \isoto \Sp(\sqrt{q}) \quad \text{and} 
   \quad \ppav(\sqrt{p}) \isoto \ppsp(\sqrt{q}),
\end{equation}
which preserve the endomorphism rings. 
In particular, Theorems~\ref{thm:main} and \ref{thm:type-num} also provides 
explicit class number and type number formulas for $\ppsp(\sqrt{q})$.


Same as the case of $\Isog(\sqrt{p})$, not every member in $\Sp(\sqrt{q})$ is principally
polarizable.  Thus, the underlying abelian surfaces of
$\ppsp(\sqrt{q})$ form a proper subset of $\Sp(\sqrt{q})$ in general.
For the purpose of considering all members of $\Sp(\sqrt{q})$ equipped
with a polarization, we make use of \emph{polarization modules}, which
are fundamental invariants appearing in the theory of
Hilbert-Blumenthal varieties \cite{Rapoport-thesis,
  Deligne-Pappas-1994, yu:reduction}. We shall recall the
notion of $(P, P_+)$-polarized abelian varieties in  \S\ref{sec:PP.1};
see also \cite[\S
X.1]{van-der-Geer-HMS} and \cite[\S2.2.2]{Goren:HibertVar-book}.  Given a member
$X\in \Sp(\sqrt{q})$, we put $R\coloneqq\Q(\sqrt{p})\cap \End(X)$, 
 which is
an order in $F=\Q(\sqrt{p})$. It is shown in
Corollary~\ref{cor:prin-polar-module} that $X$ is principally
polarizable if and only if its polarization module $\ul \calP(X)$
represents the trivial element of the narrow class group $\Pic_+(R)$. As a
result, our formulas for $\abs{\ppav(\sqrt{p})}$ may be regarded as
special cases of class number formulas for the isomorphic classes of
$(P, P_+)$-polarized superspecial abelian surfaces in the simple
$\F_q$-isogeny class corresponding to $\pi=\sqrt{q}$.  Moreover, the
same method for computing $\abs{\ppav(\sqrt{p})}$ applies to these
polarized abelian surfaces as well, and the resulting formulas are
produced in Theorem~\ref{thm:Pol-mod}.  Surprisingly, these class
number formulas coincide with the formulas for the arithmetic genus of
certain Hilbert modular surfaces. One example is given as follows.

\begin{ex}
   Let $O_F$ be the ring of integers of $F=\Q(\sqrt{p})$, and $\gra\subset F$ be a nonzero fractional $O_F$-ideal. Let $\SL(O_F\oplus \gra)$ be the stabilizer of $O_F\oplus \gra$
  in $\SL_2(F)$,  and $\Gamma\coloneqq\PSL(O_F\oplus \gra)=\SL(O_F\oplus \gra)/\{\pm 1\}$. The Hilbert modular surface $Y_\Gamma$
  is defined to be the minimal
  non-singular model of the compactification of $\Gamma\bsh \calH^2$
  \cite[\S II.7]{van-der-Geer-HMS}, where $\calH\subset \mathbb{C}$ denotes the Poincare
  upper half plane as usual. Up to isomorphism, $Y_\Gamma$ depends only on the  Gauss genus of $\gra$ (namely, the coset $\Pic_+(O_F)^2[\gra]_+$ inside the narrow class group $\Pic_+(O_F)$, see \eqref{xeq:11x}) and not on the choice of $\gra$ itself \cite[\S I.4, p.~12]{van-der-Geer-HMS}.  Following \cite[\S III.1, p.~46]{van-der-Geer-HMS}, the arithmetic genus of $Y_\Gamma$ is defined to be $\sum_{i=0}^2 (-1)^i h^i(\calO_{Y_\Gamma})$, where $h^i(\calO_{Y_\Gamma}):=\dim_{\bbC} H^i(Y_\Gamma, \calO_{Y_{\Gamma}})$. This is a birational invariant. In particular, 
  we may take $Y_\Gamma$ to be any compact non-singular model of $\Gamma\bsh \calH^2$ in the discussion below.

  
  Suppose for the moment that $p\equiv 3\pmod{4}$, and $\gra$ belongs to the unique nonprincipal Gauss genus, i.e.~the narrow (strict) ideal class $[\gra]_+\in \Pic_+(O_F)$ is not of the form $[\grb^2]_+$ for any fractional $O_F$-ideal $\grb$.  Comparing Theorem~\ref{thm:main} with the formula for the arithmetic genus $\chi(Y_\Gamma)$ in \cite[Theorems~II.5.8--9]{Freitag-HMF}, we immediately find that
  \begin{equation}
      \chi(Y_\Gamma)=\abs{\ppav(\sqrt{p})}. 
  \end{equation}
  
  Varying $p$, $\gra$ and $\Gamma$, we are able to observe several similar identities 
involving the arithmetic genus of a Hilbert modular surface on one side and the class number of certain kind of $(P, P_+)$-polarized superspecial abelian surfaces on the other side.  See Remark~\ref{rem:arith-genus} for the precise identities.   
\end{ex}

 
To explain the strategy for computing $\abs{\ppav(\sqrt{p})}$, we draw an analogy between lattices in quadratic spaces and polarized abelian varieties. Let $K$ be a number field,  $(V, Q)$ be a  finite dimensional nondegenerate quadratic $K$-space, and  $G=O(V, Q)$ be the corresponding orthogonal group. Two $O_K$-lattices $L, L'$ in $V$ are \emph{isometric} if and only if there exists $g\in G(K)$ such that $L'=gL$. 
Suppose that we want to compute the number $H_{\mathrm{uni}}(V, Q)$ of isometric classes of unimodular (i.e.~self-dual) lattices in $V$.  One way to do this is to first separate the set of unimodular lattices in $V$ into genera. Following \cite[\S102]{o-meara-quad-forms}, two $O_K$-lattices $L, L'$ in $(V, Q)$ are said to \emph{belong to the same genus} if for every finite prime $\grp$ of $K$, the $\grp$-adic completions $L_\grp$ and $L'_\grp$ are isometric. The genus $\scrG(L)$  is defined to be the set of all $O_K$-lattices $L'\subset V$ belonging to the same genus as $L$. We write $\Lambda(L)\coloneqq G(L)\bsh \scrG(L)$ for the set of isometric classes of lattices within the genus $\scrG(L)$, and put $h(L)\coloneqq \abs{\Lambda(L)}$.   Let $\whK$ be the ring of finite adeles of $K$. 
The group $G(\whK)$  acts transitively on $\scrG(L)$, thus inducing a bijection 
\begin{equation}
    \Lambda(L)\simeq G(K)\bsh G(\whK)/U_G(L),
\end{equation}
where $U_G(L)\coloneqq \Stab_{G(\whK)}(L)$ denotes the stabilizer of $L$ in $G(\whK)$. 
Therefore,  $h(L)$ is equal to the class number $h(G, U_G(L)):=|G(K)\bsh G(\whK)/U_G(L)|$ of $G=O(V,Q)$ with respect to the open compact subgroup $U_G(L)$ as explained also in \cite[Proposition~8.4]{Platonov-Rapinchuk}. From the local classification of quadratic lattices \cite[\S92--\S93]{o-meara-quad-forms}, there are only finitely many genera of unimodular $O_K$-lattices in $(V, Q)$. If we pick a representative $L_i$ from each genus $\scrG_i$ and write $m$ for the total  number of genera of such lattices, then 
\begin{equation}
    H_{\mathrm{uni}}(V, Q)=\sum_{i=1}^m h(L_i)=\sum_{i=1}^m h(G, U_G(L_i)). 
\end{equation}


For our analogy, the role of a quadratic $O_K$-lattice $(L, Q)$ is played by that of a polarized abelian variety $(X, \lambda)/\F_q$.
For each prime $\ell$ (including $\ell=p$), we attach a quasi-polarized $\ell$-divisible group $(X(\ell), \lambda_\ell)$ to $(X, \lambda)$, just like we attach a local lattice $(L_\grp, Q_\grp)$ to $(L, Q)$ for each finite prime $\grp$ of $K$.  Equivalently, we can consider the quasi-polarized Tate-module or Dieudonn\'e-module incarnations of the $\ell$-divisible groups, which really are linear objects.  The notion of the quadratic space $(V, Q)$ is replaced by that of a \emph{$\Q$-isogeny class of polarized abelian varieties}.   Given two polarized abelian varieties $(X_1, \lambda_1)$ and $(X_2, \lambda_2)$ over $\F_q$, 
a $\Q$-isogeny $\varphi: (X_1, \lambda_1)\to (X_2, \lambda_2)$ is an element $\varphi \in \Hom(X_1, X_2)\otimes \Q$ such that $\lambda_1=\varphi^*\lambda_2\coloneqq \varphi^t\circ\lambda_2\circ \varphi$. Two $\Q$-isogenous  polarized abelian varieties $(X_1, \lambda_1)$ and $(X_2, \lambda_2)$ over $\F_q$ are said to \emph{belong to the same genus} if $(X_1(\ell), \lambda_{1, \ell})$ is isomorphic to  $(X_2(\ell), \lambda_{2,\ell})$ for every prime $\ell$. A more in-depth formulation of these concepts will be given in \S\ref{sec:method-cal}. Given a polarized abelian variety $(X, \lambda)/\F_q$, we can attach  a linear algebraic $\Q$-group $G^1$ to its $\Q$-isogeny class (see \eqref{eq:20}), which shall play the role of the othorgonal group $O(V, Q)$.  
Let $\scrG(X, \lambda)$ be the genus of $(X, \lambda)$, and $\Lambda(X, \lambda)$ be the isomorphism classes of polarized abelian varieties within $\scrG(X, \lambda)$.  We then have a transitive  action of $G^1(\whQ)$ on $\scrG(X, \lambda)$, which gives rise to a bijection 
\begin{equation}\label{eq:1.16}
    \Lambda(X, \lambda)\simeq G^1(\Q)\bsh G^1(\whQ)/U_{G^1}(X, \lambda), 
\end{equation}
where $U_{G^1}(X, \lambda)$ denotes the stabilizer of $(X, \lambda)$ in $G^1(\whQ)$ as usual.  The computation of the number of isomorphism classes of polarized abelian varieties within a genus is once again reduced to the computation of the class number of a linear algebraic group with respect to an open compact subgroup.

For the computation of $\abs{\ppav(\sqrt{p})}$, we put $F=\Q(\sqrt{p})$ as before and write $D$ for the totally definite quaternion $F$-algebra $D_{\infty_1, \infty_2}$ in \eqref{eq:23}. 
It is straightforward to show that all polarized abelian surfaces $(X, \lambda)/\F_p$ with the Frobenius endomorphism $\pi_X$ satisfying $\pi_X^2=p$ form a single $\Q$-isogeny class; see Lemma~\ref{P.2}. From \S\ref{sect:neron-severi}, the linear algebraic group $G^1$ attached to this $\Q$-isogeny class is the reduced norm one subgroup of the multiplicative groupu $\ul D^\times$. 
 After some symplectic local lattice classification that is worked out in \S\ref{sec:LC}, we quickly find in Proposition~\ref{P.4} that  $\ppav(\sqrt{p})$ forms a single genus when $p\not\equiv 1\pmod{4}$,   and it partitions into two different genera $\Lambda_1^\pp$ and $\Lambda_{16}^\pp$ when $p\equiv 1\pmod{4}$. For uniformity we also put  $\Lambda_1^\pp\coloneqq \ppav(\sqrt{p})$ when $p\not\equiv 1\pmod{4}$. 
 Let us pick a representative $[Y_1, \lambda_1]\in \Lambda_1^\pp$ for every prime $p$ and denote $\grO_1\coloneqq \End(Y_1)$. Similarly, if $p\equiv 1\pmod{4}$, we pick a representative $[Y_{16}, \lambda_{16}]\in \Lambda_{16}^\pp$ and denote $\grO_{16}\coloneqq \End(Y_{16})$. Here the subscripts $r=1, 16$ are chosen so that $[\bbO: \grO_r]=r$ for every maximal $O_F$-order $\bbO$ containing $\grO_r$; see \eqref{eq:end0index}. Let $\whD$ be the ring of finite adeles of $D$.  For any subset $\scrS$ of $\whD$, we write $\scrS^1$ for its subset of elements of reduced norm one, that is, $\scrS^1\coloneqq \{x\in \scrS\mid \Nr(x)=1\}$.  In particular, $G^1(\Q)=D^1$ and $G^1(\whQ)=\whD^1$. It turns out that the stabilizer group $U_{G^1}(Y_r, \lambda_r)$ coincides with $\wh\grO_r^1$, where $\wh\grO_r$ denotes the profinite completion of $\grO_r$.
 From \eqref{eq:1.16}, there is  a bijection 
 \begin{equation}\label{eq:1.17}
     \Lambda_r^\pp\simeq D^1\bsh \whD^1/\wh\grO_r^1,\qquad \forall\:  r\in \{1, 16\}. 
 \end{equation} 
Denote the cardinality of the double coset space in \eqref{eq:1.17} by $h^1(\grO_r)$. We then have   \begin{equation}\label{eq:1.18}
    \abs{\ppav(\sqrt{p})}=\begin{cases}
        h^1(\grO_1)  & \text{if } p\not\equiv 1\pmod{4};\\
        h^1(\grO_1)+h^1(\grO_{16}) &\text{if } p\equiv 1\pmod{4}.
    \end{cases}
\end{equation}


Several challenges arise when we try to compute the right hand side of  \eqref{eq:1.18} explicitly.  First, we need to produce an order $\grO_r\subset D$ that occurs as $\End(Y_r)$ for some $[Y_r, \lambda_r]\in \Lambda_r^\pp$,  or at least to describe such an order  explicitly enough so that the computation of $h^1(\grO_r)$ is viable. Recall that two $\Z[\sqrt{p}]$-orders $\calO, \calO'$ in $D$ are said to \emph{belong to the same genus} if there exists $x\in \whD^\times$ such that $\wcO= x\wcO' x^{-1}$, or equivalently, if $\calO_\ell\simeq \calO'_\ell$ for every prime $\ell\in\bbN$. From Tate's theorem, we can determine explicitly the genera of $\grO_1$ and $\grO_{16}$, that is,  the local descriptions of $\grO_r\otimes \Z_\ell$ for all $r\in \{1, 16\}$ and all prime $\ell$; see \S\ref{sect:end-ring}.  For example, this precisely the reason why we known  that $\grO_1$ is always a maximal $O_F$-order in $D$.  Conversely from \cite[Theorem~3.13]{waterhouse:thesis}, given a maximal $O_F$-order $\bbO\subset D$,  there exists $[X]\in \Isog(\sqrt{p})$ such that $\End(X)\simeq \bbO$.  However, there is no guarantee in general that we can find such an $X$  that is principally polarizable. 
Thus, we need to find a finer classification of orders within a fixed genus of $D$. 
To solve this problem we make use  of the notion of \emph{spinor genus of orders}. From   Definition~\ref{defn:spinor-genus-class}, two $\Z[\sqrt{p}]$-orders $\calO, \calO'$ in $D$ are said to \emph{belong to the same spinor genus} if there exists $x\in D^\times\whD^1$ such that $\wcO= x\wcO' x^{-1}$. Clearly this defines an equivalence
relation finer than ``being in the same genus''.  The key Lemma~\ref{lem:spinor-genus} shows that if $\calO$ and $\calO'$ are two orders in $D$ belonging to the same spinor genus and $\calO$ occurs as the endomorphism ring of some  principally polarizable member of $\Isog(\sqrt{p})$, then so does $\calO'$.  We then determine the spinor genera of  endomorphism rings of  principally polarizable members of $\Isog(\sqrt{p})$ in Proposition~\ref{prop:spinor-gen-classif} and Lemma~\ref{lem:end-ring-p3mod4}, which is equivalent to finding out the desired $\grO_r$.

Now comes the real challenge of actually computing $h^1(\grO_r)$. 
Since $D$ is totally definite, $G^1(\bbR)$ is compact, and hence strong approximation fails for $G^1$. For a long time, the only systematic  way to compute $h^1(\grO_r)$ is via the Selberg trace formula, which requires heavy analysis of orbital integrals and can be unwieldy to apply.  The situation is quite different if the group $G^1=\ul D^1$ is replaced by the multiplicative group $\ul D^\times$.  More generally, let $\calF$ be an arbitrary totally real number field,  $\calD$ be a totally definite quaternion $\calF$-algebra, and $\calO$ be a $O_\calF$-order in $\calD$. By definition, the class number of the multiplicative group $\ul \calD^\times$ with respect to $\wcO^\times$ is given by  
\[h(\calO)\coloneqq h(\ul \calD^\times, \wcO^\times)=\abs{\calD^\times\bsh \wh\calD^\times/\wcO^\times}.\] 
The class number $h(\calO)$ can be calculated by the Eichler class number formula
\cite[Corollaire V.2.5,
p.~144]{vigneras} \cite[Theorem~30.8.6]{voight-quat-book} thanks to the work of Eichler \cite{eichler:crelle55}, Vigneras \cite{vigneras}, K\"orner \cite{korner:1987} and others. A key result underlying this formula is the trace formula \cite[Theorem~III.5.1]{vigneras} for optimal embeddings from a quadratic $O_\calF$-order $B$ into orders in the genus of $\calO$. 
Eichler's formula has also been generalized by the current authors together with Tse-Chung Yang to an arbitrary $\Z$-order of full rank in $\calD$, and this generalization plays a crucial role for the computation of $\abs{\Isog(\sqrt{p})}$ in \cite{xue-yang-yu:ECNF}. 
Therefore, it is highly desirable to have a class number formula for $h^1(\calO)=h(\ul\calD^1, \wcO^1)$ of a \emph{similar shape} to the Eichler class number formula. The advantage of this approach avoids complicated computations of orbital integrals in the Selberg trace formula.   
As explained in the previous paragraph, the class number $h^1(\calO)$ depends only on the spinor genus of $\calO$.
Thus to obtain an 
Eichler kind class number formula for $h^1(\calO)$, 
it is essential to figure out how the spinor genus information (which boasts a mixture of local and global flavor) would manifest itself in the desired formula. It is at this critical juncture where the spinor selectivity theory comes in.


The selectivity theory is first formulated by Chinburg and Friedman
\cite{Chinburg-Friedman-1999} as an refined integral version of the
Hasse-Brauer-Noether-Albert Theorem \cite[Theorem~III.3.8]{vigneras}, 
and has since been further developed by many people. See
\cite[\S31.7.7]{voight-quat-book} for a historical account. We give a brief summary of this theory in \S\ref{subsec:oss}.  Except for \cite{Arenas-Carmona-Spinor-CField-2003} and \cite[Remark, p.~99]{M.Arenas-et.al-opt-embed-trees-JNT2018}, most of the literature on selectivity theory focus on quaternion algebras that satisfies the Eichler condition (i.e.~the case where strong approximation holds).  For the current purpose, the current authors developed the selectivity theory for totally definite quaternion algebras  in more depth  and obtained in \cite{Xue-Yu-Selec-2022} an refinement for the trace formula for optimal embeddings, called the {\it spinor trace formula}. In the sequel work \cite{xue-yu:spinor-class-no} based on the Selberg trace formula and spinor trace formula we deduced a formula for $h^1(\calO)$ of Eichler type for all Eichler orders in totally definite quaternion algebras; see \cite[Corollary 4.2] {xue-yu:spinor-class-no} and \eqref{xeq:214}. 
The proof of this formula has  greatly simplified by Voight
\cite[Appendix A]{xue-yu:spinor-class-no}. 
Using this formula we compute the explicit formula for $h^1(\grO_1)$.



Note that the general formula  for $h^1(\calO)$ developed in \cite{xue-yu:spinor-class-no} can only be applied to the maximal order $\grO_1$, which alone is not enough to compute $\abs{\ppav(\sqrt{p})}$ when $p\equiv 1\pmod{4}$. 
Indeed, in the latter case, $\grO_{16}\cap F=\Z[\sqrt{p}]\neq O_F$, so $\grO_{16}$ is not an $O_F$-order. 
This calls for some ad-hoc methods. Luckily, we are able to show in \S\ref{sec:class-number} that $h^1(\grO_{16})=h(\grO_{16})/h(\Z[\sqrt{p}])$, where $h(\grO_{16})$ has been previously computed in \cite{xue-yang-yu:ECNF}. The establishment  of this equality  relies on the fact that $h(\Q(\sqrt{p}))$ is odd by \cite[Corollary~18.4]{Conner-Hurrelbrink}.  This completes a summary of the computation of $\abs{\ppav(\sqrt{p})}$. The type number formulas are obtained along the way using class-type number relations; see Lemma~\ref{lem:bijection-prin-type} for example.  A survey of current results of this paper is provided in \cite{xue-yu-RIMS-2022}.

This paper is organized as follows. In \S\ref{sec:method-cal} we
explain our strategy for computing $\abs{\ppav(\pi)}$, which is a
simpler version of the Langlands-Kottwitz method.  This strategy is divided into three steps as described in \S\ref{sect:3-steps}.  The first two steps is
carried out in detail for $\pi=\pm \sqrt{p}$ in
\S\ref{sec:computing-ppav}, and the third step in
\S\ref{sec:S4}, 
except that some class number calculations
are postponed to \S\ref{sec:max-ord}--\ref{sec:class-number}, and certain symplectic
lattice classification is postponed to \S\ref{sec:LC}. In
\S\ref{sec:PP}, we recall the notion of polarization modules, and
extend the results of \S\ref{sec:computing-ppav}--\ref{sec:S4} to produce both class
number and type number formulas for $(P, P_+)$-polarized abelian
surfaces in the simple $\F_p$-isogeny class corresponding to
$\pi=\pm\sqrt{p}$. 




{\bf Notation.} As usual, $\Z_\ell$ denotes the $\ell$-adic completion of $\Z$ at the prime $\ell$. Let  $\wh \Z=\prod_\ell \Z_\ell$ be the pro-finite
completion of $\Z$,  and $\whQ$ the finite adele ring of $\Q$.  If $M$ is a
finitely generated
$\Z$-module or a finite dimensional  $\Q$-vector space, we put
$M_\ell\coloneqq M\otimes_{\Z} \Z_\ell$ and $\wh M\coloneqq M\otimes_{\Z} \wh \Z$. The unique quaternion $\Q$-algebra ramified exactly at
$p$ and $\infty$ is denoted  by $D_{p, \infty}$. Similarly,  the unique
quaternion $\Q(\sqrt{p})$-algebra ramified exactly at the two infinite
places is denoted by $D_{\infty_1, \infty_2}$. For simplicity, we often put
$D\coloneqq D_{\infty_1, \infty_2}$. 
If $B$ is a finite-dimensional $\Q$-algebra, we denote
by $\ul B^\times$ the algebraic $\Q$-group that represents the functor
$R\mapsto (B\otimes_{\Q} R)^\times$ for any commutative $\Q$-algebra
$R$. The group $\ul B^\times$ is called the multiplicative algebraic $\Q$-group
of $B$. If $K=\prod K_i$ is a finite product of number fields or non-archimedean local
fields, then we denote by $O_K$ the maximal order of
$K$.  For any abelian variety $X$ defined over a finite field $\Fq$, we denote by
$\pi_X$ the Frobenius endomorphism of $X$ over $\Fq$. 


\section{Method of calculation}
\label{sec:method-cal}


Given a Weil $q$-number $\pi$, the study of $\Isog(\pi)$ in terms of
ideal classes of certain orders goes back to Waterhouse
\cite{waterhouse:thesis}. The general method for counting abelian
varieties equipped with a PEL-type structure in an isogeny class is
developed by Langlands \cite{Langlands-1973} and Kottwitz
\cite{kottwitz:michigan1990, Kottwitz-JAMS1992} for studying the Hasse-Weil zeta
functions of Shimura varieties.  Employing this method, Lipnowski and
Tsimerman \cite{Lipnowski-Tsimerman} give asymptotic upper bounds for the number
of isomorphism classes of $g$-dimensional PPAVs over $\F_p$ for a
fixed $p$ as $g\to \infty$. 
Nontrivial lower bounds is obtained by Jungin Lee \cite{Lee-IMRN2022}
in a follow-up work. 
Recently,
results for counting $\ppav(\pi)$ with $\pi$ being a Weil $p$-number
distinct from $\pm \sqrt{p}$ are obtained by Achter et~al.~\cite{Achter-Rud-ANT2023}  in terms of a product of local densities and the Tamagawa number of a CM torus. Bergstr\"on, Karemaker and Marseglia \cite{BKM-IMRN-2023} obtained an algorithmic breakthrough 
on counting principally polarized abelian varieties with commutative endomorphism algebras over finite fields of small cardinalities. A similar study was previously carried out by Marseglia in \cite{Marseglia:MathC2021}.  In a slight different direction, Oswal and Shankar \cite{OS-JLMS-2020} obtained explicit descriptions of almost ordinary abelian varieties over finite fields of odd characteristic that can be applied to counting problems.



For the
purpose of this paper, we follow a variant of this  method in
\cite{xue-yu:counting-av}, which is previously developed by the second
named author in \cite{yu:smf}. This  method treats both the
unpolarized case and the principally polaried case uniformly and
expresses the cardinalities as sums of class numbers
of certain linear
algebraic $\Q$-groups.  As the method is built
upon Tate's
theorem (due to Tate, Zarhin, Faltings and de Jong), its key part 
is appliable to  any \emph{finitely generated} ground field $k$ (that
is, finitely generated over its prime subfield).

Given an  abelian variety  $X$ over $k$ and a prime number $\ell$ (not
necessarily distinct from the characteristic of $k$), we write 
$X(\ell)$ for the $\ell$-divisible group
$\varinjlim X[\ell^n]$ attached to $X$. A
\emph{$\Q$-isogeny}  $\varphi: X_1\to X_2$ between two abelian
varieties over $k$ is an element $\varphi\in \Hom(X_1,
X_2)\otimes \Q$ such that $N\varphi$ is an isogeny for some $N\in \bbN$. 
Similarly, one defines the notion of $\Q_\ell$-isogenies between
$\ell$-divisible groups. Clearly, a $\Q$-isogeny $\varphi$ induces a $\Q_\ell$-isogeny $\varphi_\ell:
X_1(\ell)\to X_2(\ell)$ for each prime $\ell$, and $\varphi_\ell$ is an 
isomorphism for almost all $\ell$.  

Fix an abelian variety $X_0$ over $k$. Two $\Q$-isogenies $\varphi_1:
X_1\to X_0$ and $\varphi_2: X_2\to X_0$ are said to be \emph{equivalent} if
there exists an isomorphism $\theta: X_1\to X_2$ such that
$\varphi_2\circ \theta=\varphi_1$. 
Let $\Qisog(X_0)$ be the set of equivalence classes of $\Q$-isogenies
$(X, \varphi)$ to $X_0$. By an abuse of notation, we still write $(X,
\varphi)$ for its equivalence class.  The endomorphism ring
$\End(X)$ is realized as a canonical suborder of 
$\End^0(X_0)$ 
via the isomorphism
\begin{equation}
  \label{eq:15}
  \End^0(X)\xrightarrow{\simeq} \End^0(X_0),\qquad a\mapsto \varphi a \varphi^{-1}.
\end{equation}
The set $\Qisog(X_0)$ contains a
distinguished element $(X_0, \id_0)$, where $\id_0$ is the 
 identity map of $X_0$. For any $\Q$-isogeny  $\varphi_1: X_1\to X_0$, there is a bijection 
 \begin{equation}
   \label{eq:6-s2}
\Qisog(X_0)\to \Qisog(X_1), \qquad (X, \varphi)\mapsto (X,
\varphi_1^{-1}\varphi).    
 \end{equation}
Therefore, we may change the base abelian variety $X_0$ to suit our
purpose whenever necessary. 
 Similarly, one defines $\Qisog(X_0(\ell))$
 for every prime $\ell$. 

 Let $G$ be the algebraic $\Q$-group
 representing the functor 
\begin{equation}
  \label{eq:17}
R\mapsto G(R)\coloneqq (\End^0(X_0)\otimes_\Q R)^\times     
\end{equation}
for every commutative 
  $\Q$-algebra $R$. Up to isomorphism, $G$ depends only on the isogeny
  class of $X_0$. We have 
  $G(\Q_\ell)=(\End(X_0(\ell))\otimes_{\Z_\ell} \Q_\ell)^\times$ by Tate's
  theorem. Let
  $\whQ\coloneqq \whZ\otimes_{\Z} \Q$ be the ring of finite adeles of $\Q$.  From
  \cite[Lemma~5.2]{xue-yu:counting-av}, there is an
  action of $G(\whQ)$ on $\Qisog(X_0)$ as follows: for any $(X, \varphi)\in \Qisog(X_0)$ and any $\alpha=(\alpha_\ell)\in
    G(\whQ)$, the member $(X', \varphi')\coloneqq \alpha(X, \varphi)$
    is uniquely characterized by the following equality in $\Qisog(X_0(\ell))$ for every prime $\ell$:
    \begin{equation}
      \label{eq:16}
       (X'(\ell), \varphi_\ell')=(X(\ell), \alpha_\ell\varphi_\ell). 
     \end{equation}
     In particular, if $\alpha\in G(\Q)$, then $\alpha (X,
     \varphi)=(X, \alpha \varphi)$. 
From Tate's theorem and the identification (\ref{eq:15}), we obtain 
          \begin{equation}
       \label{eq:29}
       \End(X')\otimes \whZ=\alpha(\End(X)\otimes \whZ) \alpha^{-1}. 
     \end{equation}
    If 
    $\Qisog(X_0)$ is equipped  with the discrete topology, then the action of
    $G(\whQ)$ on $\Qisog(X_0)$ is continuous and proper. Indeed, the
    stabilizer of any $(X, \varphi)\in \Qisog(X_0)$ coincides with the open compact
    subgroup  $(\End(X)\otimes \whZ)^\times\subset G(\whQ)$.


  \begin{defn}\label{defn:adelic-action-on-Qisog}
    Let $H\subseteq G$ be an algebraic subgroup of $G$ over $\Q$. Two
    members $(X_i, \varphi_i)\in \Qisog(X_0)$ for $i=1,2$ are said to
    be in the \emph{same $H$-genus} if there exists $\alpha\in
    H(\whQ)$ such that $(X_2, \varphi_2)=\alpha(X_1,
    \varphi_1)$. They are said to be
    \emph{$H$-isomorphic} if there
    exists $\alpha\in H(\Q)$ such that $(X_2, \varphi_2)=(X_1,
    \alpha\varphi_1)$. 
  \end{defn}

Thus $\Qisog(X_0)$ is partitioned into $H$-genera, and each
  $H$-genus is further subdivided into $H$-isomorphism classes. In
  particular, it makes sense to talk about two $H$-isomorphism classes
  $[X_i, \varphi_i]_{i=1,2}$ belonging to the same $H$-genus.

  \begin{prop}\label{prop:H-genus-isom}
    Let $\scrG_H(X, \varphi)\subseteq \Qisog(X_0)$ be the $H$-genus
    containing $(X, \varphi)$, and $\Lambda_H(X, \varphi)$ be the set
     of 
    $H$-isomorphism classes within $\scrG_H(X, \varphi)$. Let
    $U_H(X, \varphi)\coloneqq \Stab_{H(\whQ)}(X, \varphi)$ be  the stabilizer
of $(X, \varphi)$ in $H(\whQ)$. Then there is a
    bijection 
\[\Lambda_H(X, \varphi)\simeq H(\Q)\backslash H(\whQ)/U_H(X, \varphi), \]
sending the $H$-isomorphic class $[X, \varphi]$ to the identity
class on the right hand side.
  \end{prop}
The proposition follows directly from definition. From \cite[Theorem~8.1]{Platonov-Rapinchuk}, $\Lambda_H(X, \varphi)$ is a
finite set. Proposition~\ref{prop:H-genus-isom} turns out to be quite
versatile. By varying $H$, it can be used to count abelian varieties
with various additional structures. We give two examples below, one
for counting unpolarized abelian varieties, and another for counting
principally polarized ones. See \S\ref{sec:PP} for counting $(P,
P_+)$-polarized abelian surfaces with real
multiplication.


\begin{sect}
  First, let us look at the case $H=G$. Two members
  $(X_i,\varphi_i)_{i=1,2}\in \Qisog(X_0)$ are in the same
  $G$-genus if and only if $X_1(\ell)$ is isomorphic to $X_2(\ell)$
  for every prime $\ell$. This matches with the classical notion of
  \emph{genus} for unpolarized abelian varieties in an isogeny class,
  cf.~\cite[Definition~5.1]{xue-yu:counting-av}.  Similarly,
  $(X_1,\varphi_1)$ and $(X_2,\varphi_2)$ are $G$-isomorphic if and
  only if $X_1$ and $X_2$ are $k$-isomorphic abelian varieties. Therefore,
  Proposition~\ref{prop:H-genus-isom} recovers
  \cite[Proposition~5.4]{xue-yu:counting-av} in the case $H=G$.
\end{sect}

\begin{sect}

Next,  we look at polarized abelian varieties. Let $X^t$ be the dual abelian
variety of $X$.  A $\Q$-isogeny $\lambda: X\to
X^t$ is said to be a \emph{$\Q$-polarization} if $N\lambda$ is a
polarization for some $N\in \bbN$. For each prime $\ell$, the
$\Q$-polarization $\lambda$ induces
a $\Q_\ell$-quasipolarization of $X(\ell)$ (see \cite[\S1]{Oort-ann-2000}
and \cite[\S5.9]{li-oort}). 
 An isomorphism (resp.~$\Q$-isogeny) from a 
$\Q$-polarized abelian variety $(X_1, \lambda_1)$ to another $(X_2,
\lambda_2)$ is an isomorphism (resp.~$\Q$-isogeny) $\varphi: X_1\to X_2$ such that
\begin{equation}
  \label{eq:27}
\lambda_1=\varphi^*\lambda_2\coloneqq \varphi^t\circ\lambda_2\circ\varphi.   
\end{equation}
Fix a $\Q$-polarized abelian variety $(X_0, \lambda_0)$. Two $\Q$-isogenies $\varphi_i: (X_i, \lambda_i)\to (X_0, \lambda_0)$
for $i=1,2$ are said to be \emph{equivalent} if there exists an isomorphism
$\theta: (X_1, \lambda_1)\to (X_2, \lambda_2)$ such that
$\varphi_1=\varphi_2\circ \theta$. 
We define
$\Qisog(X_0, \lambda_0)$ to be the set of equivalence classes of all 
$\Q$-isogenies $(X, \lambda, \varphi)$ to $(X_0, \lambda_0)$. The
forgetful map $(X, \lambda, \varphi)\mapsto (X, \varphi)$ induces a
bijection:
\begin{equation}
  \label{eq:7}
\calF(\lambda_0):  \Qisog(X_0, \lambda_0)\to \Qisog(X_0), 
\end{equation}
whose inverse is given by $(X, \varphi)\mapsto (X, \varphi^*\lambda_0,
\varphi)$.

Let $G^1\subseteq G$ be the algebraic subgroup over $\Q$
that represents the functor 
\begin{equation}\label{eq:20}
R\mapsto G^1(R)\coloneqq \{g\in (\End^0(X_0)\otimes_\Q R)^\times\mid g^t\circ
\lambda_0\circ g=\lambda_0\}  
\end{equation}
for every commutative $\Q$-algebra $R$. In light of the bijection (\ref{eq:7}), two members
$(X_i, \varphi_i)_{i=1,2}\in \Qisog(X_0)$ are in the same $G^1$-genus if and
only if $(X_1(\ell), \lambda_{1,\ell})$ is isomorphic to
$(X_2(\ell), \lambda_{2,\ell})$ for every prime $\ell$. Once again,
this recovers the classical notion of ``being in the same genus'' for
$\Q$-polarized abelian varieties in an isogeny class. Similarly, $(X_1,
\varphi_1)$ and $(X_2,
\varphi_2)$ are $G^1$-isomorphic if and only if $(X_1, \lambda_1)$ and
$(X_2, \lambda_2)$ are isomorphic $\Q$-polarized abelian varieties over $k$. Therefore, when $H=G^1$,
Proposition~\ref{prop:H-genus-isom} is a special case of
\cite[Theorem~5.10]{xue-yu:counting-av}. 
\end{sect}

In practice, we are more interested in abelian varieties with
\emph{integral} polarizations  than
  $\Q$-polarizations. 

\begin{lem}[{\cite[Remark~5.7]{xue-yu:counting-av}}]\label{lem:genus-polarization}
  Let $\scrG(X, \lambda, \varphi)\subseteq \Qisog(X_0, \lambda_0)$ be the 
  genus containing $(X, \lambda, \varphi)$. If $\lambda$
  is an integral polarization on $X$, then $\lambda'$ is  integral  for every member $(X', \lambda', \varphi')\in
  \scrG(X, \lambda, \varphi)$. If
  moreover $\lambda$ is principal, then so is
  $\lambda'$. 
\end{lem}

\begin{sect}\label{sect:genera-G1-Tate}
  For calculation purpose, it is more convenient
  to describe genera of $\Q$-polarized abelian varieties in terms of Tate
  modules and \dieu modules rather than $\ell$-divisible
  groups.  Assume that $k$ is the finite field $\F_q$, and
   $\pi$ is an arbitrary Weil $q$-number. Let
   \begin{equation}
 F=\Q(\pi),\qquad O_F=\text{the ring of
  integers of $F$}, \qquad     A=\Z[\pi]\subseteq O_F.
   \end{equation}
  If $F$ is a CM-field, then we write $a\mapsto \bar{a}$
  for the complex conjugation map on $F$; otherwise $\pi=\pm \sqrt{q}$
  so that $F$ is either $\Q$
  or $\Q(\sqrt{p})$, and we put $\bar{a}=a$ for every $a\in F$. Let
  $X_0/\F_q$ be an abelian variety in the
  simple $\F_q$-isogeny class  corresponding to $\pi$, and $\lambda_0:X_0\to
  X_0^t$  be a $\Q$-polarization.   At each prime
  $\ell\neq p$, the Tate space $V_\ell\coloneqq T_\ell(X_0)\otimes\Q_\ell$
  with its $\Gal(\Fqbar/\Fq)$-module structure is simply a free
  $F_\ell$-module of rank $2\dim(X_0)/[F:\Q]$, and $\lambda_0$ induces
  a non-dengerate alternating $\Q_\ell$-bilinear Weil pairing
  \begin{equation}
  \label{eq:22}
\psi_\ell: V_\ell\times V_\ell\to \Q_\ell\quad \text{such
    that}  \quad   \psi_\ell(ax, y)=\psi_\ell(x,
  \bar{a}y)
\end{equation}
for all $a\in F_\ell$ and $x, y\in V_\ell$.  A $\Q$-isogeny $\varphi:
(X, \lambda)\to (X_0, \lambda_0)$ identifies $T_\ell(X)$ with the
$A_\ell$-lattice $\varphi_\ell(T_\ell(X))$ in  $(V_\ell,
\psi_\ell)$. If $F$ is totally real, then there exists a unique 
non-degenerate alternating $F_\ell$-bilinear form $\psi_{\ell, F}:
V_\ell\times V_\ell\to F_\ell$ such that $\psi_\ell=\Tr_{F/\Q}\circ
\psi_{\ell, F}$. If $F$ is a CM-field, then after fixing an element
$\gamma\in F^\times$ with $\bar{\gamma}=-\gamma$, there exists a
unique  non-degenerate hermitian form $\psi_{\ell, F}: V_\ell\times
V_\ell\to F_\ell$ such that  $\psi_\ell(x,
y)=\Tr_{F/\Q}(\gamma\psi_{\ell, F}(x,y))$ for all $x, y\in V_\ell$. 

Similarly, let $M(X_0)$ be the covariant Dieudonn\'e module of $X_0$, which is a free module of rank $2\dim(X_0)$ over the ring of
Witt vectors $W(\F_q)$. For simplicity,  put 
\begin{equation}\label{eq:witt}
 \Z_q\coloneqq W(\F_q), \qquad 
  \Q_q\coloneqq \Z_q\otimes_{\Z_p}\Q_p,   
\end{equation}
 and let $\sigma\in \Gal(\Q_q/\Q_p)$
  be the Frobenius automorphism as usual.  The Frobenius operator $\sfF$
  (resp.~Verschiebung operator $\sfV$) acts $(\Z_q,\sigma)$-linearly
  (resp.~$(\Z_q, \sigma^{-1})$-linearly) on
  $M(X_0)$, and $\pi$ acts $\Z_q$-linearly on $M(X_0)$ via $\sfF^n$. 
  From \cite[\S
5.9]{li-oort}, the $\Q$-polarization $\lambda_0$  induces 
a non-degenerate alternating $\Q_q$-bilinear form $\psi_p$  on the  F-isocrystal
$V_p\coloneqq M(X_0)\otimes_{\Z_p} \Q_p$ satisfying 
\begin{equation}
  \label{eq:62}
  \psi_p(\sfF x, y)=\psi_p(x, \sfV
  y)^\sigma,\qquad 
  \forall\ x, y\in V_p.
\end{equation}
 A $\Q$-isogeny
$\varphi: (X, \lambda)\to (X_0, \lambda_0)$  again identifies $M(X)$
with the \dieu lattice $\varphi_p(M(X))$ in $(V_p,
\psi_p)$.




From the fuctorial equivalence  between
$\ell$-divisible groups and \dieu modules (when $\ell=p$) or Tate
modules (when $\ell\neq p$), we see that two members
$(X_i,\lambda_i,\varphi_i)_{i=1,2}\in \Qisog(X_0,\lambda_0)$ belong to
the same genus if and only if both of the following conditions are satisfied:
\begin{enumerate}
\item at each prime $\ell\neq p$, the two $A_\ell$-lattices
  $T_\ell(X_i)$ are isometric in $(V_\ell, \psi_\ell)$; 
\item at the prime $p$, the two \dieu modules $M(X_i)$  are isometric
  in $(V_p, \psi_p)$.  
\end{enumerate}

\end{sect}

\begin{sect}
Let  $\Qisog^\pp(X_0, \lambda_0)$ be
  the subset of $\Qisog(X_0,
  \lambda_0)$ consisting of the principally polarized members. 
A member $(X, \lambda, \varphi)\in \Qisog(X_0, \lambda_0)$ belongs to
$\Qisog^\pp(X_0, \lambda_0)$ if and only if both of 
the following conditions are satisfied:
\begin{enumerate}
\item the \dieu module $M(X)$ is self-dual in $(V_p, \psi_p)$;
\item the Tate module   $T_\ell(X)$ is self-dual in $(V_\ell,
  \psi_\ell)$ for each
 $\ell\neq p$. 
\end{enumerate}
Let $\grM_p $ be the set of isometric classes of self-dual \dieu modules in
$(V_p, \psi_p)$, and for each prime $\ell\neq p$, let $\grM_\ell$ be the
set of isomemtric classes of self-dual $A_\ell$-lattices in $(V_\ell,
\psi_\ell)$. Put $\grM\coloneqq \grM_p\times \prod_{\ell\neq
    p}\grM_\ell$. 
There is a canonical  map 
\begin{equation}
  \label{eq:71}
\Phi:  \Qisog^\pp(X_0, \lambda_0)\to \grM, \qquad  (X, \lambda, \varphi)\mapsto ([M(X)],
  ([T_\ell(X)])_{\ell\neq p}),  
\end{equation}
whose fibers are precisely the genera of principally polarized members
of $\Qisog(X_0, \lambda_0)$. Let $\grd_A$ be the discriminant of $A$ over
$\Z$, and $S_\pi$ be the following finite set of primes
\begin{equation}
  \label{eq:72}
 S_\pi\coloneqq \{\ell\mid \ell\text{ divides } \grd_A\}\cup \{p\}.  
\end{equation}
For a prime $\ell\neq p$, we have $\ell\not\in S_\pi$ if and only if
$\ell$ is unramified in $F$ and $A_\ell$ is the maximal order in
$F_\ell$. We claim that $\abs{\grM_\ell}\leq 1$ for every prime
$\ell\not\in S_\pi$. If $\pi=\pm \sqrt{q}$, then $\abs{\grM_\ell}=1$
according to Lemma~\ref{LC.2}; otherwise  $F=\Q(\pi)$ is a CM-field, and it follows
from \cite[\S7]{Jacobowitz-HermForm} that $\abs{\grM_\ell}\leq 1$.  Therefore, if 
$\grM\neq \emptyset$, then it is bijective to the finite product
$\prod_{\ell\in S_\pi}\grM_\ell$.

\end{sect}
\begin{lem}
The map $\Phi$ is surjective if $\grM\neq \emptyset$. 
\end{lem}
\begin{proof}
  For all but finitely many $\ell\neq p$, the $A_\ell$-lattice
  $T_\ell(X_0)$ is self-dual in $(V_\ell, \psi_\ell)$.  We collect the
  exceptional $\ell$'s into a finite set $S_0$, which includes $p$ by
  default. Suppose that $\grM\neq \emptyset$.  For any member
  $\grm\in \grM$, let $(M, (T_\ell)_{\ell\neq p})$ be a representative of
  $\grm$, and put $S=(S_0\cup S_\pi)\smallsetminus \{p\}$. The finite
  product of local lattices
  $(M, (T_\ell)_{\ell\in S})\subset V_p\times \prod_{\ell\in S}V_\ell$
  determines a unique member
  $(X, \lambda, \varphi)\in \Qisog^\pp(X_0, \lambda_0)$, which is
  mapped to $\grm\in \grM$ by $\Phi$.
\end{proof}

\begin{cor}
The genera within $\Qisog^\pp(X_0,
\lambda_0)$ is bijective to the set $\grM$. 
\end{cor}

\begin{sect}\label{sect:3-steps}
Thus in principle,   to compute $\abs{\ppav(\pi)}$ for a Weil 
$q$-number $\pi$, we take the following three steps: 


\begin{enumerate}[label=(Step \arabic*)]
\item Determine if $\ppav(\pi)=\emptyset$ or not. If it is nonempty, then separate $\ppav(\pi)$ into $\Q$-isogeny classes. Each such
  isogeny class determines an algebraic $\Q$-group $G^1$ and a collection of  symplectic spaces 
   $(V_\ell, \psi_\ell)$ as in \eqref{eq:22} and \eqref{eq:62} indexed\footnote{For every prime $\ell\not\in S_\pi$, the space 
 $(V_\ell, \psi_\ell)$ is uniquely determined
   up to isometry by $\pi$ itself since it  admits a self-dual
  $A_\ell$-lattice.} by  $\ell\in S_\pi$.


\item For each $\Q$-isogeny class in $\ppav(\pi)$, separate it further
  into genera. This amounts to classifying the isometric classes of
  the following objects
  \begin{enumerate}
  \item     self-dual \dieu modules in $(V_p, \psi_p)$;
  \item  self-dual $A_\ell$-modules in $(V_\ell, \psi_\ell)$ for every $\ell\in
    (S_\pi\smallsetminus\{p\})$. 
     \end{enumerate}


\item The cardinality of the genus in $\ppav(\pi)$ represented by a member
  $[X, \lambda]$ is equal to the class number
  \begin{equation}
    \label{eq:8}
    \abs{G^1(\Q)\backslash G^1(\whQ)/U_{G^1}(X)}. 
  \end{equation}
Varying $[X, \lambda]$  genus by genus, we obtain $\abs{\ppav(\pi)}$ by summing up all such class
numbers.
\end{enumerate}

However, in this fullest generality, it
is nontrivial to even determine if $\ppav(\pi)=\emptyset$ or not
(cf.~\cite{Howe-Maisner-et-al2008, Howe:1996AG}), let along
separating it into $\Q$-isogeny classes. Nevertheless, neither of
these problems
pose any difficulty when $\pi=\sqrt{q}$ for  an odd power $q$ of
$p$, as we shall see in \S\ref{sect:ppav-nonempty} and
Lemma~\ref{P.2} in the next section.  
\end{sect}



\section{The classification of $\ppav(\sqrt{p})$ into genera}
\label{sec:computing-ppav}
We carry out the first two steps of our strategy as described in
\S\ref{sect:3-steps} for the Weil $p$-number
$\pi=\sqrt{p}$ and classify $\ppav(\sqrt{p})$ into genera. Certain symplectic lattice classification in Step~2 will be
postponed to \S\ref{sec:LC}. 

\subsection{The first step}
\numberwithin{thmcounter}{subsection}
For the first step, there is no need to
restrict to the prime base field case yet. Assume that $q=p^n$ is an
\emph{odd} power of $p$. Every abelian variety $X/\F_q$ in the simple $\F_q$-isogeny class corresponding to the
Weil $q$-number $\pi=\sqrt{q}$ is a supersingular abelian surface. 
From (\ref{eq:23}),  we have
$\End^0(X)\simeq D_{\infty_1, \infty_2}$.
The endomorphism ring $\End(X)$ is a $\Z[\sqrt{q}]$-order
in $D_{\infty_1, \infty_2}$ uniquely determined by $X$ up to an inner
automorphism.  For simplicity, we put
\begin{equation}
  \label{eq:24}
D\coloneqq D_{\infty_1, \infty_2},\quad  F\coloneqq \Q(\sqrt{p})\quad \text{and}\quad   A\coloneqq \Z[\sqrt{q}]. 
\end{equation}

%



\begin{sect}\label{sect:ppav-nonempty}
  Let $E/\F_{q^2}$ be an elliptic curve with Frobenius endomorphism
  $\pi_E=q$, and $\lambda_E: E\to E^t$ be the canonical
  principal polarization of $E$.  Take $X\coloneqq \Res_{\F_{q^2}/\Fq}(E)$, the Weil restriction
  of $E$ with respect to $\F_{q^2}/\F_q$. Then $X$ is a superspecial
  abelian surface with $\pi_X^2=q$, and the Weil restriction
  $\lambda_X\coloneqq \Res_{\F_{q^2}/\Fq}(\lambda_E)$  is a principal polarization on
  $X$. Thus, if we write
  $\ppsp(\sqrt{q})$ for the subset of $\ppav(\sqrt{q})$ consisting of
  the members $[X, \lambda]$ with $X$ being superspecial, then
  $\ppsp(\sqrt{q})\neq \emptyset$. 
\end{sect}

\begin{sect}\label{sect:neron-severi}
  Let $X/\F_q$ be an arbitrary abelian surface with $\pi_X^2=q$, and
  $\lambda: X\to X^t$ be a $\Q$-polarization on $X$. Since
  $D=\End^0(X)$ is totally definite, the Rosati involution on
  $\End^0(X)$ induced by $\lambda$ coincides with the canonical
  involution $\alpha\mapsto \bar{\alpha}\coloneqq \Tr(\alpha)-\alpha$
  according to Albert's classification \cite[Theorem~2,
  \S21]{mumford:av}. The group $G^1$ as defined in (\ref{eq:20}) is
  just the reduced norm one subgroup of $G=\ul D^\times$.  Let
  $\calP(X)$ be the N\'eron-Severi group of $X$,
  $\calP^0(X)\coloneqq \calP(X)\otimes \Q$, and
  $\calP_+^0(X)\subseteq \calP^0(X)$ be the subset consisting of the
  $\Q$-polarizations of $X$.  From \cite[\S21,
  Application~III]{mumford:av}, the map
  $\lambda'\mapsto \lambda^{-1}\lambda'$ with $\lambda'$ ranging in $\calP^0(X)$
  induces an identification
\begin{equation}
  \label{eq:44}
 \varrho_\lambda: \calP^0(X)\simeq F \quad \text{such that}\quad
 \varrho_\lambda(\calP_+^0(X))=F_+^\times. 
\end{equation}
Here $F_+^\times$ denotes the subset of totally positive
elements of $F^\times$.  
\end{sect}

 \begin{lem}\label{P.2}
   For any two $\Q$-polarized  abelian surfaces $(X,\lambda)$ and
  $(X',\lambda')$ over $\F_q$ with $\pi_X^2=\pi_{X'}^2=q$, there exists a $\Q$-isogeny
  $\varphi:X\to X'$ such that $\varphi^*\lambda'=\lambda$. 
\end{lem}
\begin{proof}
  Take an arbitrary $\Q$-isogeny $\phi: X\to X'$.  By (\ref{eq:44}),
  $\lambda^{-1}(\phi^*\lambda')$ lies in $F_+^\times$. On the other hand, 
there exists $\alpha\in D^\times$
  such that $\bar{\alpha}\alpha=\lambda^{-1}(\phi^*\lambda')$ by 
  \cite[Theorem~III.4.1]{vigneras}. Put $\varphi\coloneqq  
  \phi \circ\alpha^{-1}$. Then $\varphi^*\lambda'=\lambda$ as desired. 
\end{proof}

\begin{rem}
  Thanks to Lemma~\ref{P.2}, $\ppav(\sqrt{q})$ forms a single
  $\Q$-isogeny class. As mentioned in \S\ref{sect:neron-severi}, $G^1$
  is just the reduced norm one subgroup of $G=\ul D^\times$.  Let
  $(X_0, \lambda_0)/\F_q$ be a $\Q$-polarized abelian surface with
  $\pi_{X_0}^2=q$, and $(V_p, \psi_p)$ (resp.~$(V_\ell, \psi_\ell)$)
  be the quasi-polarized $\mathrm{F}$-isocrystal (resp.~Tate space for
  $\ell\neq p$) attached to $(X_0, \lambda_0)$ as described in
  \S\ref{sect:genera-G1-Tate}.  For each $\ell\neq p$, the Tate space
  $V_\ell$ is simply a free $F_\ell$-module of rank $2$, and
  $\psi_\ell=\Tr_{F/\Q}\circ \psi_{\ell, F}$, where $\psi_{\ell, F}$
  is an $F_\ell$-linear
  symplectic  form on $V_\ell$ (hence uniquely
  determined up to isomorphism). At the prime $p$, the
  $\mathrm{F}$-isocrystal $V_p$ is a $\Q_q(\sqrt{p})$-space of
  dimension $2$ equipped with a Frobenius operator $\sfF$ such that
  $\sfF(ax)=a^\sigma \sfF x$ for every $a\in \Q_q$ and $x\in V_p$, and
  $\sfF^n=\pi_{X_0}$. We shall make no use of this explicit description
  of the pairing $\psi_p$ for a general $q$, though it can be easily
  deduced from the prime field case (i.e.~$q=p$) below.  
\end{rem}

\subsection{The second step}
Now we move on to the classification of
$\ppav(\pi)$ into genera, or equivalently,  the classification of self-dual
modules in $(V_\ell, \psi_\ell)$ for $\ell\in S_\pi$. 
  From (\ref{eq:72}) and \eqref{eq:24}, $S_\pi=\{2, p\}$
  since $\grd_A=4q$.   Thus our tasks are
  two-folds: 
  \begin{enumerate}
  \item[(i)] to classify the isometric classes of self-dual
    Dieudonn\'e modules in $(V_p, \psi_p)$; 
    \item[(ii)] to classify the isometric classes of self-dual
      $A_2$-lattices in $(V_2, \psi_2)$. 
  \end{enumerate}
  As mentioned previously, we shall restrict ourselves to the
  classification of the \emph{superspecial} \dieu modules and prove the
  bijections in \eqref{eq:PPSP=PPAV}. Recall that a \dieu module $M\subset V_p$ is
  superspecial if and only if  $\sfF M=\sfV M$.
  
\subsubsection{The superspecial Dieudonn\'e modules}\label{sec:dieu-prime-p}
  



  %

 The structure of \dieu modules is very easy to describe when $q=p$ (i.e.~$\pi=\sqrt{p}$). In this
case, $\sigma\in \Gal(\Q_q/\Q_p)$ is   
the identity map, and $\sfF$, $\sfV$ and $\pi$ all act as the same
operator on $V_p$. 
Thus $V_p$ is simply an $F_p$-vector space of
dimension $2$, and  a Dieudonn\'e module in $V_p$ is just a full
$A_p$-lattice.  Moreover,  condition (\ref{eq:62}) becomes
identical to  (\ref{eq:22}). 
Therefore, when $\pi=\sqrt{p}$, we treat a \dieu module 
$M(X)$ as if it is a Tate module and put  $T_p(X)\coloneqq M(X)$. Necessarily,
$T_p(X)$ is a free $O_{F_p}$-module of rank $2$ since
$A_p=O_{F_p}$. See also \cite[Theorem~6.2]{waterhouse:thesis}. 
From Lemma~\ref{LC.2}, we immediately obtain the following lemma.

\begin{lem}\label{lem:unique-dieu-p}
  Suppose that $\pi=\sqrt{p}$. There is a unique isometric class of
  self-dual \dieu modules (or equivalently here, $O_{F_p}$-lattices) in $(V_p, \psi_p)$.
\end{lem}

  From a classical
 result of Oort \cite{Oort-MathAnn-1975}, an abelian variety over a perfect field
 is superspecial if and only if its \dieu module is superspecial.
Clearly, every abelian
surface $X/\F_p$ with $\pi_X^2=p$ is superspecial.


 For a general $q=p^n$ with $n=2s+1$, we focus only on the
 \emph{superspecial} \dieu modules in $(V_p, \psi_p)$ for the
 sake of simplicity.  
 Given a superspecial \dieu module $M\subset V_p$,
 we define its \emph{skeleton} as
\begin{equation}
  \label{eq:P.3}
  \bbS(M)\coloneqq \{x\in M\,|\, \sfF x=\sqrt{p}x\}. 
\end{equation}
Define $\bbS(V_p)$ similarly. 
Clearly, $\bbS(M)$ is a $\Zp[\sqrt{p}]$-lattice in $\bbS(V_p)$.

\begin{lem}\label{lem:dieu}
The superspecial \dieu module $M\subset V_p$ is  canonically isomorphic
to $\bbS(M)\otimes_{\Z_p} \Z_q$. 
\end{lem}
\begin{proof}
  Since $M$ is
  superspecial, $p^{-1}\sfF^2M=M$. As $p^{-s}\sfF^{2s+1}=\sqrt{p}$, we
  have
\[ \sfF M=\sfF(p^{-s}\sfF^{2s}M)=\sqrt{p}M. \]
Put $\sfF_0\coloneqq \sqrt{p}^{-1}\sfF:M\to M$. Then $(M,\sfF_0)$ is a
finite free $\Z_q$-module of rank $4$ with a $\sigma$-linear automorphism. The pairs $(M,\sfF_0)$ are classified by
the first Galois cohomology $H^1(\Fq/\Fp, \GL_4(\Z_q))$, 
which is trivial by  
``Hilbert's Theorem 90'' \cite[Chap.~III, Lemma 4.10, p.~124]{Milne-etale}.
It follows that $(M,\sfF_0)$ is
isomorphic to the standard $\Z_q$-module $(\Z_q^4, \sigma)$, where
$\sigma$ acts coordinate-wisely. Then $\bbS(M)=M^{\sfF_0}$ is of
rank $4$ as well, and the canonical map $\bbS(M)\otimes_{\Z_p} \Z_q\to
M$ is an isomorphism.
\end{proof}

The above lemma allows us to show that every $\Q$-polarized superspecial abelian surface
$(Y, \mu)/\F_q$ with $\pi_Y^2=q$ descends to
$\F_p$. Let $(X_0, \lambda_0)/\F_p$ be an arbitrary $\Q$-polarized
abelian surface with $\pi_{X_0}^2=p$, and put $(Y_0, 
\mu_0)=(X_0, \lambda_0)\otimes_{\F_p} \F_q$. We have
$\End_{\F_q}^0(Y_0)=\End_{\F_p}^0(X_0)=D$. Thus the group $G(\whQ)=\whD^\times$ acts on both $\Qisog(X_0)$ and $\Qisog(Y_0)$. 


\begin{lem}\label{lem:Qisog-descend}
  Let $\Qisog^{\mathrm{sp}}(Y_0)$  be the subset of $\Qisog(Y_0)$
  consisting of all the superspecial members. The base change  $(X,
  \varphi)\mapsto (X, \varphi)\otimes_{\F_p}\F_q$ induces a
  $G(\whQ)$-equivariant bijection $\Qisog(X_0)\to \Qisog^{\mathrm{sp}}(Y_0)$.  
\end{lem}
\begin{proof}
From (\ref{eq:16}), the base change map is $G(\whQ)$-equivariant.   To prove the bijectivity, it suffices to
produce an  inverse map. Let $(Y, \phi)$ be an arbitrary
member of $\Qisog^{\mathrm{sp}}(Y)$.  The $\Q$-isogeny 
$\phi$ identifies $M(Y)$ with a superspecial \dieu module $M$ over
$\F_q$ in $V_p(Y_0)=V_p(X_0)\otimes_{\Q_p} \Q_q$. According to 
Lemma~\ref{lem:dieu}, $\bbS( M)$ is a $\Z_p[\sqrt{p}]$-lattice in
$V_p(X_0)$, and $M=\bbS(M)\otimes_{\Z_p}\Z_q$. At each prime $\ell\neq
p$, the $\Q$-isogeny 
$\phi$ identifies $T_\ell(Y)$ with a
$\Z_\ell[\sqrt{p}]$-lattice  $T_\ell$  in $V_\ell(Y_0)=V_\ell(X_0)$. The inclusion $(\bbS(M), (T_\ell)_{\ell\neq p})\subset
V_p(X_0)\times \prod_{\ell\neq p} V_\ell(X_0)$ corresponds to a unique
member $(X, \varphi)\in \Qisog(X_0)$. By our construction, the map
$\Qisog^{\mathrm{sp}}(Y_0)\to \Qisog(X_0)$ sending $(Y,
\phi)$ to $(X, \varphi)$ is the inverse of the base change
map. 
\end{proof}

\begin{cor}\label{cor:equiv-cat}
Let $\PolSp(\sqrt{q})$ be the category of $\Q$-polarized superspecial abelian surfaces $(Y, \mu)/\F_q$ with
$\pi_Y^2=q$, whose morphisms  are polarized 
$\Q$-isogenies. Then the base change functor
$-\!\otimes_{\F_p}\F_q:\PolSp(\sqrt{p})\to \PolSp(\sqrt{q})$ is an equivalence  of
categories.   In particular, it induces a bijection
$\ppav(\sqrt{p})\simeq \ppsp(\sqrt{q})$. 
 \end{cor}  

 \begin{proof}
Let $\varphi: (X_1, \lambda_1)\to (X_0, \lambda_0)$ be a
morphism in $\PolSp(\sqrt{p})$. Then   \[\Hom((X_1, \lambda_1), (X_0, \lambda_0))=D^1\varphi=\Hom((X_1,
     \lambda_1)\otimes_{\F_p} \F_q, (X_0, \lambda_0)\otimes_{\F_p}
     \F_q),  \]
where $D^1=G^1(\Q)$,  the reduced norm one subgroup of
$D^\times$. This shows that the base change functor is fully
faithful.  Fix $(Y_0, 
   \mu_0)\coloneqq (X_0, \lambda_0)\otimes_{\F_p} \F_q$ and let $(Y,
   \mu)/\F_q$ be an object  of
   $\PolSp(\sqrt{q})$.  From Lemma~\ref{P.2}, there exists a $\Q$-isogeny
   $\phi: (Y, \mu)\to (Y_0,    \mu_0)$.
   It follows from Lemma~\ref{lem:Qisog-descend} that there exists
   $(X, \varphi) \in \Qisog(X_0)$ and an isomorphism
   $\alpha: Y \to X\otimes_{\F_p} \F_q$ such that
   $\phi= (\varphi\otimes_{\F_p}\F_q)\circ\alpha$.  From the bijection (\ref{eq:7}), 
   $\alpha: (Y, \mu)\to (X, \varphi^*\lambda_0)\otimes_{\F_p}\F_q$ is
   an isomorphism of $\Q$-polarized abelian varieties.  Therefore,  every
   object of $\PolSp(\sqrt{q})$ descends to $\PolSp(\sqrt{p})$, and the base change
   functor is an equivalence of categories. 
\end{proof}

\subsubsection{The  genera classifcation  of $\ppav(\sqrt{p})$} 
For the rest of this paper, we focus exclusively on the Weil
$p$-number $\pi=\sqrt{p}$.  In particular, $A=\Z[\sqrt{p}]$ from now
on. From the classification of \dieu modules (particularly,
Lemma~\ref{lem:unique-dieu-p}), we see that to classify
$\ppav(\sqrt{p})$ into genera, it is enough to classify the isometric classes of self-dual
      $A_2$-lattices in $(V_2, \psi_2)$, as $A_\ell$ may  be a non-maximal order only when $\ell=2$.. 








We first give a brief
recount of the genus classification of the unpolarized case in \cite[\S6]{waterhouse:thesis} and
\cite[\S6.1]{xue-yang-yu:ECNF}. 
Recall that $\Isog(\sqrt{p})$ denotes
the set of $\F_p$-isomorphism classes of abelian surfaces in the
simple isogeny class corresponding to  $\pi=\sqrt{p}$. 



Let $X_0/\F_p$ be an arbitrary abelian surface with $\pi_{X_0}^2=p$
and $(X, \varphi)$ be a member of $\Qisog(X_0)$. For
every prime $\ell$ (including $\ell=p$), $T_\ell(X)$ is a full $A_\ell$-lattice in
$V_\ell\simeq F_\ell^2$. 
If either $p\not\equiv 1\pmod{4}$ or $\ell\neq 2$, then $A_\ell=O_{F_\ell}$, and hence
$T_\ell(X)\simeq O_{F_\ell}^2$. Thus
$\Isog(\sqrt{p})$ forms a single genus when $p\not\equiv 1\pmod{4}$. 
If $p\equiv 1 \pmod{4}$ and
$\ell=2$, then $[O_{F_2}:A_2]=2$, and $T_2(X)$ is isomorphic to one of the
following three $A_2$-lattices in $(F_2)^2$:
\begin{equation}
  \label{eq:26}
O_{F_2}^2,\qquad A_2\oplus O_{F_2},\qquad   A_2^2. 
\end{equation}
Accordingly, $\Isog(\sqrt{p})$ decomposes into three genera when
$p\equiv 1\pmod{4}$:
\begin{equation}
  \label{eq:28}
  \Isog(\sqrt{p})=\Lambda_1^\un\amalg \Lambda_8^\un\amalg \Lambda_{16}^\un.
\end{equation}
Here the superscript $^\un$ stands for ``unpolarized'', and
$\Lambda_1^\un$ is the subset consisting of all members
$[X]\in \Isog(\sqrt{p})$ such that $T_2(X)\simeq O_{F_2}^2$. The
genera $\Lambda_8^\un$ and $\Lambda_{16}^\un$ are defined similarly.
The subscripts $1, 8, 16$ are chosen for $\Lambda_r^\un$ to indicate
that \begin{equation}\label{eq:end0index}
  [\bbO: \End(X_r)]=r  
\end{equation}
for every $[X_r]\in \Lambda_r^\un$ and every maximal $O_F$-order
$\bbO\subset D$ containing $\End(X_r)$. 
 See
\cite[Theorem~6.2]{waterhouse:thesis}, 
\cite[Theorem~6.1.2]{xue-yang-yu:ECNF} or \S\ref{sect:end-ring} below.  For uniformity, we also put
$\Lambda_1^\un=\Isog(\sqrt{p})$ when $p\not\equiv 1\pmod{4}$.  As a
convention, when a result is stated for $\Lambda_r^\un$ with
$r\in \{1, 8, 16\}$, it means that the said result holds for
$\Lambda_1^\un$ for all primes $p$, and also for $\Lambda_8^\un$ and
$\Lambda_{16}^\un$ when $p\equiv 1\pmod{4}$.

\begin{prop}\label{P.4}
  Consider the following  forgetful
  map:
  \begin{equation}
    \label{eq:66}
    f: \ppav(\sqrt{p})\to \Isog(\sqrt{p}), \qquad [X, \lambda] \mapsto
    [X]. 
  \end{equation}
Put $ \Lambda_1^\pp\coloneqq \ppav(\sqrt{p})$ if $p\not\equiv 1\pmod{4}$, and $ \Lambda_r^\pp\coloneqq f^{-1}(\Lambda_r^{\un})$ for $r\in \{1, 8, 16\}$ if
$p\equiv 1\pmod{4}$. Then the following holds true: 
\begin{enumerate}
\item[(1)] when $p\not\equiv
1\pmod{4}$, $\Lambda_1^\pp=\ppav(\sqrt{p})$ forms a single genus;  
\item[(2)] when $p\equiv 1\pmod{4}$,  $\Lambda_8^\pp=\emptyset$, and 
  \begin{equation}
    \label{eq:50}
    \ppav(\sqrt{p})=\Lambda_1^\pp\amalg \Lambda_{16}^\pp,
  \end{equation}
where each $\Lambda_r^\pp$ for $r\in \{1, 16\}$ forms a single
nonempty genus.
\end{enumerate}
\end{prop}
\begin{proof}
If $p\not\equiv 1\pmod{4}$, then $A_2=O_{F_2}$, and there is a unique
isometric class of self-dual $O_{F_2}$-lattices in $(V_2, \psi_2)$ by
Lemma~\ref{LC.2}. 
 Part (1) of  the proposition follows directly, so assume that $p\equiv
 1\pmod{4}$ for the rest of the proof.

We first prove that $\Lambda_8^\pp=\emptyset$. It is enough to show
that there is  no self-dual  $A_2$-lattice $L$ in 
  $(V_2, \psi_2)$ such that $L\simeq O_{F_2}\oplus A_2$.    Suppose that such a lattice $L$ exists. Then $L=O_{F_2}e_1+A_2 e_2$
  for some $F_2$-basis $e_1,e_2$ of $V_2$ and $\psi_2(L,L)\subset
  \Z_2$. Denote by $\psi_{2, F}:V_2\times V_2 \to F_2$ the
unique $F_2$-bilinear pairing such that 
$\Tr_{F_2/\Q_2}\circ \psi_{2, F}=\psi_2$.
  If
  we put $a\coloneqq\psi_{2,F}(e_1,e_2)$, then $\psi_{2,F}(L,L)=O_{F_2}a$.  Clearly,
  the lattice $\wt L\coloneqq O_{F_2}e_1+O_{F_2} e_2$ also has the property
  $\psi_F(\wt L,\wt L)=O_{F_2} a$ and
  $\psi_2(\wt L,\wt L)\subset \Z_2$. Let $\wt{L}^\vee$ be the
  dual lattice of $\wt{L}$ with respect to $\psi_2$. Then   $L\subsetneq \wt L\subseteq \wt L^\vee \subsetneq L^\vee$, which is
  absurd.

Next, the Tate module $T_2(X)$ of any $[X]\in
\Lambda_{16}^\un$ is a free
$A_2$-lattice of rank $2$.  According to Proposition~\ref{LC.3}, there is a unique isometric class of
self-dual free $A_2$-lattices in $(V_2,
\psi_2)$. Thus $\Lambda_{16}^\pp$ forms a single nonempty genus. The case for
$\Lambda_1^\pp$ is proved similarly, except that one  applies 
Proposition~\ref{LC.2} instead of Proposition~\ref{LC.3}. 
\end{proof}

\begin{rem}\label{rem:8-pol-mod}
  Let $[X]$ be a member in $\Lambda^\un_8$, and put $L\coloneqq T_2(X)$ and $\wt
      L\coloneqq O_{F_2} L$.  
      Let $\eta: X\to \wt X$ be the isogeny
      corresponding to the inclusion $L\subset \wt L$, which is the minimal
      isogeny with the property that $O_F\subset \End(\wt X)$.  It
      follows from the  proof of Proposition~\ref{P.4} and
      \cite[Theorem~16.4]{Milne-AV} that every  $\lambda\in \calP(X)$
      is of the form $\eta^*(\wt\lambda)$ for some $\wt \lambda\in
      \calP(\wt X)$, where $\calP(X)$ denotes the N\'eron-Severi group
      of $X$ as in  \S\ref{sect:neron-severi}. 
 Therefore, $\eta^*$ induces a bijection $\calP(\wt X)\to \calP(X)$
 that identifies the set of polarizations $\calP_+(\wt X)\subset
 \calP(\wt X)$ with $\calP_+(X)\subset \calP(X)$, so we write
      \begin{equation}
        \label{eq:36}
     \eta^*:   (\calP(\wt X), \calP_+(\wt X)) \cong (\calP(X), \calP_+(X)).
   \end{equation}
   The ordered pair $(\calP(X), \calP_+(X))$ will be investigated in
   detail in \S\ref{sec:PP}.
\end{rem}

\section{The class number and type number of $\ppav(\sqrt{p})$}
\label{sec:S4}
\numberwithin{thmcounter}{section}

We move on to  Step 3 of our strategy for computing
$\abs{\ppav(\sqrt{p})}$  as described in
\S\ref{sect:3-steps}, so keep the notation of the previous section. Particularly, 
$F=\Q(\sqrt{p})$ and $A=\Z[\sqrt{p}]$. 
  Let $D=D_{\infty_1, \infty_2}$ be the unique totally definite
quaternion $F$-algebra that is unramified at all the finite primes of
$F$ as in \eqref{eq:24}. 
Recall that  the algebraic group $G$ defined in \eqref{eq:17} is just the
multiplicative group $\ul D^\times$, and $G^1$ defined in \eqref{eq:20} 
  is  the reduced norm one subgroup of $\ul D^\times$. 
Given a subset $S\subseteq \whD$, we write $S^1$ for the subset of
elements in $S$ with reduced norm $1$, that is,
$S^1\coloneqq \{x\in S\mid \Nr(x)=1\}$. Thus $G^1(\whQ)=\whD^1$ and
$G^1(\Q)=D^1$. Assume that one of the following conditions holds:
\begin{itemize}
    \item $p\not \equiv 1\pmod{4}$ and $r=1$;
    \item $p\equiv 1\pmod{4}$ and $r\in \{1, 16\}$.
\end{itemize}
Let $\Lambda_r^\pp$ be the genus of principally polarized abelian
surfaces as defined in Proposition~\ref{P.4}. 
Fix a member $[Y_r, \lambda_r]\in \Lambda_r^\pp$ and put
\begin{equation}
  \label{eq:74}
 \grO_r\coloneqq \End(Y_r).
\end{equation}
From Proposition~\ref{prop:H-genus-isom}, there is a canonical bijection
  \begin{equation}
\label{eq:75}
\Lambda_r^\pp\simeq D^1\bsh \whD^1/\wh\grO_r^1   
\end{equation}
sending $[Y_r, \lambda_r]$ to the identity class.  Thus to compute
$\abs{\Lambda_r^\pp}$,  we need a concrete characterization of the
$A$-orders in $D$ that appear as endomorphism rings of the underlying
(principally polarizable) abelian surfaces of $\Lambda_r^\pp$.
The class number calculations  will be 
postponed to Sections~\ref{sec:max-ord}--\ref{sec:class-number}.

\begin{sect}\label{sect:end-ring}
We first recount the description of endomorphism rings in the
unpolarized case.
  Let $[X_1]$ be a member of $\Lambda_1^\un$ and put
  $\calO_1\coloneqq \End(X_1)$. Then $\calO_1$ is a maximal order in $D$ since
  $\calO_\ell=\End_{A_\ell}(T_\ell(X_1))\simeq \Mat_2(O_{F_\ell})$ for
  every prime $\ell$. Fix an isomorphism $T_2(X_1)\simeq O_{F_2}^2$,
  which induces an identification
  $\calO_1\otimes\Z_2=\Mat_2(O_{F_2})$.  If $p\equiv 1\pmod{4}$, then
  the inclusion $A_2\oplus O_{F_2}\hookrightarrow O_{F_2}^2$
  (resp.~$A_2^2\hookrightarrow O_{F_2}^2$) gives rise to an isogeny
  $X_8\to X_1$ (resp.~$X_{16}\to X_1$) with $[X_r]\in \Lambda_r^\un$
  for $r\in\{8, 16\}$. The endomorphism rings $\calO_r\coloneqq \End(X_r)$ for
  $r=8, 16$ are characterized by
\begin{equation}
\label{eq:31}
\begin{gathered}
  (\calO_8)_2\coloneqq \calO_8\otimes_{\Z} \Z_2=
\begin{pmatrix}
  A_2 & 2 O_{F_2} \\  O_{F_2} & O_{F_2}\\
\end{pmatrix},\qquad (\calO_{16})_2=\Mat_2(A_2),
 \\   
(\calO_r)_\ell=(\calO_1)_\ell\qquad \forall\,
\text{prime } \ell\neq 2, \quad r\in \{8,16\}. 
\end{gathered}
\end{equation}
The order $\calO_r$ has index $r$ in $\calO_1$.

From Proposition~\ref{prop:H-genus-isom}, there is a natural bijection
\begin{equation}
  \label{eq:32}
  \Lambda_r^\un\simeq G(\Q)\bsh G(\whQ)/\wcO_r^\times=D^\times\bsh
  \whD^\times/\wcO_r^\times  
\end{equation}
sending the base member $[X_r]$ to the identity class.   This
also establishes a bijection between $\Lambda_r^\un$ and the set
$\Cl(\calO_r)$ of locally principal right ideal $\calO_r$-classes as
discovered by Waterhouse in
\cite[Theorem~6.2]{waterhouse:thesis}. Indeed, from \cite[\S
III.5]{vigneras}, $\Cl(\calO_r)$ admits
the same adelic double coset description as in (\ref{eq:32}).  Thus the 
computation of $\abs{\Isog(\sqrt{p})}$ is reduced to that of the class numbers
of the quaternion orders $\calO_r$. 

When $p\equiv 1\pmod{4}$, the class
number formula for $h(\calO_1)\coloneqq \abs{\Cl(\calO_1)}$ is first computed explicitly by Peters
\cite[p.~363]{Peters1968}, and also by Kitaoka \cite{kitaoka:nmj1973}
in light of its 
relationship to the type number $t(\calO_1)$ as in \eqref{eq:43} below,
which in turn can be interpreted as the 
proper class number of quaternary positive definite even
quadratic lattices of discriminant $p$
\cite[p.~85]{chan-peters}. Ponomarev \cite{ponomarev:aa1981, ponomarev:aa1976}
extended Kitaoka's result to all prime $p$. Vign\'eras
\cite{vigneras:ens} gave explicit formulas for the class number of any
Eichler order of square-free level in a totally definite quaternion
algebra over an arbitrary quadratic real field. When $p\equiv 1\pmod{4}$, the class number
formulas for $\calO_8$ and $\calO_{16}$ are worked out by the present
authors together with Tse-Chung Yang in \cite{xue-yang-yu:ECNF}.
\end{sect}

\begin{rem}
 We provide another proof for 
$\Lambda_8^\pp=\emptyset$. Suppose that on the
  contrary  a member $[X,\lambda]\in \Lambda_{8}^\pp$ exists.
  The Rosati involution  induced by $\lambda$ necessarily leaves the
  endomorphism rings $\End(X)$ and $\End(X)\otimes\Z_2$ stable. On the
  other hand, the Rosati involution coincides with the canonical
  involution.  This already leads to a contradiction since
  $\End(X)\otimes\Z_2$, which is conjugate to $\calO_{8}\otimes \Z_2$
  in (\ref{eq:31}), is not stable under the canonical involution. 
\end{rem}

\begin{sect}
The bijection $\Lambda_r^\un\simeq \Cl(\calO_r)$ fits into the general
framework of the arithmetic of quaternion algebras as follows. Two
$A$-orders $\calO$ and $\calO'$ in $D$ are said to belong to \emph{the
  same genus} if there exists $x\in \whD^\times$ such that
$\wcO'=x \wcO x^{-1}$, or equivalently, if $\calO_\ell$ and
$\calO'_\ell$ are $A_\ell$-isomorphic at every prime $\ell\in\bbN$.  For
example, all maximal orders of $D$ belong to the same genus. The orders $\calO$ and $\calO'$ are said to be of the \emph{same
  type} if they are $A$-isomorphic (equivalently,
$D^\times$-conjugate). Let
\[ \dbr{\calO}\coloneqq \{\alpha \calO \alpha^{-1}\mid \alpha\in D^\times\} \] 
be
the type of $\calO$, and $\Tp(\calO)$ be the set of types of
$A$-orders in the genus of $\calO$.  It can be described adelically as
\begin{equation}
  \label{eq:37}
  \Tp(\calO)\simeq D^\times\bsh   \whD^\times/\calN(\wcO), 
\end{equation}
where $\calN(\wcO)$ denotes the normalizer of $\wcO$ in
$\whD^\times$. 
Given a locally principal right
$\calO$-ideal $I$, the left order of $I$ is defined to be
\begin{equation}
  \label{eq:34}
\calO_l(I)=\{\alpha\in D\mid \alpha I\subseteq I\}.   
\end{equation}
If we write $\whI=x\wcO$ for some $x\in \whD^\times$, then $\widehat{\calO_l(I)}=x
\wcO x^{-1}$,  so $\calO_l(I)$ belongs to the same genus as
$\calO$. There is a natural surjective map
\begin{equation}
  \label{eq:39}
 \Upsilon:\Cl(\calO)\twoheadrightarrow \Tp(\calO),\qquad 
[I]\mapsto \dbr{\calO_l(I)}.
\end{equation}

Now let $\calO=\calO_r$ for some $r\in \{1, 8, 16\}$. We have a
commutative diagram
\begin{equation}
  \label{eq:38}
  \begin{tikzcd}
    \Lambda_r^\un\ar[r, leftrightarrow, "\simeq"]\ar[rd, leftrightarrow,
    "\simeq"]  & \Cl(\calO_r)\ar[d, leftrightarrow, "\simeq"]\ar[r,
    twoheadrightarrow, "\Upsilon"] & \Tp(\calO_r)\ar[d, leftrightarrow, "\simeq"]\\
          & D^\times\bsh
  \whD^\times/\wcO_r^\times\ar[r, twoheadrightarrow] & D^\times\bsh
  \whD^\times/\calN(\wcO_r)
  \end{tikzcd}
\end{equation}
where the bottom horizontal map is the canonical projection. From
(\ref{eq:29}), the composition of the maps in the top row coincides
with the following  map
\begin{equation}
  \label{eq:40}
\Upsilon^\un:  \Lambda_r^\un\twoheadrightarrow \Tp(\calO_r), \qquad [X]\mapsto \dbr{\End(X)}. 
\end{equation}
Thus an $A$-order $\calO'\subset D$ is isomorphic
to the endomorphism ring of some $[X_r']\in \Lambda_r^\un$ if and
only if it belongs to the same genus as $\calO_r$. Compare with
\cite[Theorem~3.13]{waterhouse:thesis} and 
Lemma~\ref{lem:spinor-genus} below.

The explicit formula for the type number
$t(\calO_1)\coloneqq \abs{\Tp(\calO_1)}$ can be traced back to the work of Peters
\cite{Peters1968}, Kitaoka \cite{kitaoka:nmj1973}, Ponomarev
\cite{ponomarev:aa1976, ponomarev:aa1981} as before.  The type number
formulas for $\calO_8$ and $\calO_{16}$ are computed by the
present authors in \cite[\S4]{xue-yu:type_no}. In particular, it is shown
there that
\begin{equation}
  \label{eq:41}
  \calN(\wcO_1)=\whF^\times\wcO_1^\times,  \qquad
   \calN(\wcO_{16})=\whF^\times\wcO_{16}^\times,
   \quad\text{while}\quad    \calN(\wcO_8)=\whF^\times\wcO_4^\times.
\end{equation}
Here $\calO_4\coloneqq O_F\calO_8$, which is the unique suborder of $\calO_1$
such that
\begin{equation}
\label{eq:C.3}
   \calO_4\otimes \Z_2=
\begin{pmatrix}
  O_{F_2} & 2 O_{F_2} \\  O_{F_2} & O_{F_2}\\
\end{pmatrix},\qquad (\calO_4)_\ell=(\calO_1)_\ell,\ 
\forall \ell\neq 2. 
\end{equation}
According to \cite[Proposition~4.1]{xue-yu:type_no}, the
class numbers $h(\calO_r)$ and the type numbers $t(\calO_r)\coloneqq \abs{\Tp(\calO_r)}$ are related by
\begin{equation}
  \label{eq:43}
  h(\calO_1)=h(O_F)t(\calO_1),\quad
  h(\calO_{16})=h(A)t(\calO_{16}),\quad h(\calO_4)=h(O_F)t(\calO_8). 
\end{equation}
See \eqref{eq:131} for the formula of $h(A)$.
\end{sect}


Now let us move on to the principally polarized case. Recall from
\eqref{eq:74} that $\grO_r=\End(Y_r)$ denotes the endomorphism ring of
a fixed member $[Y_r, \lambda_r]\in \Lambda_r^\pp$.  For any other
member $[Y_r', \lambda_r']\in \Lambda_r^\pp$, the endomorphism ring
$\End(Y_r')$ belongs to the same genus as $\grO_r$. However, we would like to emphasize that the
converse needs not to be true, that is, a priori not every order in
the genus of $\grO_r$ is necessarily the endomorphism ring of some underlying
(\emph{principally polarizable}) abelian surface of $\Lambda_r^\pp$.
Thus we need to characterize the image of the map
 \begin{equation}
   \label{eq:68}
\Upsilon^\pp:  \Lambda_r^\pp\to  \Tp(\grO_r),\qquad   [X, \lambda]\mapsto
 \dbr{\End{X}}.   
 \end{equation}
For this purpose,  let us recall the notion of \emph{spinor genus} of orders  from
\cite[\S1]{Brzezinski-Spinor-Class-gp-1983}.

\begin{defn}\label{defn:spinor-genus-class}
  Two $A$-orders $\calO, \calO'\subset D$  are in the \emph{same
  spinor genus}  if there
  exists  $x\in D^\times\whD^1$ such that
  $\wcO'= x \wcO x^{-1}$. 
\end{defn}
Clearly, ``being in the same spinor genus'' is an equivalence
relation that is finer than ``being in the same genus'' and coarser
than ``being of the same type''.

\begin{lem}\label{lem:spinor-genus}
Keep $\grO_r=\End(Y_r)$ for some $[Y_r, \lambda_r]\in
\Lambda_r^\pp$ as in (\ref{eq:74}).   For any $A$-order $\calO\subset D$, there exists
  $[X, \lambda]\in \Lambda_r^\pp$ such that $\End(X)\simeq
  \calO$ if and only if $\calO$ belongs to the same spinor genus as
  $\grO_r$. 
\end{lem}
\begin{proof}
Clearly, $\Upsilon^\pp:\Lambda_r^\pp\to \Tp(\grO_r)$ is the composition of $f:
\Lambda_r^\pp\to \Lambda_r^\un$ with the map $\Upsilon^\un$ in (\ref{eq:40}) (with
$\grO_r$ in place of $\calO_r$).
Combining (\ref{eq:75}) and (\ref{eq:38}), we
get a commutative diagram
\begin{equation}
  \label{eq:73}
  \begin{tikzcd}
    \Lambda_r^\pp\ar[r, "f"] \ar[d, "\simeq"]& \Lambda_r^\un\ar[r,
    twoheadrightarrow, "\Upsilon^\un"] \ar[d, "\simeq"] & \Tp(\grO_r) \ar[d, "\simeq"]\\
D^1\bsh \whD^1/\wh\grO_r^1 \ar[r]& D^\times\bsh
\whD^\times/\wh\grO_r^\times\ar[r, twoheadrightarrow] & D^\times\bsh \whD^\times/\calN(\wh\grO_r)
  \end{tikzcd}
\end{equation}
Both of 
the bottom horizontal maps are canonical projections. Clearly,  $\dbr{\calO}\in \Upsilon^\pp(\Lambda_r^\pp)$ if
and only if $\calO$ belongs to the same spinor genus as
  $\grO_r$. 
\end{proof}

\begin{prop}\label{prop:spinor-gen-classif}
  \begin{enumerate}[leftmargin=*, label=(\arabic*)]
  \item   If $p\not\equiv 3\pmod{4}$, then all maximal orders of $D$ form a
single spinor genus.   
  \end{enumerate}\vspace*{-2\partopsep}
  \begin{enumerate}[resume]
    \item If $p\equiv 1\pmod{4}$, then all $A$-orders  in the
      genus of $\grO_{16}$ form a
      single spinor genus.
    \item   If $p\equiv 3\pmod{4}$, then the maximal orders of $D$
      separate into two spinor genera. Accordingly the type set
      $\Tp(D)$ of maximal orders in $D$ decomposes
into two  nonempty subsets
\begin{equation}
  \label{eq:67}
  \Tp(D)=\Tp^+(D)\amalg \Tp^{-}(D),  
\end{equation}
where $\Tp^+(D)\coloneqq\img(\Upsilon^\pp)$, and $\Tp^-(D)\coloneqq
\Tp(D)\smallsetminus \Tp^+(D)$. 
  \end{enumerate}
\end{prop}
\begin{proof}
For any $A$-order $\calO\subset D$, we write $\dbr{\calO}_\sg$ for the
spinor genus of $\calO$, and $\SG(\calO)$ for the set of
spinor genera within the genus of $\calO$, regarded as a pointed set with 
base point $\dbr{\calO}_\sg$; see
\cite[Definition~2.2]{Xue-Yu-Selec-2022}.  In other words,
\begin{equation}
  \label{eq:143}
  \SG(\calO)=\{\dbr{\calO'}_\sg \mid \exists x\in \whD^\times \text{
    such that } \wcO'=x \wcO x^{-1}\}.
\end{equation}
From
\cite[(2.2)]{Xue-Yu-Selec-2022}, there is an adelic description of
$\SG(\calO)$ as follows: 
\begin{equation}\label{eq:3xy}
  \SG(\calO)\simeq  (D^\times\whD^1)\bsh \whD^\times/\calN(\wcO)\xrightarrow[\simeq]{\Nr}
F_+^\times\bsh \whF^\times/\Nr(\calN(\wcO)), 
\end{equation}
where the two double coset spaces are canonically bijective via the
reduced norm map. Here we have applied the  Hasse-Schilling-Maass
theorem \cite[Theorem~33.15]{reiner:mo}
\cite[Theorem~III.4.1]{vigneras} to obtain 
$\Nr(D^\times)=F_+^\times$, the subgroup of totally positive elements of
$F^\times$.



For simplicity, let us put $R_1\coloneqq O_F$, and
$R_{16}\coloneqq A$.  Let $  \Pic_+(R_r)$ be the narrow class
group of $R_r$, which can be described adelically as 
\begin{equation}
  \label{eq:14}
  \Pic_+(R_r)\simeq 
\whF^\times/F_+^\times\widehat R_r^\times. 
\end{equation}
From (\ref{eq:41}),
$\calN(\wh\grO_r)=\whF^\times\wh\grO_r^\times$ for $r\in \{1,
16\}$. It follows from \eqref{eq:31} and \eqref{eq:3xy} that
\begin{equation}\label{eq:sg9}
\SG(\grO_r)\simeq
\whF^\times/\big(F_+^\times\whF^{\times 2}\widehat R_r^\times\big), 
\end{equation}
which is canonically identifiable with the \emph{Gauss
  genus group} \cite[Definition~14.29]{Cohn-invitation-Class-Field} 
\begin{equation}
  \label{xeq:11x}
\grG(R_r)\coloneqq \Pic_+(R_r)/\Pic_+(R_r)^2.
\end{equation}
Here $\Pic_+(R_r)^2$ denotes the subgroup of $\Pic_+(R_r)$ consisting
of the classes that are perfect squares in  $\Pic_+(R_r)$, and
$\whF^{\times 2}$ is defined similarly.
First, suppose that $r=1$ so that $R_1=O_F$. It is
well known \cite[Theorem~14.34]{Cohn-invitation-Class-Field} that the
Gauss genus group $\grG(O_F)$ has order $2^{t-1}$, where $t$ is the number of
finite ramified primes   in $F/\Q$, so in our case 
\begin{equation}
  \label{xeq:26x}
 \abs{\grG(O_F)}=
  \begin{cases}
     1 & \text{ if } p\not\equiv 3\pmod{4},\\
     2 & \text{ if } p\equiv 3\pmod{4}.
  \end{cases}
\end{equation}
The proves part (1) and (3) of the lemma. 
Next, suppose that $p\equiv 1\pmod{4}$ and $r=16$ so that
$R_{16}=A$.  We shall show in  Lemma~\ref{lem:odd-class-no} that
the narrow class number $h_+(A)\coloneqq \abs{\Pic_+(A)}$ is odd. It
then 
follows from definition \eqref{xeq:11x} that
\begin{equation}\label{eq:x30y}
  \abs{\grG(A)}=1\qquad \text{if}\quad p\equiv 1\pmod{4},  
\end{equation}
which completes the proof of part (2) of the lemma. 
\end{proof}
 As
a result,  when $p\not\equiv 3\pmod{4}$, for every  maximal $O_F$-order
  $\calO_1\subset D$ there exists some member
  $[X_1, \lambda_1]\in \Lambda_1^\pp$ such that
  $\End(X_1)\simeq \calO_1$. A similar statement holds for every 
  $A$-order $\calO_{16}\subset D$ satisfying (\ref{eq:31}) and the genus
  $\Lambda_{16}^\pp$  when
  $p\equiv 1\pmod{4}$. In fact, the member $[X_r, \lambda_r]\in
  \Lambda_r^\pp$ with $r\in \{1, 16\}$ turns out to be unique
   by Lemmas~\ref{lem:bije-eichler-p1mod4} and \ref{lem:biject-O16} respectively.  Thus we have 


\begin{lem}\label{lem:Phi-bije-p1mod4}
  Suppose that either $p\equiv 1\pmod{4}$ and $r\in \{1, 16\}$ or
  $(p, r)=(2,1)$. The map $\Upsilon^\pp: \Lambda_r^\pp\to \Tp(\grO_r)$ is
  bijective. 
\end{lem}


 The situation is quite different when
 $p\equiv 3\pmod{4}$.  Combining the equations
 \eqref{eq:3xy}--\eqref{xeq:26x}, we have constructed a map 
\begin{equation}\label{eq:11xi}
\Xi:  \Tp(D)=\Tp(\grO_1)\to \SG(\grO_1)\simeq \grG(O_F), 
\end{equation}
whose neutral (resp.~non-neutral) fiber is $\Tp^+(D)$
(resp.~$\Tp^{-}(D)$). A maximal order $\bbO$ in
  $D$ is said to belong to the \emph{principal spinor genus} if
  $\dbr{\bbO}\in \Tp^+(D)$,   otherwise it is said to belong to the
  \emph{nonprincipal spinor genus}.
We will 
produce an explicit maximal order $\bbO_0$ in (\ref{eq:42}) such that
$\dbr{\bbO_0}\in \Tp^+(D)$.  Thus both $\Tp^+(D)$ and $\Tp^{-}(D)$ are
characterized purely in terms of quaternion arithmetic.

%

\begin{sect}
For the moment let $p$ be an arbitrary prime.   Let $(X, \lambda_X)=\Res_{\F_{p^2}/\Fp}(E, \lambda_E)$ be as in
  \S\ref{sect:ppav-nonempty}, where  $E/\F_{p^2}$ is an  elliptic curve with
  $\pi_{E}=p$. By functoriality, $\End_{\F_{p^2}}(E)\otimes \Z[\pi_X]$ acts on
$X$. From \cite[Remark 4]{Diem-Naumann-2003}, this gives rise to an
identification
\begin{equation}
  \label{eq:46}
\End_{\F_p}^0(X)=\End_{\F_{p^2}}^0(E)\otimes \Q(\pi_X)
\end{equation}
such that 
\begin{equation}
  \label{eq:endo}
\End_{\F_p}(X)\otimes   \Z[{1}/{p}]=  \End_{\F_{p^2}}(E)\otimes \Z[{1}/{p}][\pi_X]. 
\end{equation}
In other words, $\End_{\F_p}(X)$ and $\End_{\F_{p^2}}(E)\otimes
\Z[\pi_X]$ differs at most at  $p$. 
If $p\equiv 1\pmod 4$, then $\Z[1/p][\pi_X]\otimes_{\Z[1/p]}
  \Z_2\simeq A_2$, and we find that 
\[ \End_{\F_p}(X)\otimes_\Z \Z_2 \simeq \End_{\F_{p^2}}(E)\otimes_\Z
  A_2\simeq \Mat_2(\Z_2)\otimes_{\Z_2}A_2\simeq \Mat_2(A_2). \]
  It follows from (\ref{eq:31}) that $[X, \lambda_X]\in
  \Lambda_{16}^\pp$ in this case.  





When written down explicitly, the identification in (\ref{eq:46}) is
just $D=D_{p, \infty}\otimes F$, where  $D_{p, \infty}$ is the unique
quaternion $\Q$-algebra ramified precisely at $p$ and
$\infty$. From
\cite[Theorem~4.2]{waterhouse:thesis}, $\bbo\coloneqq \End_{\F_{p^2}}(E)$ is a
maximal $\Z$-order in $\End_{\F_{p^2}}^0(E)=D_{p, \infty}$, and
conversely by \cite[Theorem~3.13]{waterhouse:thesis},  every
maximal $\Z$-order in $D_{p, \infty}$ occurs as the endomorphism ring
of some elliptic curve in the isogeny class of $E/\F_{p^2}$. 
According to \cite[Lemma~2.11]{li-xue-yu:unit-gp}, there is a \emph{unique}
$A$-order $\calM(\bbo)$ in $D$ properly containing $\bbo\otimes A$ such that
$  \calM(\bbo)\otimes \Z_\ell =\bbo\otimes A_\ell$ at every prime
$\ell\neq p$. Such an $A$-order $\calM(\bbo)$ is necessarily maximal
at $p$, i.e. $\calM(\bbo)\otimes \Z_p\simeq
\Mat_2(O_{F_p})$. In particular, $\calM(\bbo)$ is a maximal $O_F$-order in $D$ when
$p\not\equiv 1\pmod{4}$. 
On the other hand,
$\End_{\F_p}(X)$ is maximal at $p$ as described in
\S\ref{sect:end-ring}. Therefore, 
\begin{equation}
  \label{eq:48}
  \End_{\F_p}(X)=\calM(\bbo).
\end{equation}
 We work out an explicit example in the case $p\equiv 3\pmod{4}$
 below. 
\end{sect}

\begin{ex}\label{ex:res-p3mod4}
Suppose that $p\equiv 3\pmod{4}$.  According to
  \cite[Exercise~III.5.2]{vigneras}, $D_{p, \infty}$ can be presented
  as $\qalg{-1}{-p}{\Q}$,      and $\bbo_2\coloneqq \Z[i, (1+j_p)/2]$ is a maximal
  $\Z$-order in $\qalg{-1}{-p}{\Q}$. Here $\{1, i, j_p, k_p\}$ denotes the
  standard basis of $\qalg{-1}{-p}{\Q}$, where we write a subscript
  $_p$ to emphasize that $j_p^2=k_p^2=-p$. Then
  $D=D_{p, \infty}\otimes F=\qalg{-1}{-1}{F}$, which has a new standard
  basis $\{1, i, j, k\}$ by putting $j\coloneqq j_p/\sqrt{p}$ and
  $k\coloneqq k_p/\sqrt{p}$.  According to
  \cite[Proposition~5.7]{li-xue-yu:unit-gp}, the following is a
  maximal $O_F$-order\footnote{This order in denoted by $\bbO_8$ in \cite{li-xue-yu:unit-gp} since
    $\abs{\bbO_0^\times/O_F^\times}=\abs{\bfD_4}=8$. 
    Unfortunately,  we used the same 
    notation $\bbO_8$ in
    \cite{xue-yang-yu:ECNF, xue-yu:type_no} for the
    $A$-order that is currently denoted as $\calO_8$. To make a distinction, throughout this
    paper the letter $\bbO$ is
    reserved for a maximal $O_F$-order, and $\bbO_0$ is reserved for
    the maximal order in \eqref{eq:42}. } in $D$:
  \begin{equation}
    \label{eq:42}
    \bbO_0\coloneqq O_F+O_Fi+O_F\frac{\sqrt{p}+j}{2}+O_F\frac{\sqrt{p}i+k}{2}\subset
\qalg{-1}{-1}{F}. 
  \end{equation}
  Clearly, $\bbO_0\supset
\bbo_2\otimes O_F$, so $\bbO_0=\calM(\bbo_2)$ and $\dbr{\bbO_0}\in \Tp^+(D)$. 
  In fact, $\bbO_0$ is the unique maximal $O_F$-order in $D$
  up to conjugation satisfying
  $\bbO_0^\times/O_F^\times\simeq \mathbf{D}_4$ (resp.~$\bfD_{12}$) if
  $p\geq 7$ (resp.~$p=3$). Here $\bfD_n$ denotes the dihedral group of
  order $2n$. 
\end{ex}




Combining Lemma~\ref{lem:spinor-genus} with
Example~\ref{ex:res-p3mod4}, we obtain the following result. 
\begin{lem}\label{lem:end-ring-p3mod4}
  Suppose that $p\equiv 3\pmod{4}$. For any maximal $O_F$-order
  $\bbO\subset D$, there exists
  $[X, \lambda]\in \ppav(\sqrt{p})$ such that $\End(X)\simeq
  \bbO$ if and only if $\bbO$ belongs to the same spinor genus as
  $\bbO_0$. 
\end{lem}

\begin{proof}[Proofs of Theorem~\ref{thm:main} and Theorem~\ref{thm:type-num}]
For any $A$-order $\calO$ in $D$, we put
\begin{equation}
  \label{eq:60}
  h^1(\calO)\coloneqq \abs{D^1\bsh \whD^1/\wcO^1}, 
\end{equation}
which depends only on the spinor genus of $\calO$. 
Let $t(\calO)=\abs{\Tp(\calO)}$ be the type number of $\calO$, and
$t(\Lambda_r^\pp)$ be the cardinality of 
$\Upsilon^\pp(\Lambda_r^\pp)$ in (\ref{eq:68}).

From (\ref{eq:75}), Proposition~\ref{P.4},   and Lemmas~\ref{lem:Phi-bije-p1mod4}
and \ref{lem:end-ring-p3mod4}, we have   
\begin{equation}
  \label{eq:61}
h^\pp(\sqrt{p})\coloneqq \abs{\ppav(\sqrt{p})}=
  \begin{cases}
        h^1(\calO_1) \qquad &\text{if } p=2;\\
    h^1(\calO_1)+h^1(\calO_{16}) \qquad &\text{if } p\equiv 1\pmod{4};\\
    h^1(\bbO_0)    \qquad &\text{if } p\equiv 3\pmod{4}.
  \end{cases}
\end{equation}
Here $\calO_1\subset D$ is an arbitrary maximal $O_F$-order,
$\calO_{16}\subset D$ is an arbitrary $A$-order satisfying
(\ref{eq:31}), and $\bbO_0$ is the maximal $O_F$-order in  (\ref{eq:42}).

First, suppose that $p=2$. From Lemma~\ref{lem:Phi-bije-p1mod4}
and \cite[(4.7)]{xue-yu:type_no}, we obtain
\begin{equation}
  \label{eq:76}
  h^\pp(\sqrt{2})=t^\pp(\sqrt{2})=h^1(\calO_1)=t(\calO_1)=1.
\end{equation}

Next, suppose that $p\equiv 1\pmod{4}$. If $p=5$ then it follows from
Lemma~\ref{lem:Phi-bije-p1mod4} and
\cite[(4.7) and (4.8)]{xue-yu:type_no} that 
\begin{align}
\label{eq:99}
\abs{\Lambda_1^\pp}&=t(\Lambda_1^\pp)=h^1(\calO_1)=t(\calO_1)=1,\\ 
  \abs{\Lambda_{16}^\pp}&=t(\Lambda_{16}^\pp)=h^1(\calO_{16})=t(\calO_{16})=1. \label{zeq:104}
\end{align}
Similarly, if $p\geq 13$ and $p\equiv 1\pmod{4}$, then according to
\cite[(4.10) and (4.12)]{xue-yu:type_no}: 
\begin{align}
  \label{eq:77}
\abs{\Lambda_1^\pp}&=t(\Lambda_1^\pp)=h^1(\calO_1)=t(\calO_1)=\frac{\zeta_F(-1)}{2}+\frac{h(-p)}{8}+\frac{h(-3p)}{6},\\ \
  \begin{split}\label{eq:92}
\abs{\Lambda_{16}^\pp}&=t(\Lambda_{16}^\pp)=h^1(\calO_{16})=t(\calO_{16})\\&=\left(4-\Lsymb{2}{p}\right)\zeta_F(-1)+\frac{h(-p)}{4}+\left(2+\Lsymb{2}{p}\right)\frac{h(-3p)}{6}.    
  \end{split}
\end{align}
The
formula for $h^\pp(\sqrt{p})$, which is necessarily identical to that
of $t^\pp(\sqrt{p})$ in this case, is  obtained by summing up  the formulas for $\Lambda_1^\pp$ and $\Lambda_{16}^\pp$.

Lastly, according to
Proposition~\ref{prop:p=3mod4}, we have
\[h^\pp(\sqrt{3})=h^1(\bbO_0)=1,\qquad
  t^\pp(\sqrt{3})=\abs{\Tp^+(D)}=1\qquad \text{if}\quad p=3.\]
Moreover, if $p\geq 7$ and $p\equiv 3\pmod{4}$, then 
\begin{align*}
  h^\pp(\sqrt{p})&=h^1(\bbO_0)=
  \frac{\zeta_F(-1)}{2}+\left(11-3\Lsymb{2}{p}\right)\frac{h(-p)}{8}+\frac{h(-3p)}{6}\\
    t^\pp(\sqrt{p})&=\abs{\Tp^+(D)}=\frac{\zeta_F(-1)}{4}+\left(17-\Lsymb{2}{p}\right)\frac{h(-p)}{16}+\frac{h(-2p)}{8}+\frac{h(-3p)}{12}. 
\end{align*}
The theorems are proved. 
\end{proof}

\section{$(P,P_+)$-polarized superspecial abelian surfaces}
\label{sec:PP}
\numberwithin{thmcounter}{section}

In this section, we study in details the abelian surfaces in the simple $\F_p$-isogeny
class corresponding to $\pi=\sqrt{p}$ equipped with the \emph{polarization modules}
(\cite{Rapoport-thesis, Deligne-Pappas-1994}, \cite[\S
X.1]{van-der-Geer-HMS}). This allows us to generalize
the results of \S\ref{sec:computing-ppav}--\ref{sec:S4} to
nonprincipally polarized abelian surfaces. 



\subsection{$(P,P_+)$-polarized abelian varieties}
\label{sec:PP.1}
For the moment  let $F$ be a totally real number field, and $A$ be a
$\Z$-order in $F$.  Let $P$ be a finitely generated torsion free
$A$-module of rank one, i.e.~$P\otimes_A F\simeq F$ so that $P$ is
isomorphic to a fractional $A$-ideal in $F$. We say $P$ is a
\emph{proper} $A$-module if $\End_A(P)=A$, i.e.~$A=\{a\in
F\mid aP\subseteq P\}$. 
There is a canonical $\R$-algebra isomorphism $F\otimes_\Q\R\simeq
\R^{[F:\Q]}$, so $\R^{[F:\Q]}$ is naturally a free  $F\otimes_\Q
\R$-module of rank one. 
A \emph{notion of positivity} on 
 $P$ means an $A\otimes_\Z\R$-isomorphism $P\otimes_\Z \R\simeq
\R^{[F:\Q]}$, and we denote by $P_+$ the pre-image of the set of totally positive
elements $\R_{+}^{[F:\Q]}$.  Clearly, $P_+$ is closed under $A_+$-linear combinations,
where $A_+\subset F_+$ denote the subsets of totally positive
elements in $A$ and $F$ respectively.  Denote by 
\[ 
\begin{split}
\Pic_+(A)= & \text{the set of isomorphism classes
of \emph{invertible} $A$-modules with} \\
&\text{a notion of
positivity.}  
\end{split}
\]
Each invertible fractional $A$-ideal $\gra$ has a 
canonical notion of positivity from the $\R$-algebra isomorphism
$F\otimes_\Q\R\simeq \R^{[F:\Q]}$.  The map $\gra \mapsto (\gra, \gra_+)$
induces a bijection between the narrow class group of $A$ and $\Pic_+(A)$, which equips
$\Pic_+(A)$ with a canonical abelian group  structure. From now on we make no
distinction between the narrow class group of $A$ and $\Pic_+(A)$ in
the above sense.


Fix a base field $k$. By definition, an $F$-abelian variety over
$k$ is a pair $(X,\iota)$, 
where $X$ is an abelian variety over $k$
and $\iota:F\to \End^0(X)$ is a ring
homomorphism. We shall assume that $(X,\iota)$ satisfies the condition
$\dim X=[F:\Q]$. 
For an $F$-abelian variety 
$\ul X=(X,\iota)$, we put
$A\coloneqq \iota^{-1}(\End(X))\subset F$ and define
\begin{equation}
  \label{eq:51}
  \begin{split}
   \calP(\ul X)&\coloneqq  \Hom_{A}(X,X^t)^{\rm sym}=\text{the N\'eron-Severi
  group of $\ul X=(X,\iota)$,}\\ 
  \calP_+(\ul X)& \coloneqq \text{the subset of $A$-linear
  polarizations on $X$ (in $\calP(\ul X)$)}.
  \end{split}
\end{equation}
By \cite[Propositions 1.12 and 1.18]{Rapoport-thesis} 
that $\calP(\ul X)$ is a finite torsion-free $A$-module of rank one
equipped with notion of positivity $\calP_+(\ul X)$. The pair
$(\calP(\ul X),\calP_+(\ul X))$ is called the ($F$-linear)
\emph{polarization module} of $\ul X$.  In general, $\calP(\ul X)$ need not
to be a proper $A$-module nor a projective $A$-module.  Nevertheless,
any rank one proper module over a Gorenstein order is projective by
\cite[Characterization B 4.2]{Gorenstein-orders-JPAA-2015}. In
particular, if $A$ is Bass 
(see \cite[\S37]{curtis-reiner:1} and \cite{Levy-Wiegand-1985} for the definition and properties of Bass orders), i.e.~any order in $F$ containing $A$ is
Gorenstein, then $\calP(\ul X)$ is a projective $\End_A(\calP(\ul X))$-module. 

In general, let $R$ be an order in $F$, and $(P, P_+)$ be an
invertible $R$-module with a notion of positivity. 
A \emph{$(P,P_+)$-polarized $F$-abelian variety} over $k$ 
is a triple $(X,\iota,\xi)$, where 
\begin{itemize}
\item $\ul X=(X,\iota)$ is an $F$-abelian variety over $k$ with $\dim X=[F:\Q]$, and
\item $\xi:(P,P_+)\isoto (\calP(\ul X),\calP_+(\ul X))$ is an
  isomorphism of $R$-modules with notion of positivity.
\end{itemize}
The isomorphism $\xi$  will be called a
$(P, P_+)$-\emph{parametrization} of $(\calP(X),\calP_+(X))$. Two
$(P, P_+)$-polarized abelian varieties  $(X_i, \iota_i, \xi_i)/k$ for $i=1,2$  are isomorphic if there
exists a $k$-isomorphism $\alpha: X_1\to X_2$ such that
$\alpha\iota_1(a)=\iota_2(a)\alpha$ for every $a\in F$ and 
$\alpha^*\xi_2(b)=\xi_1(b)$ for every $b\in P$.

\subsection{$(P,P_+)$-polarized superspecial abelian surfaces.}
\label{sec:PP.2}

Let $F=\Q(\sqrt{p})$, $A=\Z[\sqrt{p}]$ and $D=D_{\infty_1, \infty_2}$.  For any abelian surface $X/\F_p$ with
$\pi_X^2=p$, there is a canonical embedding $\iota:
F\to \End^0(X)$ sending $\sqrt{p}$ to $\pi_X$ making $(X,
\iota)$ an $F$-abelian variety.  For simplicity, 
we omit $\iota$  from the notation. 
The polarization module $(\calP(X),\calP_+(X))$ of $X$ is defined as in
(\ref{eq:51}).  In the present case, $\calP(X)$ coincides with the
full 
N\'eron-Severi group of $X/\F_p$, and $\calP_+(X)$ coincides with the set
of polarizations.
Since any quadratic $\Z$-order is Bass \cite[\S2.3]{Levy-Wiegand-1985}, 
\[ \ul \calP (X)\coloneqq (\calP(X),\calP_+(X)) \]
represents an element in $\Pic_+(R)$
with $R=A$ or $O_F$. 
The association $X\mapsto  \ul\calP(X)$ induces a map
\begin{equation}
  \label{eq:V.4}
 \ul \calP: \Isog(\sqrt{p}) \to  \Pic_+(A)\amalg \Pic_+(O_F).    
\end{equation}
When $p\not \equiv 1 \pmod 4$, the map $\ul \calP$ sends
$\Isog(\sqrt{p})$ to $\Pic_+(O_F)$. When $p\equiv 1 \pmod 4$, the
set $\Isog(\sqrt{p})$ is the union
$\Lambda^\un_1 \amalg \Lambda^\un_{8} \amalg \Lambda^\un_{16}$. We claim that  \[ \ul \calP(\Lambda^\un_1)\subseteq \Pic_+(O_F),\quad
\ul \calP(\Lambda^\un_8)\subseteq \Pic_+(O_F),\quad \text{and}\quad 
\ul \calP(\Lambda^\un_{16})\subseteq \Pic_+(A). \] 
  Indeed, the first inclusion is obvious, and the middle one follows directly from  (\ref{eq:36}), so only the last one needs a proof.  Nevertheless, for later applications, we treat the cases $\Lambda_1^\un$ (for all $p$) and $\Lambda_{16}^\un$ (for $p\equiv 1\pmod{4}$) uniformly. Let $R=O_F$ if $r=1$ and $R=A$ if $r=16$.  The
Tate module $T_\ell(X)$ of any member $[X]\in \Lambda_r^\pp$ with $r\in \{1, 16\}$ is a \emph{free} $R_\ell$-module
of rank $2$  for every prime $\ell$ (including $\ell=p$,
see~\S\ref{sec:dieu-prime-p}). Let
$\Gamma_{\F_p}=\Gal(\bar{\F}_p/\F_p)$ and
$\calP(X)_\ell=\calP(X)\otimes \Z_\ell$. We have
\begin{equation}
  \label{eq:54}
  \begin{split}
  \calP(X)_\ell& =\{\,\lambda\in \Hom_{\Z_\ell[\Gamma_{\F_p}]}(T_\ell(X),T_\ell(X^t))\,|\, \lambda^t=-\lambda\,\} \\ 
  &=\left\{\psi_\ell\in \Alt_{\Z_\ell} (
  T_\ell(X)\otimes T_\ell(X), \Z_\ell)\middle\vert
  \begin{aligned}
&\psi_\ell(ax, y)=\psi_\ell(x,ay)\\
&\forall a\in R_\ell \text{ and } x, y\in T_\ell(X)  
  \end{aligned} \right\}\\
  & =\Hom_{R_\ell} (\wedge^2_{R_\ell} T_\ell(X), R^\vee_\ell) \\& =(\wedge^2_{R_\ell} T_\ell(X))^*\otimes_{R_\ell}  R^\vee_\ell, 
  \end{split}  
\end{equation}
where ${R}^\vee_\ell\subset F_\ell$ is the dual lattice of $R_\ell$
with respect to the  trace $\Tr_{F/\Q}$, and $(\cdot)^*$ denotes the
$R_\ell$-linear dual. Since $R$ is Gorenstein, $R_\ell^\vee$ is a free $R_\ell$-module of rank one by \cite[Proposition~3.5]{Gorenstein-orders-JPAA-2015} (which also follows immediately from direct calculations).  Therefore, $\calP(X)$ is an invertible $R$-module, and our claim is verified.


From now on, whenever we write $r\in \{1, 8, 16\}$ and $[P, P_+]\in
\Pic_+(R_r)$, we assume that one of the following conditions holds:
\begin{equation}
  \label{eq:513}
  \begin{split}
  &(a)\ r=1 \ \text{ and }\ R_1=O_F, \\
  &(b)\ r=8,\  p\equiv 1\!\pmod 4,  \text{ and }\ R_8=O_F
  \text{\ or} \\   
  &(c)\  r=16,\ p\equiv 1\!\pmod 4, \text{ and }\ R_{16}=A.
  \end{split}
\end{equation}
Observe that for   all $r\in \{1,8, 16\}$ and  $\calO_r$ in \S\ref{sect:end-ring}, we have
\begin{equation}
  \label{eq:116}
  \Nr(\wcO_r^\times)=\wh{R}_r^\times.
\end{equation}
Let us define
\begin{equation}
  \label{eq:80}
 \Lambda_r^\un[P,P_+]\coloneqq \{ [X]\in  \Lambda_r^\un\mid \exists \,\xi: (P,P_+)\xrightarrow{\simeq}
(\calP(X),\calP_+(X))\}. 
\end{equation}
 In other words, $\Lambda_r^\un[P,P_+]$ is the
fiber of the element $[P,P_+]\in \Pic_+(R_r)$ under the map $\ul \calP:
\Lambda_r^\un\to \Pic_+(R_r)$.
Fix a representative $(P, P_+)$ for the class $[P, P_+]\in
\Pic_+(R_r)$.  Let $\Lambda_r^\pmm(P,P_+)$ denote\footnote{Here we use
the notation $\Lambda_r^\pmm(P,P_+)$ instead of
$\Lambda_r^\pmm[P,P_+]$ since the parametrization $\xi:(P,P_+)\isoto
(\calP(\ul X),\calP_+(\ul X))$ depends on $(P, P_+)$ itself,
not just on its isomorphism class $[P, P_+]\in \Pic_+(R_r)$. If $\tau: (P, P_+)\to (P', P_+')$  is an isomorphism of invertible
$R_r$-modules with notion of positivity, then $\tau$ induces a
bijection $\Lambda_r^\pmm(P,P_+)\to \Lambda_r^\pmm(P',P'_+)$ sending
each $[X, \xi]$ to $[X, \xi\tau^{-1}]$. 
} the set of
isomorphism classes of $(P,P_+)$-polarized abelian surfaces
 $(X,\xi)/\F_p$ with $[X]\in \Lambda_r^\un[P, P_+]$.
There is  a surjective forgetful map 
\begin{equation}
  \label{eq:511}
  f: \Lambda_r^\pmm(P,P_+)\twoheadrightarrow \Lambda_r^\un[P,P_+], \quad [X,\xi]\mapsto [X].
\end{equation}



Our goal is to give adelic double coset descriptions for both
$\Lambda_r^\pmm(P,P_+)$ and $\Lambda^\un_r[P,P_+]$ and compute their
class numbers and type numbers for each $r\in \{1, 8, 16\}$ and each $[P, P_+]\in \Pic_+(R_r)$.
%



Fix an abelian surface $X_0/\F_p$  with $\pi_{X_0}^2=p$. Recall that $\Qisog(X_0)$ denotes the set of
equivalence classes of $\Q$-isogenies to $X_0$ as in  \S\ref{sec:method-cal}. 
For each member $(X, \varphi)\in \Qisog(X_0)$, we
realize $\calP(X)$ as an $A$-submodule of
$\calP^0(X_0)\coloneqq \calP(X_0)\otimes \Q$ by pushing forward along
$\varphi$: 
\begin{equation}
\varphi_*\lambda\coloneqq (\varphi^t)^{-1} \lambda\varphi^{-1}\in
\calP^0(X_0), \qquad \forall \lambda\in \calP(X). 
\end{equation}
Since push-forwards and pull-backs preserve $\Q$-polarizations,
\begin{equation}
  \label{eq:81}
\varphi_*\calP_+(X)=  \varphi_*\calP(X)\cap \calP_+^0(X_0),  
\end{equation}
where $\calP_+^0(X_0)$ denotes the subset of $\calP^0(X_0)$
consisting of all $\Q$-polarizations (see \S\ref{sect:neron-severi}). 
For each order  $R\supseteq A$ in $F$,  the finite idele group $\whF^\times$ acts transitively
on the set of invertible $R$-modules inside $\calP^0(X_0)$  by 
\[ a\cdot P\coloneqq \calP^0(X_0)\cap \prod_\ell a_\ell P_\ell, \quad \forall a=(a_\ell)_\ell\in
  \whF^\times. \]
Each such invertible $R$-module is equipped with the canonical notion of
positivity induced from $\calP^0_+(X_0)$.




  

Let $G$ (resp.~$G^1$) be the algebraic $\Q$-groups $\ul D^\times$
(resp.~$\ul D^1$) as before. Recall that there is an action of $\whD^\times=G(\whQ)$ on the set
$\Qisog(X_0)$, whose orbit containing the base point $(X_0, \id_0)$ is
denoted by $\scrG_G(X_0, \id_0)$ (or $\scrG_G(X_0)$ for brevity) and called the $G$-genus of $(X_0, \id_0)$ in $\Qisog(X_0)$ (see
Definition~\ref{defn:adelic-action-on-Qisog}).  The set of isomorphism
classes of underlying abelian varieties for $\scrG_G(X_0)$ is denoted by
$\Lambda^\un(X_0)$. Thus $\Lambda^\un(X_0)=\Lambda_1^\un$ if
$p\not\equiv 1\pmod{4}$, and $\Lambda^\un(X_0)=\Lambda_r^\un$ for some
$r\in \{1, 8, 16\}$ if $p\equiv 1\pmod{4}$.

\begin{prop}\label{prop:adel-polar-module}
Let $(X, \varphi)\in \Qisog(X_0)$ be
a member in the same 
$G$-genus of $(X_0, \id_0)$ 
so that there exists $g\in \whD^\times$ such that $(X,
\varphi)=g(X_0, \id_0)$. Then
$\varphi_*\calP(X)=\Nr(g)^{-1}\calP(X_0)$.  In particular, there is a
commutative diagram
\begin{equation}
  \label{eq:510}
  \begin{tikzcd}
  D^\times\backslash \whD^\times/\wh \calO_0^\times\ar[r, leftrightarrow, "\simeq" ]\ar[d, "\Nr^{-1}"]  &
  \Lambda^\un(X_0)\ar[d, "\ul\calP"] \\
  F^\times_+\backslash \whF^\times/\Nr(\wh \calO_0^\times)\ar[r,
  leftrightarrow, "\simeq"] &
  \Pic_{+}(R)
  \end{tikzcd}
\end{equation}
where $\calO_0\coloneqq \End(X_0)$ and
$R\coloneqq \End_{A}(\calP(X_0))$. Moreover,  the lower horizontal
bijection\footnote{It should be emphasized that this bijection is not
  the canonical group isomorphism as in \eqref{eq:14}. Rather it is
  the composition of this group isomorphism together with
  the translation by the narrow class of $\ul\calP(X_0)$. } 
is induced from sending each $a\in \whF^\times$ to the  $R$-module
$a\cdot \calP(X_0)$ with the induced notion of positivity  from $\calP^0_+(X_0)$. 
\end{prop}

\begin{proof}
First, suppose that $[X_0]\in \Lambda_r^\un$ with $r\in \{1,
16\}$. Correspondingly,  $R=O_F$ if $r=1$ and $R=A$ if
 $r=16$. Write $(X, \varphi)=g(X_0, \id_0)$ for $g=(g_\ell)_\ell\in
 \whD^\times$.  From Definition~\ref{defn:adelic-action-on-Qisog},
 $\varphi_*T_\ell(X)=g_\ell T_\ell(X_0)$ for every prime $\ell$. 
Plugging
in  (\ref{eq:54}), we obtain
\begin{equation}
  \label{eq:58}
  \begin{split}
   \varphi_* \calP(X)_\ell
   &=(\wedge^2_{R_\ell}\varphi_*T_\ell(X))^*\otimes_{R_\ell} 
     {R}^\vee_\ell \\
   &=(\wedge^2_{R_\ell} g_\ell T_\ell(X_0))^*\otimes_{R_\ell} 
     { R}^\vee_\ell \\
   &=\Nr(g_\ell)^{-1} (\wedge^2_{R_\ell}  T_\ell(X_0))^*
     \otimes_{R_\ell} { R}^\vee_\ell \\
   &=\Nr(g_\ell)^{-1} \calP(X_0)_\ell. 
  \end{split}
\end{equation}
This proves that $\varphi_*\calP(X)=\Nr(g)^{-1}\calP(X_0)$, and the
commutative diagram follows easily from \eqref{eq:116} and \eqref{eq:58}. 

Next we treat the case $[X_0]\in \Lambda_8^\un$, so $R=O_F$ in
this case.  For each
$(X, \varphi)\in \scrG_G(X_0)$, let $\eta: X\to \wt X$ be the minimal
isogeny constructed in Remark~\ref{rem:8-pol-mod}. There is a 
 $\whD^\times$-equivariant map:
\begin{equation}
  \label{eq:82}
 \scrG_G(X_0)\to \scrG_G(\wt X_0), \qquad (X, \varphi) \to (\wt X,
  \wt\varphi\coloneqq \eta_0\varphi\eta^{-1}). 
\end{equation}
It follows that the surjective map $\Lambda_8^\un\to \Lambda_1^\un$
sending each $[X]\mapsto [\wt X]$ fites into a commutative diagram 
\begin{equation}  \label{eq:83}
  \begin{tikzcd}
    D^\times\backslash \whD^\times/\wh \calO_0^\times\ar[r,
    leftrightarrow, "\simeq" ]\ar[d, twoheadrightarrow] &
    \Lambda_8^\un\ar[d, twoheadrightarrow] \\
     D^\times\backslash \whD^\times/\wh \scrO_0^\times\ar[r, leftrightarrow, "\simeq" ] & \Lambda_1^\un 
  \end{tikzcd}
\end{equation}
where $\scrO_0\coloneqq \End(\wt X_0)$, and the left vertical map is  the
canonical projection.   From (\ref{eq:36}), the map $\ul\calP:
\Lambda_8^\un\to \Pic_+(O_F)$ factors through $\Lambda_8^\un\to
\Lambda_1^\un$. Thus the proposition for the case $r=8$ is reduced to
the $r=1$ case. 
\end{proof}
\begin{cor}
  \begin{enumerate}
  \item If $p\not\equiv 1 \pmod 4$, then  $\ul \calP(\Isog(\sqrt{p}))=\Pic_+(O_F)$. 
  \item If $p\equiv 1 \pmod 4$, then
    \begin{equation}
      \label{eq:56}
      \ul \calP(\Lambda^\un_r)=
      \begin{cases}
        \Pic_+(O_F) & \text{for $r=1,8$;} \\
        \Pic_+(A) & \text{for $r=16$.}
      \end{cases}
    \end{equation}
   \end{enumerate}  
\end{cor}
\begin{proof}
  This follows directly from the commutative diagram (\ref{eq:510})
  since the left vertical map $\Nr^{-1}$ is surjective. 
\end{proof}

  \begin{cor}\label{cor:prin-polar-module}
  A member $[X]\in \Isog(\sqrt{p})$ 
  is  principal
   polarizable if and only if 
\[[X]\in \Lambda_1^\un[O_F, O_{F, +}]\amalg \Lambda_{16}^\un[A, A_+].\]
Moreover, there are canonical bijections:
  \begin{equation}
    \label{eq:V.5}
\Lambda_1^\pmm(O_F, O_{F,+})\cong     \Lambda_1^\pp, \qquad
     \Lambda_{16}^\pmm(A, A_+)\cong \Lambda_{16}^\pp. 
  \end{equation}
\end{cor}
\begin{proof}
  First, suppose that $X$ is equipped with a principal polarization
  $\lambda: X\to X^t$.  Necessarily,
  $[X]\in \Lambda_1^\un\cup \Lambda_{16}^\un$ by Proposition~\ref{P.4}.  The map
  \[\calP(X)\to R\coloneqq F\cap \End(X), \qquad \lambda'\mapsto \lambda^{-1}\lambda'\]
establishes an isomorphism of $R$-modules that identifies
$\calP_+(X)$ with $R_+$.

Conversely, suppose that $[X]\in \Lambda_1^\un\cup \Lambda_{16}^\un$
and $(\calP(X), \calP_+(X))$ represents the trivial class in
$\Pic_+(R)$. According to Proposition~\ref{P.4}, there exists a principally
polarizable $X_0$ in the
same genus of $X$. Let $\varphi: X\to X_0$ be a $\Q$-isogeny and pick
$g\in \whD^\times$ such that $(X, \varphi)=g(X_0, \id_0)$ in
$\Qisog(X_0)$. From Proposition~\ref{prop:adel-polar-module},
$(\calP(X), \calP_+(X))$ is isomorphic to  $\Nr(g)^{-1}(R, R_+)$,
so our assumption implies that $\Nr(g)\in
F_+^\times\Nr(\wh\calO_0^\times)$. Pick $\alpha\in D^\times$ and
$u\in \wh\calO_0^\times$ such that $\Nr(\alpha g u)=1$. Let $g'=\alpha
g u$ and $\varphi'=\alpha \varphi$. Then $g'\in \whD^1$ and
$(X, \varphi')=g'(X_0, \id_0)$. Now it follows from  Lemma~\ref{lem:genus-polarization}
that the pull-back of any principal polarization on $X_0$ along
$\varphi'$ is a principal polarization on $X$.

Lastly, let $[X, \xi]$ be a member of $\Lambda_r^\pmm(R, R_+)$ with
$r=1, 16$ and $R=O_F$ or $A$ accordingly. Then there exists
$u\in R_+$ such that $\xi(u)$ is a principal
polarization. Comparing the degrees on both sides of $\xi(u)=u\xi(1)$,
one immediately sees that $u\in R_+^\times$ and $\xi(1)$ is a
principal polarization as well. It follows the map
$[X, \xi]\mapsto [X, \xi(1)]$ establishes a bijection
$\Lambda_r^\pmm(R, R_+)\cong \Lambda_r^\pp$. 
\end{proof}


\begin{rem}
  Suppose that $p=2$. Then $\Pic_+(\Z[\sqrt{2}])$ is trivial, and
from \cite[Theorem~1.2]{xue-yang-yu:ECNF}, there is a unique abelian
surface $X/\F_2$ up to isomorphism satsifying
  $\pi_X^2=2$.  It admits a unique principal polarization up to
  isomorphism by \eqref{eq:76}. Moreover,  
  $\End(X)$ represents the unique type of
  maximal orders in the quaternion $\Q(\sqrt{2})$-algebra
  $D_{\infty_1, \infty_2}$. Therefore, for the rest of this section we
  assume that $p$ is an \emph{odd} prime (unless specified otherwise).  
\end{rem}

\begin{prop}\label{PP.6}
Let $p, r$ and $R_r$ be as  in (\ref{eq:513}), and
let $(P, P_+)$ be an invertible $R_r$-module with a notion of
positivity. 
  Fix
  a member $[X_r, \xi_r]\in \Lambda_r^\pmm(P, P_+)$ and put 
  $\calO_r\coloneqq \End_{\F_p}(X_r)$. Let $\calT_r[P, P_+]\subseteq \Tp(\calO_r)$ be the image of the
  map
  \begin{equation}\label{eq:84}
\Upsilon^\un: \Lambda_r^\un[P, P_+]\to \Tp(\calO_r)\qquad [X]\mapsto \dbr{\End(X)}.    
  \end{equation}
 There is a commutative diagram 
  \begin{equation}
    \label{eq:517}
    \begin{tikzcd}
       \Lambda_r^\pmm(P,P_+)\ar[r, twoheadrightarrow, "f"] \ar[d, leftrightarrow,
       "\simeq"]&  \Lambda_r^\un[P,P_+] \ar[d, leftrightarrow,
       "\simeq"]\ar[r, twoheadrightarrow, "\Upsilon^\un"]& \calT_r[P, P_+]\ar[d, leftrightarrow, "\simeq"]\\
D^1\backslash
       \whD^1/\wcO_r^1          \ar[r, twoheadrightarrow] & D^\times\backslash
    D^\times \whD^1\wcO_r^\times/\wcO_r^\times\ar[r, twoheadrightarrow] & D^\times\bsh D^\times\whD^1\calN(\wcO_r)/\calN(\wcO_r) 
    \end{tikzcd}
  \end{equation}
  where $f$ is the forgetful map in (\ref{eq:511}), and      the bottom two
horizontal maps are canonical projections. 
\end{prop}

\begin{proof}
From   (\ref{eq:510}),   $\Lambda_r^\un[P,P_+]$ can be identified
with the neutral fiber of the map 
\[D^\times \bsh \whD^\times/\wh \calO_r^\times\xrightarrow{\Nr^{-1}}
  F^\times_+\backslash \whF^\times/\Nr(\wh \calO_r^\times),\]
which shows that the middle vertical map is a bijection.  The
commutativity of right
square follows from that of (\ref{eq:38}), and bijectivity of the right vertical map
follows from the definition of $\calT_r[P, P_+]$.

To get the left vertical bijection, let
$\scrG_G^\pmm(X_r)$ be the subset of the genus
$\scrG_G(X_r)\subseteq \Qisog(X_r)$ consisting all members $(X,
\varphi:X\to X_r)$ such that 
$\varphi_*\calP(X)=\calP(X_r)$. 
Given $(X, \varphi)\in \scrG_G^\pmm(X_r)$, 
write $(X, \varphi)=g(X_r, \id_r)$
for some $g\in \whD^\times$, where $\id_r$ denotes the
identity map of $X_r$. Then $\Nr(g)\in \wh R_r^\times$ by
Proposition~\ref{prop:adel-polar-module}, so there exists $u\in \wh
\calO_r^\times$ such that $\Nr(g)=\Nr(u)$ by \eqref{eq:116}.  Put $g'=gu^{-1}$. Then
$g'\in \whD^1$ and $(X, \varphi)=g'(X_r, \id_r)$. It follows
that $\whD^1$ acts transitively on $\scrG_G^\pmm(X_r)$, that is, 
\[\scrG_G^\pmm(X_r)=\scrG_{G^1}(X_r)\simeq \whD^1/\wh\calO_r^1.\]

Each member $(X, \varphi)\in \scrG_G^\pmm(X_r)$ 
gives rise to a 
$(P, P_+)$-polarized abelian surface $(X,
\xi)$ by putting $\xi=\varphi^*\xi_r$, i.e.~$\xi(b)=\varphi^*\xi_r(b)$
for every $b\in P$. If $(X, \varphi)$ and $(X',
\varphi')$ in $ \scrG_G^\pmm(X_r)$  give rise to isomorphic pairs $(X, \xi)$ and $(X', \xi')$, then
there exists an isomorphism $\alpha: X\to X'$ such that
$\alpha^*\xi'(b)=\xi(b)$ for every $b\in P$.  Let $\beta=\varphi'\alpha\varphi^{-1}\in
D^\times$. Then $\beta(X, \varphi)=(X', \varphi')$ and
$\beta^*\lambda_r=\lambda_r$ for every $\lambda_r\in \calP(X_r)$. 
This implies that $\beta\in D^1$ since the Rosati involution
induced by every polarization $\lambda_r\in \calP_+(X_r)$ coincides
with the canonical involution on $D=\End^0(X_r)$. Conversely, if
$\beta(X, \varphi)=(X', \varphi')$ for some $\beta\in
D^1$, then it is  straightforward  to show that $(X,
\xi)$ and $(X', \xi')$ are isomorphic. This proves that there is an injective map
\begin{equation}
  \label{eq:V.7}
D^1\backslash
      \whD^1/\wcO_r^1\hookrightarrow   \Lambda_r^\pmm(P,P_+)
\end{equation}
sending the neutral class to $[X_r, \xi_r]$. 

Let $(X,\xi)$ be a $(P, P_+)$-polarized abelian surface with
$[X]\in \Lambda_r^\un$.  Fix $b_r\in P_+$ and put
$\lambda_r\coloneqq \xi_r(b_r)\in \calP_+(X_r)$ and
$\lambda\coloneqq \xi(b_r)\in \calP_+(X)$. By Lemma~\ref{P.2}, there exists a
quasi-isogeny $\varphi: X\to X_r$ such that
$\varphi^*\lambda_r=\lambda$. As any element $b\in P$ is of the form
$ab_r$ for some $a\in F$, we get
  \begin{equation}
    \label{eq:515}
    \begin{split}
    \varphi^* \xi_r(b)&=\varphi^t \xi_r(ab_r) \varphi
            = a \varphi^t \xi_r(b_r) \varphi 
      =a \varphi^* \xi_r(b_r) =a \xi(b_r)=\xi(b). \\     
    \end{split}
  \end{equation}
It follows that $(X, \varphi)\in \scrG_G^\pmm(X_r)$, so the map in
(\ref{eq:V.7}) is surjective as well.

The left square commutes because both  double coset descriptions of
$\Lambda_r^\pmm(P,P_+)$ and $\Lambda_r^\un[P, P_+]$ are induced from
the same $\whD^\times$-action on $\scrG_G(X_r)\subseteq \Qisog(X_r)$. 
\end{proof}

Let $\calO_r$ be as in Proposition~\ref{PP.6}. It is clear from (\ref{eq:517}) that
\begin{equation}
  \label{eq:90}
\calT_r[P, P_+]=\left\{\dbr{\calO}\mid \text{$\calO$
    is in the same spinor genus as $\calO_r$}\right\}\eqqcolon \Tp_\sg(\calO),  
\end{equation}
which generalizes 
Lemma~\ref{lem:spinor-genus}. 
  When $p\equiv
  1\pmod{4}$ and $r\in \{1, 16\}$,  we have seen in \eqref{xeq:26x} and
  \eqref{eq:x30y} that $\Tp(\calO_r)$ in fact forms a single spinor
  genus (i.e.~$\Tp(\calO_r)=\Tp_\sg(\calO_r)$), so $\calT_r[P, P_+]=\Tp(\calO_r)$ for $r\in \{1, 16\}$. A similar result holds
  when $r=8$ according to the following
  corollary. 
\begin{cor}\label{cor:p1mod4-P-plus-bij}
Suppose that $p\equiv
1\pmod{4}$.   Then the map
\begin{equation}
  \label{eq:93}
\Upsilon^\pmm=\Upsilon^\un f:\Lambda_r^\pmm(P, P_+)\xrightarrow{}
\Tp(\calO_r)  
\end{equation}
 is 
bijective if $r\in \{1, 16\}$, and it is surjective if $r=8$. In
particular, $f: \Lambda_r^\pmm(P, P_+)\to \Lambda_r^\un[P, P_+]$ is 
bijective for $r\in \{1, 16\}$, and $\calT_r[P, P_+]=\Tp(\calO_r)$ for all $r\in \{1, 8,
16\}$.  
\end{cor}
\begin{proof}
In light of Corollary~\ref{cor:prin-polar-module},  the present corollary for
the cases $r\in \{1, 16\}$ is just a
generalization of Lemma~\ref{lem:Phi-bije-p1mod4} and follow from  Lemma~\ref{lem:bije-eichler-p1mod4} or
   Lemma~\ref{lem:biject-O16} as before. 

Suppose that $r=8$. From (\ref{eq:41}),  $\calN(\wcO_8)=\whF^\times
\wcO_4^\times$, where $\calO_4$ is the Eichler order given in
(\ref{eq:C.3}). Interpreted adelically, the map $\Upsilon^\pmm: \Lambda_8^\pmm(P, P_+)\to 
\Tp(\calO_8)$ is the composition of the following canonical
projections:
\[D^1\backslash
      \whD^1/\wcO_8^1\twoheadrightarrow D^1\backslash
      \whD^1/\wcO_4^1\to D^\times\backslash
      \whD^\times/\whF^\times \wcO_4^\times.\]
The first map is obviously surjective, and the second map is bijective
by Lemma~\ref{lem:bije-eichler-p1mod4}. Thus the composition is
surjective. 
  \end{proof}


  If $p\equiv 3\pmod{4}$, then we have constructed a map $\Xi:
  \Tp(D)\to \grG(O_F)$ in \eqref{eq:11xi} from the type set $\Tp(D)$ of maximal orders in $D$ to
  the Gauss genus group $\grG(O_F)\coloneqq
  \Pic_+(O_F)/\Pic_+(O_F)^2$, whose fibers are precisely the distinct spinor
  genera of maximal orders in $D$. Since $\abs{\grG(O_F)}=2$ when
  $p\equiv 3\pmod{4}$, 
  $\Tp(D)$ accordingly decomposes into a disjoint union of two spinor genera
  $\Tp^+(D)\amalg \Tp^-(D)$, where the principal spinor genus $\Tp^+(D)$
  (i.e.~the neutral fiber) consists of the types of
  maximal orders that occur as endomorphism rings of principally
  polarizable members of $\Isog(\sqrt{p})$. From \eqref{eq:90}, 
$\calT_1[P, P_+]$ coincides with either $\Tp^+(D)$ or $\Tp^{-}(D)$. 

\begin{lem}\label{lem:p3mod4-spinor-genus}
  Suppose that $p\equiv 3\pmod{4}$.  The map $\Upsilon^\un:
\Lambda_1^\un[P, P_+]\to \calT_1[P, P_+]$ is  bijective for every $[P,
P_+]\in \Pic_+(O_F)$. Moreover,
\begin{equation}
  \label{eq:89}
  \calT_1[P, P_+]=
  \begin{cases}
    \Tp^+(D) &\text{if } [P, P_+]\in \Pic_+(O_F)^2;\\
    \Tp^{-}(D)&\text{otherwise.}
  \end{cases}
\end{equation}
\end{lem}
\begin{proof}
  In light of the surjectivity of
  $\Lambda_1^\un\twoheadrightarrow \Tp(D)$ in \eqref{eq:40} and the
  adelic description of the map
  $\Upsilon^\un:\Lambda_1^\un[P, P_+]\to \calT_1[P, P_+]$ in
  (\ref{eq:517}), to show that $\Upsilon^\un$ is bijective, it is
  equivalent to prove that the canonical projection map 
  \[D^\times\backslash D^\times
    \whD^1\widehat{\bbO}^\times/\widehat{\bbO}^\times\twoheadrightarrow
    D^\times\bsh D^\times\whD^1\calN(\widehat\bbO)/\calN(\widehat\bbO) \]
is bijective for every maximal $O_F$-order $\bbO \subset D$.
The proof will  be worked out in Lemma~\ref{lem:bijection-prin-type}. 

To prove (\ref{eq:89}), let us fix a principally polarizable abelian
surface $Y_1/\F_p$ with $\pi_{Y_1}^2=p$ and put
$\grO_1\coloneqq \End(Y_1)$ as in \eqref{eq:74}. Then $\dbr{\grO_1}\in \Tp^+(D)$ by definition, and 
$\ul\calP(Y_1)\simeq (O_F,
O_{F, +})$ by 
Corollary~\ref{cor:prin-polar-module}.  Combining (\ref{eq:38}) and 
(\ref{eq:510}), we obtain a diagram
\begin{equation}
  \label{eq:91}
  \begin{tikzcd}
    \Lambda_1^\un\ar[d,"\ul\calP"]\ar[r,leftrightarrow, "\simeq"]  &
    D^\times\bsh \whD^\times/\wh\grO_1^\times \ar[r, twoheadrightarrow] \ar[d, "\Nr^{-1}"]&
    D^\times\bsh \whD^\times/\calN(\wh\grO_1)\ar[d,
    "\Nr^{-1}"]\ar[r,leftrightarrow, "\simeq"]  & \Tp(D)\ar[d, "\Xi"]\\
    \Pic_+(O_F)\ar[r,leftrightarrow, "\simeq"] & F^\times_+\bsh\whF^\times/\whO_F^\times \ar[r,twoheadrightarrow]&
    F^\times_+\bsh\whF^\times/\whF^{\times 2}\whO_F^\times \ar[r,leftrightarrow, "\simeq"] & \grG(O_F)
  \end{tikzcd}
\end{equation}
Here the commutativity of the right square follows from the
construction of $\Xi$ in \eqref{eq:3xy}, \eqref{eq:11xi}, and the fact
that $\abs{\grG(O_F)}=2$.   The middle square is obviously commutative (Recall that ). 
The choice of $Y_1/\F_p$ guarantees that the left square is
commutative  as well. The composition of the  maps in the top row is none other
than 
$\Upsilon^\un: \Lambda_1^\un\to \Tp(D)$, and the composition for the
bottom row is just the canonical projection $\Pic_+(O_F)\to \grG(O_F)$.


By definition, $ \Lambda_1^\un[P, P_+]$ is the fiber of the map
$\ul\calP$ over $[P, P_+]\in \Pic_+(O_F)$, and $\calT_1[P,
P_+]$ is the image of $\Lambda_1^\un[P, P_+]$ under the map
$\Upsilon^\un$. From the commutativity of (\ref{eq:91}), we have
$\calT_1[P, P_+]\subseteq \Tp^+(D)$ if $[P, P_+]\in \Pic_+(O_F)^2$,
and $\calT_1[P, P_+]\subseteq \Tp^{-}(D)$ otherwise. The equalities now 
follow from \eqref{eq:90}. 
\end{proof}



\begin{rem}\label{rem:spinor-class}
    Let $\calO$ be an arbitrary $A$-order in $D$. Following
  \cite[\S1]{Brzezinski-Spinor-Class-gp-1983}, we say two locally principal right
$\calO$-ideals $I$ and $I'$ 
are in the \emph{same spinor class} if there exists $x \in D^\times\whD^1$
such that $\whI'=x \whI$. We denote by $\Cl_\scc(\calO)$ the set of locally principal right
$\mathcal{O}$-ideal classes within the spinor class of $\calO$ itself
(regarded as a principal right $\calO$-ideal), and put 
$h_\scc(\calO)\coloneqq \abs{\Cl_\scc(\calO)}$.
The set $\Cl_\scc(\calO)$ can be 
described adelically  as 
\begin{equation}
  \label{eq:86}
  \Cl_\scc(\calO)\simeq D^\times\bsh
  (D^\times\whD^1\wcO^\times)/\wcO^\times. 
\end{equation}
Let $\SCl(\calO)$ be the set
of spinor classes of locally principal right $\calO$-ideals. Similar
to \eqref{eq:3xy}, $\SCl(\calO)$ admits an adelic description as
follows:
\begin{equation}\label{zeq:14}
    \SCl(\calO)\simeq  (D^\times\whD^1)\bsh \whD^\times/\wcO^\times\xrightarrow[\simeq]{\Nr}
F_+^\times\bsh \whF^\times/\Nr(\wcO^\times).  
\end{equation}

Let $[X_r, \xi_r]\in \Lambda_r^\pmm(P, P_+)$ and 
$\calO_r=\End(X_r)$ be as in Proposition~\ref{PP.6}.  Combining
\eqref{zeq:14} with 
\eqref{eq:116}, we see that the reduced norm map induces an
isomorphism
\begin{equation}
  \label{zeq:57}
\SCl(\calO_r)\simeq
\Pic_+(R_r)  
\end{equation}
 sending the spinor class of $\calO_r$ to the identity of
$\Pic_+(R_r)$. 
On the other hand, if we combine \eqref{eq:86} with (\ref{eq:517}),
then we get a 
bijection
\begin{equation}
  \label{eq:87}
  \Lambda_r^\un[P, P_+]\simeq \Cl_\scc(\calO_r)
\end{equation}
sending the prior fixed class $[X_r]\in \Lambda_r^\un[P, P_+]$ to the principal right $\calO_r$-ideal
class $[\calO_r]$. This is a refinement of the bijection
$\Lambda_r^\un\simeq \Cl(\calO_r)$ in (\ref{eq:38}). See also
\cite[(2.11)]{xue-yu:spinor-class-no}.  Similarly, under the
identifications \eqref{eq:90} and \eqref{eq:87}, the map $\Upsilon^\un:
\Lambda_1^\un[P, P_+]\to \calT_1[P, P_+]$ can be identified with
\begin{equation}
  \label{zeq:13}
  \Cl_\scc(\calO_r)\to \Tp_\sg(\calO_r), \qquad [I]\mapsto
  \dbr{\calO_l(I)}, 
\end{equation}
which refines the map $\Upsilon: \Cl(\calO_r)\to \Tp(\calO_r)$ in
\eqref{eq:39}. 
\end{rem}

To state the main theorem of class number and type number formulas for
our  abelian surfaces equipped with polarization modules,   we need some extra notation. 
If $p\equiv 1\pmod{4}$, then $A=\Z[\sqrt{p}]$ is a suborder of index
$2$ in $O_F$. We define
\begin{equation}
  \label{eq:94}
  \varpi\coloneqq [O_F^\times:A^\times]. 
\end{equation}
According to \cite[\S 4.2]{xue-yang-yu:num_inv}, $\varpi\in \{1,3\}$,
and $\varpi=1$  if $p\equiv 1\pmod{8}$. Let $\delta_{3, \varpi}$ be the
  Kronecker $\delta$-symbol, which takes value $1$ if $\varpi=3$, and
  $0$ otherwise.    If $p\equiv 3\pmod{4}$, then thanks to
  Lemma~\ref{lem:p3mod4-spinor-genus}, our formulas are
  separated into two cases according to whether $[P, P_+]\in
  \Pic_+(O_F)^2$ or not.


\begin{thm}\label{thm:Pol-mod}
Let $p$, $r$ and $R_r$ with $r\in \{1, 8, 16\}$ be as  in (\ref{eq:513}). For
each $[P, P_+]\in \Pic_+(R_r)$,  
we list the formulas for following quantities 
\begin{equation*}
h_r^\pmm[P, P_+]:=\abs{\Lambda_r^\pmm(P, P_+)},\quad h_r^\un[P, P_+]:=\abs{\Lambda_r^\un[P, P_+]}, \quad t_r[P,
P_+]:=\abs{\calT_r[P, P_+]}. 
\end{equation*}

\begin{enumerate}
\item If $p=2, 3, 5$, then  
  \begin{equation}
    \label{zeq:98}
    h_1^\pmm[P, P_+]=h_1^\un[P, P_+]=t_1[P,
P_+]=1, \qquad \forall [P, P_+]\in \Pic_+(O_F). 
  \end{equation}
  
\item Suppose that $p\equiv 1\pmod{4}$. 
  Then 
  \begin{align}
    \begin{split}\label{zeq:101}
    h_1^\pmm[P, P_+]&=h_1^\un[P, P_+]=t_1[P,
                      P_+]\\
    &=
    \frac{\zeta_F(-1)}{2}+\frac{h(-p)}{8}+\frac{h(-3p)}{6}\qquad
    \text{for } p>5;
  \end{split}\\[2ex]
     h_8^\pmm[P,
    P_+]&=\frac{3}{2}\left(4-\Lsymb{2}{p}\right)\zeta_F(-1)+\left(2-\Lsymb{2}{p}\right)\frac{h(-p)}{8},\\
     h_8^\un[P, P_+]&=
                      \frac{3}{2\varpi}\left(4-\Lsymb{2}{p}\right)\zeta_F(-1)+\left(2-\Lsymb{2}{p}\right)\frac{h(-p)}{8\varpi}+\frac{\delta_{3,\varpi}}{\varpi}h(-3p), \label{zeq:106}
    \\
  \label{zeq:103}  t_8^\un[P,
    P_+]&=\frac{1}{2}\left(7+2\Lsymb{2}{p}\right)\zeta_F(-1)+\frac{h(-p)}{8}+\left(1-\Lsymb{2}{p}\right)\frac{h(-3p)}{6};\\[2ex]  
\begin{split}\label{zeq:102}
    h_{16}^\pmm[P, P_+]&=h_{16}^\un[P, P_+]=t_{16}[P,
                      P_+]\\
    &=\left(4-\Lsymb{2}{p}\right)\zeta_F(-1)+\frac{h(-p)}{4}+\left(2+\Lsymb{2}{p}\right)\frac{h(-3p)}{6}.
  \end{split}
  \end{align}

  \item Supppose that $p\equiv 3\pmod{4}$ and $p\geq 7$. 
    \begin{enumerate}[label=(\theenumi\alph*)]
    \item If $[P, P_+]\in \Pic_+(O_F)^2$, then
      \begin{align}
        h_1^\pmm[P,
        P_+]&=\frac{\zeta_F(-1)}{2}+\left(11-3\Lsymb{2}{p}\right)\frac{h(-p)}{8}+\frac{h(-3p)}{6},
        \label{zeq:107}\\
        h_1^\un[P, P_+]=t_1[P, P_+]&=\frac{\zeta_F(-1)}{4}+\left(17-\Lsymb{2}{p}\right)\frac{h(-p)}{16}+\frac{h(-2p)}{8}+\frac{h(-3p)}{12}. 
      \end{align}
    \item   If $[P, P_+]\not\in \Pic_+(O_F)^2$, then
      \begin{align}
                h_1^\pmm[P,
        P_+]&=\frac{\zeta_F(-1)}{2}+3\left(1-\Lsymb{2}{p}\right)\frac{h(-p)}{8}+\frac{h(-3p)}{6},\\
        h_1^\un[P, P_+]=t_1[P, P_+]&=\frac{\zeta_F(-1)}{4}+9\left(1-\Lsymb{2}{p}\right)\frac{h(-p)}{16}+\frac{h(-2p)}{8}+\frac{h(-3p)}{12}.\label{zeq:108} 
      \end{align}
    \end{enumerate}
\end{enumerate}
\end{thm}
\begin{proof}
  Let $[X_r, \xi_r]$ be an
  arbitrary member of $\Lambda_r^\pmm(P, P_+)$, and put
  $\calO_r:=\End_{\F_p}(X_r)$.  It follows from (\ref{eq:517}), 
  Remarks~\ref{rem:spinor-class}, and \eqref{eq:90} 
  that
  \begin{equation}
    \label{eq:85}
    h_r^\pmm[P, P_+]=h^1(\calO_r),\quad h_r^\un[P,
    P_+]=h_\scc(\calO_r), \quad t_r[P, P_+]=\abs{\Tp_\sg(\calO_r)}.
  \end{equation}
The following inequalities always hold true: 
  \begin{equation}
    \label{eq:100}
h_r^\pmm[P, P_+]\geq     h_r^\un[P, P_+]\geq t_r[P, P_+]. 
  \end{equation}
  Moreover, if $p\equiv 1\pmod{4}$ and $r\in \{1, 16\}$, all three
  quantities are equal by Corollary~\ref{cor:p1mod4-P-plus-bij}; if
  $p\equiv 3\pmod{4}$ (hence necessarily $r=1$), then
  $h_1^\un[P, P_+]=t_1[P, P_+]$ according to
  Lemma~\ref{lem:p3mod4-spinor-genus}. 

  Already, this explains why the formula in (\ref{zeq:101})
  (resp.~(\ref{zeq:102})) is identical to the one in (\ref{eq:77})
  (resp.~(\ref{eq:92})).  A priori, (\ref{eq:92}) is stated for primes
  $p\equiv 1\pmod{4}$ and $p\geq 13$. Nevertheless, 
  a direct calculation in the case $p=5$ (cf.~(\ref{zeq:104})) shows
  that  it holds for all primes $p\equiv
  1\pmod{4}$. 

  Next, suppose that $p\equiv 1\pmod{4}$ and $r=8$. The formula for
  $h_8^\pmm[P, P_+]=h^1(\calO_8)$ is calculated in
  Proposition~\ref{prop:class-number-O8}.  According to Corollary~\ref{cor:p1mod4-P-plus-bij}, $\calT_8[P, P_+]=\Tp(\calO_8)$. We find
  that $t_8[P, P_+]=t(\calO_8)$, which is calculated in
  \cite[(4.11)]{xue-yu:type_no}. Similar to the $r=16$ case, the
  formula for $t(\calO_8)$ in (\ref{zeq:103}) holds for all primes
  $p\equiv 1\pmod{4}$ (cf. \cite[(4.8)]{xue-yu:type_no}).  To
  calculate $h^\un[P, P_+]$, we claim that its value is independent of
  the choice of $[P, P_+]\in \Pic_+(O_F)$.  Indeed, the class number
  $h_\scc(\calO_8)$ depends only on the spinor genus of $\calO_8$, while the entire genus of
  $\calO_8$ forms a single spinor genus by combining \eqref{eq:90} and Corollary~\ref{cor:p1mod4-P-plus-bij}.
  This verifies the claim and shows that
  \begin{equation}
    \label{eq:105}
    h_8^\un[P, P_+]=\frac{\abs{\Lambda_8^\un}}{\abs{\Pic_+(O_F)}}=\frac{h(\calO_8)}{h_+(F)}=\frac{h(\calO_8)}{h(F)}.
  \end{equation}
 The class number formula for $\calO_8$ has already been
calculated in \cite[(6.8) and (6.10)]{xue-yang-yu:ECNF}. One then use
a result of Herglotz (see \cite[\S 2.10]{xue-yang-yu:num_inv}) to
factor $h(F)$ out  from $h(\calO_8)$ to obtain 
(\ref{zeq:106}).

Now suppose that $p\equiv 3\pmod{4}$. 
It is shown in Lemma~\ref{lem:p3mod4-spinor-genus} that $\calO_1$ belongs to the principal spinor genus of maximal
$O_F$-orders in $D$ if and only if $[P, P_+]\in \Pic_+(O_F)^2$. In particular, $t_1[P,
P_+]=\abs{\Tp^+(D)}$ if $[P, P_+]\in \Pic_+(O_F)^2$, and  $t_1[P,
P_+]=\abs{\Tp^{-}(D)}$ otherwise. The class number formulas for
$h^1(\calO_1)$ (for both the principal and nonprincipal cases) and type number formulas for $\abs{\Tp^+(D)}$
and $\abs{\Tp^{-}(D)}$ are  produced in Proposition~\ref{prop:p=3mod4}. This takes care of (\ref{zeq:98}) for $p=3$ and
the formulas (\ref{zeq:107})--(\ref{zeq:108}) for $p\equiv 3\pmod{4}$
and $p>7$. 


Lastly, if $p\in \{2, 5\}$, then $\Pic_+(O_F)$ consists a single member,
namely $[O_F,
O_{F, +}]$. In light of Corollary~\ref{cor:prin-polar-module},
(\ref{zeq:98}) is just a restatement of (\ref{eq:76}) (resp.~(\ref{eq:99})) in
the case $p=2$ (resp.~$p=5$).
\end{proof}

\begin{rem}\label{rem:arith-genus}
  We would like to exhibit some identities relating the following
three kinds of  objects:
  \begin{itemize}
  \item the arithmetic genera of
    some Hilbert modular surfaces;
    \item   the class numbers
      $h_1^\pmm[P, P_+]$, $h_1^\un[P, P_+]$ and type numbers
      $t_1[P, P_+]$;
     \item the proper class numbers of quaternary positive
       definite even quadratic lattices within certain genera. 
      \end{itemize}

Let $\GL^+_2(F)$ be the subgroup of $\GL_2(F)$ consisting of the elements with totally positive determinants, and $\PGL_2^+(F)$ be its canonical image in $\PGL_2(F)$. 
  For each nonzero fractional ideal $\gra$ of $O_F$, we define three arithmetic subgroups 
  \[ \Gamma(O_F\oplus \gra)\subseteq \wt \Gamma(O_F\oplus \gra) \subseteq \Gamma_m(O_F\oplus \gra) \]
  of $\PGL_2^+(F)$ as follows.
  We write $\SL(O_F\oplus \gra)$ for the stabilizer of $O_F\oplus \gra$
  in $\SL_2(F)$, and $\Gamma(O_F\oplus \gra)$ for its image in
  $\PGL_2^+(F)$. Let $K_\gra\subset \GL_2(\whF)$ be the stabilizer of the $\wh O_F$-lattice $\wh O_F \oplus \wh \gra$, and let $\wt \Gamma(O_F\oplus \gra)\subseteq \PGL_2^+(F)$ be the image of the subgroup $K_\gra\cap \GL_2^+(F)$.
  We refer the definition of the Hurwitz-Maass extension $\Gamma_m(O_F\oplus \gra)$ of $\Gamma(O_F\oplus \gra)$ to \cite[\S
  I.4]{van-der-Geer-HMS}. 
  It is shown that $\Gamma_m(O_F\oplus \gra)$ is the maximal discrete subgroup of $\PGL_2^+(\R)^{[F:\Q]}$ that contains $\wt \Gamma(O_F\oplus \gra)$; see \cite[p.~13]{van-der-Geer-HMS}.
  One can show that this is the same as the image of the subgroup $\GL_2^+(F)\cap \calN(K_\gra)$ in $\PGL_2^+(F)$, where $\calN(K_\gra)\subseteq \GL_2(\whF)$ is the normalizer of $K_\gra$. 
  Up to conjugation in $\PGL_2^+(F)$,
  these three groups depend only on the Gauss genus $\gamma\coloneqq \Pic_+(O_F)^2[\gra]_+\in
  \Pic_+(O_F)/\Pic_+(O_F)^2$ represented by $\gra$.    
  
  Let $\calH$ be
  the upper half plane of $\bbC$ as usual. 
  For $\Gamma=\Gamma(O_F\oplus \gra)$ or
  $\Gamma_m(O_F\oplus \gra)$, the Hilbert modular surface $Y_\Gamma$
  is defined to be the minimal
  non-singular model of the compactification of $\Gamma\bsh \calH^2$
  \cite[\S II.7]{van-der-Geer-HMS}.  The arithmetic genus
  $\chi(Y_\Gamma)$ is an important global invariant of $Y_\Gamma$ 
  (see \cite[\S VI]{van-der-Geer-HMS} and \cite[\S
  II.4--5]{Freitag-HMF}).  Since  both $Y_{\Gamma(O_F\oplus \gra)}$
  and $Y_{\Gamma_m(O_F\oplus \gra)}$ depend only on the Gauss genus
  $\gamma=\Pic_+(O_F)^2[\gra]_+$, we denote them as $Y(\grd_F,
  \gamma)$ and $Y_m(\grd_F, \gamma)$  respectively,  where $\grd_F$ is the discriminant of $\Q(\sqrt{p})$ as in
    Remark~\ref{rem:Bernoulli}.  Similarly, let us put
  \begin{equation}
    \label{eq:162}
  h_1^\pmm(\gamma)\coloneqq h_1^\pmm[\gra, \gra_+], \quad
  h_1^\un(\gamma)\coloneqq h_1^\un[\gra, \gra_+], \quad
  t_1(\gamma)=t_1[\gra, \gra_+], 
  \end{equation}
  where $\gra$ is equipped  with the canonical notion of
  positivity $\gra_+$ (see \S\ref{sec:PP.1}). Lastly, according to 
  \cite{chan-peters}, the genera of quaternary positive
       definite even quadratic lattices of discriminant $\grd_F$ can be labeled by the
       Gauss genus group $\Pic_+(O_F)/\Pic_+(O_F)^2$. Let $H^+(\grd_F,
       \gamma)$ denote the proper class number of such lattices
       in the genus labeled by $\gamma$ (cf.~\cite[\S3.2]{chan-peters}). 

       The formulas for $\chi(Y(\grd_F, \gamma))$ can be found in
       \cite[Theorems~II.5.8--9]{Freitag-HMF}, and those for
       $\chi(Y_m(\grd_F, \gamma))$ and $H^+(\grd_F, \gamma)$ can be found in
       \cite[\S1]{chan-peters}.  Comparing these formulas with the
       ones in Theorem~\ref{thm:Pol-mod}, we immediately obtain several
interesting       identities (\ref{eq:163})--(\ref{eq:165}) described below.

First supose that $p\not\equiv 3\pmod{4}$. Then there is a unique Gauss genus
$\gamma$, and hence a unique genus of quaternary positive
       definite even quadratic lattices of discriminant $\grd_F$. The two groups $\Gamma_m(O_F\oplus \gra)$ and
$\Gamma(O_F\oplus \gra)$ coincide by \cite[\S I.4,
p.~13]{van-der-Geer-HMS}, which implies that $Y(\grd_F,
\gamma)=Y_m(\grd_F, \gamma)$.  In this case we have 
\begin{equation}
  \label{eq:163}
\chi(Y(\grd_F, \gamma))=\chi(Y_m(\grd_F,
\gamma))=h_1^\pmm(\gamma)=h_1^\un(\gamma)=t_1(\gamma)=H^+(\grd_F, \gamma).
\end{equation}
The coincidence of the arithmetic genus $\chi(Y(\grd_F, \gamma))$ with
$H^+(\grd_F, \gamma)$ has been observed by many authors \cite{vigneras:inv,
  Kramer-1988Ann, chan-peters}. 

Now suppose that $p\equiv 3\pmod{4}$. There are two Gauss genera, so we denote the principal
Gauss genus by $\gamma^+$, and the nonprincipal one by $\gamma^-$. In
this case
\begin{equation}
\chi(Y(4p, \gamma^+))=h_1^\pmm(\gamma^-), \qquad  \chi(Y(4p, \gamma^-))=h_1^\pmm(\gamma^+).
\end{equation}
Moreover, 
\begin{align}
  \chi(Y_m(4p, \gamma^+))&=h_1^\un(\gamma^-)=t_1(\gamma^-)=H^+(4p,\gamma^+),\\
                                         \chi(Y_m(4p,
                                          \gamma^-))&=h_1^\un(\gamma^+)=t_1(\gamma^+)=H^+(4p,\gamma^-). \label{eq:165}
\end{align}
Once again, the coincidence between the arithmetic genus $\chi(Y_m(4p,
\gamma))$ and the proper class number $H^+(4p, \gamma)$ has been 
observed by Chan and Peters \cite{chan-peters}.




\end{rem}

\section{The class and type number formulas for maximal orders}
\label{sec:max-ord}
Throughout this section, $p\in \bbN$ denotes a prime number,
$F=\Q(\sqrt{p})$, and $D=D_{\infty_1, \infty_2}$ denotes the totally
definite quaternion $F$-algebra that splits at all finite primes of
$F$. All maximal $O_F$-orders in $D$ form a single genus. Let $\Tp(D)$
be the type set of maximal orders in $D$, and $\bbO\subset D$ be an
arbitrary maximal order. The goal of this section is to compute the
following class numbers
\begin{equation}
  \label{eq:3}
  h^1(\bbO)\coloneqq \abs{D^1\bsh \whD^1/\wh\bbO^1}, \qquad
  h_\scc(\bbO)\coloneqq \abs{\Cl_\scc(\bbO)}=\abs{D^\times\bsh
    \big(D^\times\whD^1\wh\bbO^\times\big)/\wh\bbO^\times}, 
\end{equation}
where $\whD^1$ denotes the reduced norm one subgroup of $\whD^\times$,
and $D^1, \wh\bbO^1$ are defined similarly. 
Moreover,  $\Cl_\scc(\bbO)$ denotes the  set of locally principal right
$\bbO$-ideal classes within the spinor class of principal ideal 
$\bbO$ itself as defined inside Remark~\ref{rem:spinor-class}.  
We first show that in our case the
class number $h_\scc(\bbO)$ coincides with the type number of maximal
orders within the
spinor genus of $\bbO$.


\begin{lem}\label{lem:bijection-prin-type}
Let $\Tp_\sg(\bbO)$ be the subset of $\Tp(D)$ consisting of types
of maximal orders in the spinor genus of $\bbO$ as in
\eqref{eq:90}.  Then the following map is bijective:
\begin{equation}
  \label{eq:236}
\Upsilon:   \Cl_\scc(\bbO)\to \Tp_\sg(\bbO), \quad [I]\mapsto
  \dbr{\calO_l(I)}. 
\end{equation}
\end{lem}
\begin{proof}
  Adelically, this map is given by the canonical projection 
  \begin{equation}
    \label{eq:237}
 D^\times\bsh \big(D^\times\whD^1
    \wh\bbO^\times\big)/\wh\bbO^\times\to  D^\times\bsh \big(D^\times\whD^1
    \calN(\wh\bbO)\big)/\calN(\wh\bbO),   
  \end{equation}
  where $\calN(\wh\bbO)$ denotes the normalizer of $\wh\bbO$ in
  $\whD^\times$. 
 Since $D$ splits at all the finite primes of $F$ and $\bbO$ is maximal, we
have $\calN(\wh\bbO)=\whF^\times\wh\bbO^\times$ as in \eqref{eq:41}.
Thus to show that the map
in (\ref{eq:237}) is bijective, it is enough to show that the following
$D^\times$-equivariant 
 surjective canonical projection map 
\begin{equation}
  \label{eq:238}
\big(D^\times\whD^1
    \wh\bbO^\times\big)/(F^\times\wh\bbO^\times)\twoheadrightarrow \big(D^\times\whD^1
    \calN(\wh\bbO)\big)/\calN(\wh\bbO)
  \end{equation}
  is injective as well.  Since $\whD^1$ is normal in
  $\whD^\times$ with an abelian quotient $\whD^\times/\whD^1\simeq \whF^\times$, both $D^\times\whD^1
    \wh\bbO^\times$ and $D^\times\whD^1
    \calN(\wh\bbO)$ are subgroups of $\whD^\times$. 
  
According to
\cite[Corollary~18.4]{Conner-Hurrelbrink}, $h(F)=h(\Q(\sqrt{p}))$ is
odd for every prime $p$, so
$F_+^\times\cap \whF^{\times2}\whO_F^\times=
F^{\times2}O_{F,+}^\times$, where $F_+^\times$ denotes the subgroup of
totally positive elements of $F^\times$, and $O_{F, +}^\times$ denotes the group
 of totally positive units. Indeed, given $\alpha\in F_+^\times\cap
\whF^{\times2}\whO_F^\times$, the principal ideal $\alpha O_F$ is a
perfect square, that is,  $\alpha O_F=\grb^2$ for some fractional
$O_F$-ideal $\grb\subseteq F$. Now it follows from the class number parity statement
above that $\grb$ is principal as well. Write $\grb=\beta O_F$ for some
$\beta\in F^\times$. Then $\alpha=\beta^2u$ for some  $u\in O_{F,
  +}^\times$.

Thus if $x\in D^\times\whD^1\wh\bbO^\times\cap
  \whF^\times\wh\bbO^\times$, then $\Nr(x)\in \big(
  F_+^\times\whO_F^\times\cap
  \whF^{\times2}\whO_F^\times\big)=F^{\times 2}\whO_F^\times$, and hence
  there exists $a\in F^\times$ such that $\Nr(a x)\in \whO_F^\times$.
On the other hand, any element 
$y\in \whF^\times\wh\bbO^\times$ with $\Nr(y)\in \whO_F^\times$
necessarily lies in $\wh\bbO^\times$.  It follows that 
\[ D^\times\whD^1\wh\bbO^\times\cap
  \whF^\times\wh\bbO^\times=F^\times\wh\bbO^\times, \]
which in turn implies the injectivity of (\ref{eq:238}). The lemma is proved. 
\end{proof}

As it turns out, if $p\not\equiv 3\pmod{4}$ then the computations of both
$h^1(\bbO)$ and $h_\scc(\bbO)$ are relatively easy because we have even more coincidences of
class numbers and type numbers. 
Recall that  $\zeta_F(s)$ denotes the Dedekind zeta function of the
quadratic real field $F=\Q(\sqrt{p})$, and  $h(d)$ denotes the class number of the quadratic field
$\Q(\sqrt{d})$ for each square-free integer
$d\in \bbZ$.

\begin{prop}\label{prop:CNF-p=1mod4}
  If $p\not\equiv 3\pmod{4}$, then
  \begin{equation}
    \label{eq:151}
    h^1(\bbO)=  h_\scc(\bbO)=t(\bbO)=h(\bbO)/h(F),  
\end{equation}
where $t(\bbO)\coloneqq\abs{\Tp(D)}$ denotes the type number of
$\bbO$, and $h(\bbO)\coloneqq \abs{\Cl(\bbO)}$ denotes the class
number of $\bbO$. In particular, 
\begin{gather}
    h^1(\bbO)=h_\scc(\bbO)=1 \qquad \text{if}\quad
  p\in \{2, 5\}, \quad\text{and }    \label{eq:152}\\
  h^1(\bbO)=h_\scc(\bbO)=\frac{\zeta_F(-1)}{2}+\frac{h(-p)}{8}+\frac{h(-3p)}{6}\quad 
  \text{if }  p\equiv 1(\bmod{4})\text{ and } p>5.   \label{eq:97}
\end{gather}
\end{prop}
\begin{proof}
  From Lemma~\ref{lem:bijection-prin-type},
  $h_\scc(\bbO)=\abs{\Tp_\sg(\bbO)}=t(\bbO)$, where the last equality
  follows from the fact that all maximal orders in $D$ belong to the
  same spinor genus in this case by
  Proposition~\ref{prop:spinor-gen-classif}. The equality
  $t(\bbO)=h(\bbO)/h(F)$ has been proved in
  \cite[Proposition~4.1]{xue-yu:type_no}, and we shall prove the
  equality $h^1(\bbO)=t(\bbO)$ in Lemma~\ref{lem:bije-eichler-p1mod4}.
The formulas in  \eqref{eq:152} and \eqref{eq:97} are reproduced from
the type number formulas in \cite[(4.5) and (4.8)]{xue-yu:type_no}. In
light of the class-type number relationship in (\ref{eq:151}), these formulas are not
new. See \S\ref{sect:end-ring} for a brief historical note. 
\end{proof}
  




The computations for $h^1(\bbO)$ and $h_\scc(\bbO)$ when
$p\equiv 3\pmod{4}$ are made difficult by the existence of \emph{two}  spinor
genera of maximal orders in $D$ according to
Proposition~\ref{prop:spinor-gen-classif}. A systematic study of these class
number formulas were carried out by the current authors in
\cite{xue-yu:spinor-class-no}.  For the moment let $\calF$ be a
totally real number field, $\calD$ be a totally definite quaternion
$\calF$-algebra, and $\calO\subset\calD$ be an Eichler order.  We  are
able to derive 
in \cite[Corollary~4.2]{xue-yu:spinor-class-no} a class number
formula for $h^1(\calO)\coloneqq\abs{\calD^1\bsh \wh\calD^1/\wcO^1}$ using the
Selberg trace formula. The proof of this formula is  greatly simplified by Voight
\cite[Appendix A]{xue-yu:spinor-class-no}, who gives an alternative
proof  by replacing calculations with
the Selberg trace formula with a direct, conceptual argument. A
formula for
$h_\scc(\calO)\coloneqq\abs{\calD^\times\bsh
  \big(\calD^\times\wh\calD^1\wcO^\times\big)/\wcO^\times}$ is also
obtained in \cite[Theorem~3.2]{xue-yu:spinor-class-no} by refining the
original proof of the classical Eichler class number formula
\cite[Corollaire V.2.5,
p. 144]{vigneras}\cite[Theorem~30.8.6]{voight-quat-book}.  However,
different from Eichler's formula, which depends purely on the
local datum of $\calO$ (namely, the genus of $\calO$), the formulas
for both $h^1(\calO)$ and $h_\scc(\calO)$ depend more subtly on the
spinor genus of $\calO$, which has a mixture of local and global
flavor; see Definition~\ref{defn:spinor-genus-class}.  In particular,
our formulas require certain new invariants arising from the theory of
\emph{optimal spinor selectivity}, which we recall briefly below.

\subsection{Optimal spinor selectivity}
\label{subsec:oss}
\numberwithin{thmcounter}{subsection}
The selectivity theory is first formulated by Chinburg and Friedman
\cite{Chinburg-Friedman-1999} as an refined integral version of the
Hasse-Brauer-Noether-Albert Theorem \cite[Theorem~III.3.8]{vigneras}, 
and has since been further developed by many people. See
\cite[\S31.7.7]{voight-quat-book} for a brief historical account.  We
shall follow closely the exposition in
\cite{Xue-Yu-Selec-2022}, which is most
suitable for the current purpose. 


 Let $\calF, \calD$ be as above, and $\grn\subseteq O_\calF$ be an integral ideal coprime to the
discriminant of $\calD$.   All Eichler orders of level $\grn$ in
$\calD$ form a single genus $\scrG_\grn$. 
 The spinor genus of $\calO$
is denoted by $\dbr{\calO}_\sg$. Let $\SG(\scrG_\grn)$ be the set of
spinor genera within $\scrG_\grn$. In other words,
$\SG(\scrG_\grn)=\{\dbr{\calO}_\sg \mid \calO\in \scrG_\grn\}$. 


Since $\calD$ is totally definite, a quadratic field extension
$\calK/\calF$ embeds into $\calD$ only if $\calK/\calF$ is a
CM-extension.  An $O_\calF$-order $B$ (of full rank) in a CM-extension
of $\calF$ will be called a \emph{CM $O_\calF$-order}.  Given a CM
$O_\calF$-order $B$ with fractional field $\calK$ and an Eichler order
$\calO\in \scrG_\grn$, we write
$\Emb(B, \calO)$ for the set of \emph{optimal embeddings} of $B$ into
$\calO$, that is
\begin{equation}
  \label{eq:1-22}
  \Emb(B, \calO):=\{\varphi\in \Hom_\calF(\calK, \calD)\mid \varphi(\calK)\cap \calO=\varphi(B)\}.
\end{equation}
Similarly, at each finite prime $\grp$ of $\calF$, we write
$\Emb(B_\grp, \calO_\grp)$ for the set of optimal embeddings of
$B_\grp$ into $\calO_\grp$.  Clearly, if $\Emb(B, \calO)\neq \emptyset$, then 
\begin{equation}\label{eq:local-emb-cond}
  \Emb(B_\grp,
\calO_\grp)\neq \emptyset \quad\text{for every finite prime $\grp$ of $\calF$}.
\end{equation}
Conversely, if condition \eqref{eq:local-emb-cond} holds\footnote{A concrete criterion for the existence of local
  optimal embeddings into Eichler orders is given by Guo and Qin in
  \cite[Lemma~2.2]{Guo-Qin-embedding-Eichler-JNT2004}, which in
  turn is based on a theorem of Brzezinski
  \cite[Theorem~1.8]{Brzezinski-embed-num-1991}. In our case,
  condition \eqref{eq:local-emb-cond} holds automatically  for every maximal $O_F$-order
  $\bbO\subseteq D_{\infty_1, \infty_2}$ and every $O_F$-order $B$
  in a CM-extension of $F=\Q(\sqrt{p})$ by 
\cite[Theoreme~II.3.2]{vigneras}. } for one and hence all orders in $\scrG_\grn$,  then there exists
an order $\calO_0\in \scrG_\grn$ such that
$\Emb(B, \calO_0)\neq \emptyset $ by
\cite[Corollary~30.4.18]{voight-quat-book}. However, in the totally
definite case, it does not make sense to expect $\Emb(B, \calO')\neq
\emptyset $ for every order $\calO'\in \scrG_\grn$; see
\cite[Example~2.1]{Xue-Yu-Selec-2022}. Instead, one asks whether
condition \eqref{eq:local-emb-cond} implies the existence of 
an order $\calO'$ in  every spinor genus of $\scrG_\grn$ such that $\Emb(B, \calO')\neq
\emptyset $.  It turns out that while the answer is positive for most
CM $O_\calF$-orders and Eichler orders, there exist some particular cases
where the answer is negative. 
The selectivity theory is developed to give a concrete
criterion for when the answer is negative.

Let us first define 
the \emph{optimal
    spinor selectivity symbol} as follows
\begin{equation}
  \label{eq:6}
  \Delta(B, \calO)=
  \begin{cases}
    1 \qquad &\text{if } \exists\, \calO'\in \dbr{\calO}_\sg\text{ such that }\Emb(B, \calO')\neq \emptyset, \\
    0 \qquad &\text{otherwise}.
  \end{cases}
  \end{equation}

\begin{defn}\label{defn:oss}
  Suppose that condition \eqref{eq:local-emb-cond} holds.  We say $B$ is
  \emph{optimally spinor selective} (selective for short) for the
  genus $\scrG_\grn$ if $\Delta(B, \calO)=0$ for some (but not
  all)
  $\dbr{\calO}_\sg\in \SG(\scrG_\grn)$. If $B$ is selective for
  $\scrG_\grn$, then a spinor genus $\dbr{\calO}_\sg$ with
  $\Delta(B, \calO)=1$ is said to be \emph{selected} by $B$.
\end{defn}

To state the main theorem of optimal spinor selectivity, we need  some more notation. 
For each finite prime $\grp$ of $O_\calF$, we write
$\nu_\grp: \calF^\times\twoheadrightarrow \Z$ for the normalized discrete
valuation of $\calF$ attached to $\grp$. Let $\grd(\calO)$ be the
\emph{reduced discriminant} of $\calO\in \scrG_\grn$, which is
 the product of $\grn$ with all the finite ramified primes of
$\calD/\calF$. Given a CM $O_\calF$-order $B$ with fractional field
$\calK$, we write $\grf(B)$ for the \emph{conductor} of $B$. In other
words, $\grf(B)$ is the unique integral ideal of $O_\calF$ such that
$B=O_\calF+\grf(B)O_\calK$.  If $B'$ is another $O_\calF$-order in
$\calK$, then we put $\grf(B'/B)\coloneqq \grf(B')^{-1}\grf(B)$ and
call it the \emph{relative conductor} of $B$ with respect to $B'$,
which is now a fractional ideal of $O_\calF$. 
The following theorem is a combination of  \cite[Theorem~2.15,
Lemma~2.17, and
Theorem~3.2]{Xue-Yu-Selec-2022}. 

\begin{thm}\label{thm:selectivity}
Let $\scrG_\grn$ be the genus of Eichler orders of level $\grn$ in the
totally definite quaternion $\calF$-algebra $\calD$, and $B$ be a CM
$O_\calF$-order with fractional field $\calK$. Suppose that condition
\eqref{eq:local-emb-cond} holds for $B$ and $\scrG_\grn$. Then
\begin{enumerate}[label=(\arabic*)]
\item $B$ is
optimally spinor selective for the genus $\scrG_\grn$ if and only if
both of the following conditions hold:
 \begin{enumerate}[label=(\alph*)]
  \item both $\calK$ and  $\calD$ are unramified
  at every finite prime $\grp$ of $\calF$;
  \item $\calK/\calF$ splits at every finite prime $\grp$ of $\calF$
    satisfying $\nu_\grp(\grd(\calO))\equiv
  1\pmod{2}$.  
\end{enumerate}
\item If $B$ is optimally spinor selective for $\scrG_\grn$, then
$\Delta(B, \calO)=1$ for   exactly half of the spinor genera
$\dbr{\calO}_\sg\in \SG(\scrG_\grn)$.
\item Suppose that $\calK$ and $\calD$ satisfy both of the conditions  in
  part (1) of the theorem so that $B$ is optimally spinor selective
  for $\scrG_\grn$. Let $B'$ be another $O_\calF$-order in $\calK$
  satisfying condition \eqref{eq:local-emb-cond}. Then for every $\calO\in \scrG_\grn$, we have 
  \begin{equation}
    \label{eq:11}
        \Delta(B, \calO)=\big(\grf(B'/B),
    \calK/\calF\big)+\Delta(B', \calO),  
  \end{equation}
where $\big(\grf(B'/B),
    \calK/\calF\big)\in \Gal(\calK/\calF)$ is the Artin symbol \cite[\S
    X.1]{Lang-ANT}, and  the 
  summation on the right hand side of \eqref{eq:11} is taken inside $\zmod{2}$ with the canonical
  identification $\Gal(\calK/\calF)=\zmod{2}$.
\end{enumerate}
\end{thm}

In fact, if $B$ is selective for the genus $\scrG_\grn$, then starting
from a known selectivity symbol $\Delta(B, \calO)$ for some
$\calO\in \scrG_\grn$, we can express any other $\Delta(B, \calO')$
with $\calO'\in \scrG_\grn$ in terms of $\Delta(B, \calO)$ and the
``relative position'' of $\dbr{\calO}_\sg$ and $\dbr{\calO'}_\sg$ in
$\SG(\scrG_\grn)$. See \cite[Theorem~2.15]{Xue-Yu-Selec-2022} or
\cite[\S31.1.9]{voight-quat-book} for more details.  In general, for
an arbitrary CM $O_\calF$-order $B$ with fractional field $\calK$, we
put
\begin{equation}\label{eq:142}
s(B, \calO):=  \begin{cases}
    1 & \text{if $\calK$,  $\calD$ satisfy  conditions (a) and (b) in
      Theorem~\ref{thm:selectivity},}\\
    0 & \text{otherwise.} 
  \end{cases}
\end{equation}
By definition, $s(B, \calO)$ depends only on  $\calK, \calD$ and
$\grn$.


Theorem~\ref{thm:selectivity} was first obtained by Maclachlan 
\cite{Maclachlan-selectivity-JNT2008} for Eichler orders of
square-free levels in indefinite quaternion algebras. Independently, Arenas et~al.\
\cite{M.Arenas-et.al-opt-embed-trees-JNT2018} and Voight
\cite[Chapter~31]{voight-quat-book} removed the square-free condition
and obtained the first two parts of the theorem for Eichler orders of
arbitrary levels. The generalization of \eqref{eq:11} to Eichler
orders of arbitrary levels are due to the current authors
\cite{Xue-Yu-Selec-2022}.    This concludes our brief account of the optimal spinor selectivity theory.

\subsection{The class number formulas for $h^1(\bbO)$ and $h_\scc(\bbO)$ when $p\equiv
   3\pmod{4}$}
We return to the original setup  where $F=\Q(\sqrt{p})$, and $\bbO$ is a maximal
$O_F$-order in the quaternion $F$-algebra $D=D_{\infty_1,
  \infty_2}$. Assume that $p\equiv 3\pmod{4}$ for the rest of this
section unless specified otherwise. 

Let $O_{F, +}^\times\coloneqq F_+^\times\cap O_F^\times$ be the
group of totally positive units of $F$, and $\varepsilon>1$ be the
fundamental unit of $O_F^\times$. By \cite[\S 11.5]{Alaca-Williams-intro-ANT}
or \cite[Corollary~18.4bis]{Conner-Hurrelbrink}, $\varepsilon$ is
totally positive (i.e.~$\Nm_{F/\Q}(\varepsilon)=1$) since 
$p\equiv 3\pmod{4}$. Hence $O_{F,+}^\times= \dangle{\varepsilon}$ while $O_F^{\times2}=\dangle{\varepsilon^2}$, from which it follows that
\begin{equation}
  \label{xeq:1}
  [O_{F,+}^\times:O_F^{\times 2}]=2. 
\end{equation}
According to \cite[Lemma~11.6]{Conner-Hurrelbrink}, the narrow class number
$h_+(O_F)$ is related to the (wide) class number $h(F)$ by
$h_+(F)=h(F)  [O_{F,+}^\times:O_F^{\times 2}]$, and hence 
\begin{equation}
  \label{eq:140}
h_+(F)  =2h(F).
\end{equation}
From \cite[Corollary~18.4]{Conner-Hurrelbrink}, $h(F)$ is
odd.



Given a CM $O_F$-order $B$, we write $\bmu(B)$ for the subgroup of roots
of unity of $B^\times$. 
Let $h(B)$ be the class
number of $B$, and put 
\begin{equation}
  \label{eq:57}
  w(B)\coloneqq [B^\times:
 O_F^\times]. 
\end{equation}
According to \cite[Remarks, p.~92]{Pizer1973}
(cf.~\cite[\S3.1]{li-xue-yu:unit-gp} and
\cite[\S3.3]{xue-yang-yu:ECNF}), there exist only finitely many CM
$O_F$-orders $B$ with $w(B)>1$, so we collect them into a finite
set $\scrB$. The CM $O_F$-orders $B$ with $\abs{\bmu(B)}>2$ form a
subset of $\scrB$ which will be denoted by $\scrB^1$. 


Applying  \cite[Corollary~4.2 and (3.13)]{xue-yu:spinor-class-no} to
the maximal order $\bbO\subseteq D$, we obtain
  \begin{equation}
    \label{xeq:214}
    h^1(\bbO)=\frac{1}{2}\zeta_F(-1)+\frac{1}{4h(F)}\sum_{B\in
      \scrB^1}2^{s(B, \bbO)}\Delta(B, \bbO)(\abs{\bmu(B)}-2)h(B)/w(B).
  \end{equation}
Similarly, combining \eqref{xeq:1}, \eqref{eq:140} with  \cite[Theorem~3.2 and
(3.8)]{xue-yu:spinor-class-no}, we obtain
\begin{equation}
  \label{eq:139}
  h_\scc(\bbO)=\frac{\zeta_F(-1)}{4}+ \frac{1}{4h(F)}\sum_{B\in
      \scrB}2^{s(B, \bbO)}\Delta(B, \bbO)\left(1-\frac{1}{w(B)}\right)h(B). 
\end{equation}


Note that the summations in \eqref{xeq:214} and \eqref{eq:139} range
over $\scrB^1$  and $\scrB$ respectively.  It remains to classify
these CM $O_F$-orders and work out the respective invariants. 
According to \cite[(4.4) and \S7.2]{li-xue-yu:unit-gp} or
\cite[\S2.8]{xue-yang-yu:num_inv}, if $p\geq 7$ and $B\in \scrB$, then
the fractional field $K\coloneqq\Frac(B)$ coincides with $F(\sqrt{-1})$, $F(\sqrt{-2})$
or $F(\sqrt{-3})$, where $F(\sqrt{-2})=F(\sqrt{-\varepsilon})$ by \cite[Proposition~2.6]{xue-yang-yu:num_inv}. Moreover, $K\in \{F(\sqrt{-1}), F(\sqrt{-3})\}$ if
 $B\in \scrB^1$.  
The prime $p=3$ will be treated
 separately. For simplicity, we put
\begin{equation}
K_m\coloneqq F(\sqrt{-m})\quad\text{for}\quad  m\in \{1,2, 3\}.   
\end{equation}
We make no use of $2$-adic or $3$-adic completions for the rest
of this section, so the notation should cause no confusion.  


By definition, the quaternion $F$-algebra $D=D_{\infty_1, \infty_2}$
is unramified at all the finite primes of $F$, so part of condition
(a) in Theorem~\ref{thm:selectivity} is already satisfied.  Moreover, condition
(b) is vacuous in this
case since the reduced discriminant $\grd(\bbO)=O_F$ as $\bbO$ is maximal. On the other hand,
recall from \cite[\S 15G]{Cohn-invitation-Class-Field}
(cf.~\cite[\S6]{Cox-Primes}) that the \emph{strict genus field}
$\Sigma/F$ is defined to be the maximal unramified\footnote{Here by ``unramified''
  we mean ``unramified at all the finite primes of $F$''. } elementary
abelian $2$-extension of $F$. The Artin reciprocity map induces an
isomorphism between the Gauss genus group
$\grG(O_F)=\Pic_+(O_F)/\Pic_+(O_F)^2$ and the Galois group of
$\Sigma/F$:
\begin{equation}
  \label{eq:145}
 \grG(O_F)\cong\Gal(\Sigma/F), \qquad   \Pic_+(O_F)^2[\gra]_+\mapsto
 (\gra, \Sigma/F). 
\end{equation}
From \eqref{xeq:26x}, $\abs{\grG(O_F)}=2$ since
$p\equiv 3\pmod{4}$. 
 One checks directly that 
$F(\sqrt{-1})/F$ is unramified for $p\equiv 3\pmod{4}$, so it
coincides with $\Sigma$ and is the unique unramified quadratic
extension of $F$. According to Theorem~\ref{thm:selectivity}, an CM
$O_F$-order $B\in \scrB$ is selective for the genus of maximal orders
of $D$
if and only if its fractional field is $K_1=F(\sqrt{-1})$, so we have 
\begin{equation}
  \label{eq:141}
s(B, \bbO)=
  \begin{cases}
   1  \qquad &\text{if } \Frac(B)=F(\sqrt{-1}),\\
0\qquad &\text{otherwise.}
  \end{cases}  
\end{equation}



To apply part (3) of Theorem~\ref{thm:selectivity}, we need a more
concrete description of the Artin reciprocity map in \eqref{eq:145}
for $\Sigma=K_1$. 
For every narrow ideal class in $\Pic_+(O_F)$, there is an integral
ideal $\gra\subseteq O_F$ coprime to $p$ representing this class.  If
we identify $\Gal(K_1/F)$ with the multiplicative group $\{\pm 1\}$,
then the canonical map $\Pic_+(O_F)\to \grG(O_F)\simeq \Gal(K_1/F)$ is
identified with the unique nontrivial \emph{genus character}
\cite[\S14G, p.~150]{Cohn-invitation-Class-Field}
\begin{equation}
  \label{eq:y3}
 \chi: \Pic_+(O_F)\to \{\pm 1\},
\end{equation}
which can be  expressed in
terms of the Kronecker symbol \cite[\S10D]{Cohn-invitation-Class-Field} as follows: 
\begin{equation}
  \chi([\gra]_+)=\Lsymb{-p}{\Nm(\gra)},\qquad \forall\, [\gra]_+\in
  \Pic_+(O_F)\quad \text{with} \gcd(\gra, p)=1. 
\end{equation}
For example, if $\grq$ is the unique dyadic prime of $F$, then 
\begin{equation}
  \label{eq:4}
(\grq, K_1/F)= \chi([\grq]_+)=\Lsymb{-p}{2}=\Lsymb{2}{p}=
  \begin{cases}
   1 &\text{if } p\equiv 7\pmod{8},    \\
   -1 &\text{if } p\equiv 3\pmod{8}.    
  \end{cases}
\end{equation}
This can also be re-interpreted as follows. By \cite[Lemma~3]{MR1344833} or
\cite[Lemma~3.2(1)]{MR3157781}, there exists $\vartheta\in F^\times$ such
that $\varepsilon=2\vartheta^2$, so $\grq=2\vartheta O_F$.
In particular, $[\grq]_+$ is a $2$-torsion in $\Pic_+(O_F)$, and it has order $2$ if and only if
$\Nm_{F/\Q}(\vartheta)<0$.  
From
(\ref{eq:4}), we have $\Nm_{F/\Q}(\vartheta)=\frac{1}{2}\Lsymb{2}{p}$
(cf.~the proof of \cite[Lemma~6.2.6]{li-xue-yu:unit-gp}).

From Proposition~\ref{prop:spinor-gen-classif}, 
 there are two spinor genera of maximal $O_F$-orders in
$D$ in this case,  which leads to a decomposition of the type set 
as follows: 
\begin{equation}\label{eq:149}
  \Tp(D)=\Tp^+(D)\amalg \Tp^{-}(D).   
\end{equation}
A maximal order $\bbO$ in
  $D$ is said to belong to the \emph{principal spinor genus} if
  $\dbr{\bbO}\in \Tp^+(D)$,   otherwise it is said to belong to the
  \emph{nonprincipal spinor genus}.
Following Example~\ref{ex:res-p3mod4}, if we present $D$ as
$\qalg{-1}{-1}{F}$ and write $\{1, i, j, k\}$ for its standard basis, then
$\Tp^+(D)=\Tp_\sg(\bbO_0)$, where 
\begin{equation}
  \label{eq:146}
  \bbO_0=O_F+O_Fi+O_F\frac{\sqrt{p}+j}{2}+O_F\frac{\sqrt{p}i+k}{2}\subset
\qalg{-1}{-1}{F}. 
\end{equation}

  \begin{prop}\label{prop:p=3mod4}
    Suppose that $p\equiv 3\pmod{4}$. Let $\bbO$ (resp.~$\bbO'$) be a
    maximal $O_F$-order in $D$ belonging to the principal
    (resp.~nonprinicipal) spinor genus. If $p=3$, then
    \begin{equation}
      \label{eq:148}
      h^1(\bbO)=h_\scc(\bbO)=\abs{\Tp^+(D)}=1=h^1(\bbO')=h_\scc(\bbO')=\abs{\Tp^{-}(D)}. 
    \end{equation}
     If $p\geq 7$, then
\begin{align*}
h^1(\bbO)&=  \frac{\zeta_F(-1)}{2}+\left(11-3\Lsymb{2}{p}\right)\frac{h(-p)}{8}+\frac{h(-3p)}{6},\\
h^1(\bbO')&=\frac{\zeta_F(-1)}{2}+\left(3-3\Lsymb{2}{p}\right)\frac{h(-p)}{8}+\frac{h(-3p)}{6},\\
\abs{\Tp^+(D)}= h_\scc(\bbO)&=\frac{\zeta_F(-1)}{4}+\left(17-\Lsymb{2}{p}\right)\frac{h(-p)}{16}+\frac{h(-2p)}{8}+\frac{h(-3p)}{12},\\
\abs{\Tp^-(D)}=h_\scc(\bbO')&=\frac{\zeta_F(-1)}{4}+\left(9-9\Lsymb{2}{p}\right)\frac{h(-p)}{16}+\frac{h(-2p)}{8}+\frac{h(-3p)}{12}. 
\end{align*}
  \end{prop}
  \begin{proof}
 We
focus on the values of $\Delta(B, \bbO)$ and $\Delta(B, \bbO')$ for
$B\subseteq K_1$,
since the rest of the data required by
(\ref{xeq:214}) and \eqref{eq:139}  are routine calculations (see
\cite[\S7]{li-xue-yu:unit-gp} and
\cite[\S3]{xue-yang-yu:num_inv}). From part (2) of Theorem~\ref{thm:selectivity},  exactly one  of the two spinor
genera of maximal orders is selected by  $B$, so we have
\begin{equation}
  \label{eq:111}
\Delta(B, \bbO)+\Delta(B, \bbO')=1\qquad \text{if}\quad B\subseteq K_1. 
\end{equation}



We first treat the case $p\geq 7$. Define 
  \begin{equation}\label{eq:6-110}
  B_{1,2}\coloneqq O_F+\grq
  O_{K_1}=\Z+\Z\sqrt{p}+\Z\sqrt{-1}+\Z(1+\sqrt{-1})(1+\sqrt{p})/2,
  \end{equation}
  where $\grq=2\vartheta O_F$ is the unique dyadic prime of $F$. According to
  \cite[\S7]{li-xue-yu:unit-gp}, we have 
\begin{equation}
  \label{eq:109}
\scrB=\{O_F[\sqrt{-1}],\quad  B_{1,2}, \quad O_{K_1}, \quad O_{K_2}, \quad
O_{K_3}\} \quad\text{if} \quad p\geq 7.  
\end{equation}
Here we have used the fact that $O_F[\sqrt{-\varepsilon}]$ coincides
with the maximal order $O_{K_2}$ in
$K_2= F(\sqrt{-2})$ by \cite[Proposition~2.6]{xue-yang-yu:num_inv}
(see also \cite[Lemma~7.2.5]{li-xue-yu:unit-gp}). Moreover,
\[\scrB^1=\scrB\smallsetminus \{O_{K_2}\} \quad\text{if} \quad p\geq 7. \]

It is clear from (\ref{eq:146}) that
$\bbO_0\cap F(i)=O_F[i]\simeq O_F[\sqrt{-1}]$, and
$\bbO_0\cap F(j)=O_{F(j)}\simeq O_{K_1}$. Since $\bbO_0$ belongs to
the principal spinor genus, we have
\begin{align}
  \label{eq:112}
\Delta(O_F[\sqrt{-1}], \bbO)&=1,&  \Delta(O_{K_1}, \bbO)&=1, \qquad
\text{and hence}\\
  \label{eq:114}
\Delta(O_F[\sqrt{-1}], \bbO')&=0, & \Delta(O_{K_1}, \bbO')&=0
\qquad \text{ by  (\ref{eq:111})}. 
\end{align}
Applying (\ref{eq:11}) with $B=B_{1,2}$ and $B'=O_{K_1}$, we obtain
from (\ref{eq:4}) that 
\begin{equation}
\Delta(B_{1,2}, \bbO)=\frac{1}{2}\left(1+\Lsymb{2}{p}\right),\qquad
\Delta(B_{1,2}, \bbO')=\frac{1}{2}\left(1-\Lsymb{2}{p}\right). 
\end{equation}

The data required by 
(\ref{xeq:214}) and \eqref{eq:139} are gathered in the following table (see
\cite[\S8.2]{li-xue-yu:unit-gp}): 
\begin{center}
\renewcommand{\arraystretch}{1.5}
  \begin{tabular}{*{7}{|>{$}c<{$}}|}
\hline
B & \abs{\bmu(B)}& w(B) & h(B)/h(F)& s(B, \bbO)& \Delta(B,
                                                           \bbO) &
                                                                   \Delta(B, \bbO')\\
\hline

O_F[\sqrt{-1}]&   4 & 2 & \big(2-\Lsymb{2}{p}\big)h(-p)& 1& 1 & 0 \\
\hline 

 B_{1,2}&   4& 4&
                                         \big(2-\Lsymb{2}{p}\big)h(-p)&
                                                                        1& 
                                                                        \frac{1}{2}\big(1+\Lsymb{2}{p}\big)&\frac{1}{2}\big(1-\Lsymb{2}{p}\big) 
                                                                                                              \\

\hline 
O_{K_1} &   4& 4& h(-p)& 1& 1 & 0 \\
    \hline

O_{K_3} &   6& 3 & h(-3p)/2 & 0&1& 1\\
    \hline \hline
O_{K_2} &   2& 2& h(-2p)& 0& 1 & 1 \\
\hline         
  \end{tabular}
\end{center}
Here the data for $O_{K_2}$ is separated from the rest to
emphasize that $O_{K_2}$ is the unique order in $\scrB\smallsetminus \scrB^1$. 
The class number formulas for $p\geq 7$ in the proposition now follow by a straightforward calculation. 

Next, suppose that $p=3$.  According to
\cite[Theorem~1.6]{li-xue-yu:unit-gp}, we have $\abs{\Tp(D)}=2$,
and hence every spinor genus of maximal orders contains precisely one
type. Now it follows from Lemma~\ref{lem:bijection-prin-type} that
\[h_\scc(\bbO)=\abs{\Tp^+(D)}=1=\abs{\Tp^{-}(D)}=h_\scc(\bbO').\]
To compute $h^1(\bbO)$ and $h^1(\bbO')$, note that  $K_1=K_3=\Q(\sqrt{3},
\sqrt{-1})$ in this case. Let $B_{1,3}\coloneqq \Z[\sqrt{3}, (1+\sqrt{-3})/2]$ as in
\cite[\S 6.2.6]{xue-yang-yu:ECNF}, which has conductor $\sqrt{3}O_F$
in $O_{K_1}$.  Since $\sqrt{3}O_F$ represents the unique nontrivial
element of the
order $2$ group $\Pic_+(O_F)$ here,   we combine (\ref{eq:11}) and \eqref{eq:y3} to obtain
\[ \Delta(B_{1,3}, \bbO)=0, \qquad \Delta(B_{1,3}, \bbO')=1.\]
The relevant data required by (\ref{xeq:214}) is now given by the
following table 
\begin{center}
\renewcommand{\arraystretch}{1.3}
  \begin{tabular}{*{7}{|>{$}c<{$}}|}
\hline
B & \abs{\bmu(B)}& w(B) & h(B)/h(F)& s(B, \bbO)& \Delta(B,
                                                           \bbO) &
                                                                   \Delta(B, \bbO')\\
\hline

O_F[\sqrt{-1}]&   4 & 2 & 1& 1& 1 & 0 \\
\hline 

 B_{1,2}&   4& 4& 1&1&0 &1\\
\hline
B_{1,3} &  6 & 3 & 1 & 1& 0 & 1\\
\hline
O_{K_1} &   12& 12& 1& 1& 1 & 0 \\
\hline 
  \end{tabular}
\end{center}
Recall that $\zeta_F(-1)=1/6$ by \cite[\S 6.2.6]{xue-yang-yu:ECNF}
again, we get  $h^1(\bbO)=h^1(\bbO')=1$ when $p=3$. 
  \end{proof}
  \begin{rem}
Summing up the formulas 
     for $\abs{\Tp^+(D)}$ and
$\abs{\Tp^-(D)}$, we  recover the type
number formula of $\abs{\Tp(D)}$ for $p\equiv 3\pmod{4}$ and $p\geq 7$, which was previously computed in
\cite[(4.7)]{xue-yu:type_no}:   
    \begin{equation}
  \label{eq:245}
  \abs{\Tp(D)}=\frac{\zeta_F(-1)}{2}+\left(13-5\Lsymb{2}{p}\right)\frac{h(-p)}{8}+
  \frac{h(-2p)}{4}+\frac{h(-3p)}{6}.  
\end{equation}
    Similar to the $p\equiv 1\pmod{4}$ case discussed in \S\ref{sect:end-ring}, the  type
    numbers $\abs{\Tp^+(D)}$ and $\abs{\Tp^-(D)}$ can be interpreted
    as proper class numbers of even definite quaternary quadratic
    forms of discriminant $4p$ within a fixed genus. In this setting such formulas
    have been previously obtained by Ponomarev \cite{ponomarev:aa1981} and by
    Chan and Peters \cite[\S3]{chan-peters} using another method.
  \end{rem}

\begin{rem}\label{rem:Bernoulli}
  Let $p\in \bbN$ be an arbitrary prime and $\grd_F$ be the discriminant of
  $F$, i.e.~$\grd_F=p$ if $p\equiv 1\pmod{4}$, and $\grd_F=4p$
  otherwise.  The special value $\zeta_F(-1)$ can be calculated by
  Siegel's formula (see \cite[Table~2, p.~70]{Zagier-1976-zeta},
  \cite[Theorem~I.6.5]{van-der-Geer-HMS}):
\begin{equation}\label{eq:113}
  \zeta_F(-1)=\frac{1}{60}\sum_{\substack{b^2+4ac=\grd_F\\ a,c>0}} a,
\end{equation}
where  $b\in \Z$ and $a,c\in \bbN_{>0}$.  According to
\cite[Remark~3.4]{xue-yu:counting-av}, we can also express
$\zeta_F(-1)$ in terms of the second generalized Bernoulli number
(see \cite[\S4.2]{AIK:2014}, \cite[Exercise~4.2(a)]{Washington-cyclotomic}): 
\begin{equation}
    \zeta_F(-1)=\frac{B_{2,\chi_{_F}}}{2^3\cdot 3}=\frac{1}{24\grd_F}\sum_{a=1}^{\grd_F}\chi_{_F}(a)a^2,  
\end{equation}
where $\chi_{_F}$ is the quadratic character associated to $F/\Q$, that
is, $\chi_{_F}$ is the unique even (i.e.~$\chi_{_F}(-1)=1$) quadratic character
of conductor $\grd_F$. 
\end{rem}

\section{Class number calculations when $p\not\equiv 3\pmod{4}$}
\label{sec:class-number}
\numberwithin{thmcounter}{section}
\numberwithin{table}{section}

Throughout this section, we assume that $p\in \bbN$ is a prime number
that is not congruent to $3$ modulo $4$. Let $F=\Q(\sqrt{p})$, and
$D=D_{\infty_1, \infty_2}$ be the totally definite quaternion
$F$-algebra that splits at all finite primes of $F$.  If $p\equiv
1\pmod{4}$, then $A\coloneqq \Z[\sqrt{p}]$ is a suborder of $O_F$
of index $2$.  The main goal of this section is to compute
the class numbers
\[h^1(\calO_r)\coloneqq \abs{D^1\bsh \whD^1/\wcO_r^1} \qquad\text{
    for }\quad r\in \{8,16\} \text{ and } p\equiv
1\pmod{4},  \]
 where
$\calO_8$ and $\calO_{16}$ are the $A$-orders in $D$ defined in
\eqref{eq:31}.  In particular, $\calO_r\cap F=A$ for both $r\in \{8,
16\}$. 
As far as we know, there are no systematic ways
to compute $h^1(\calO)$ when the center of the quaternion order $\calO$
is non-maximal (other than the Selberg trace formula). Thus we adopt some ad-hoc methods. 

Let $\varepsilon>1$ be the fundamental unit of $O_F^\times$. By
\cite[Lemma~2.4]{xue-yang-yu:num_inv}, $\Nm_{F/\Q}(\varepsilon)=-1$ if
$p\not\equiv 3\pmod{4}$, which implies that in this case
\begin{equation}
    \label{eq:133}
      \Pic_+(O_F)\cong\Pic(O_F) \quad\text{and}\quad  O_{F, +}^\times= O_F^{\times
        2}=\dangle{\varepsilon^2}.
\end{equation}

\begin{lem}\label{lem:bije-eichler-p1mod4}
  Suppose that $p\not\equiv 3\pmod{4}$.  Then for every  Eichler $O_F$-order $\calO$ in
  $D$, the following canonical projection is a bijection
  \begin{equation}\label{eq:C.5}
D^1\backslash \whD^1/\wcO^1\to D^\times\backslash
  \whD^\times/\whF^\times\wcO^\times. 
  \end{equation}
Moreover, $h^1(\calO)=h(\calO)/h(F)$. 
\end{lem}
\begin{proof}
  According to
\cite[Corollary~18.5]{Conner-Hurrelbrink}, the narrow class number 
$h_+(O_F)$ is odd, and $O_{F,
  +}^\times=O_F^{\times 2}$ by (\ref{eq:133}). It follows that 
\begin{equation}\label{eq:C.6}
\whF^\times=F^\times_+\whF^{\times 2}
  \whO_F^\times,   \quad\text{and}\quad F^\times_+\cap \whF^{\times 2}
  \whO_F^\times=F^{\times 2}O_{F, +}^\times=F^{\times 2}.
\end{equation}
From  \cite[Theorem~4.1]{vigneras}, 
$\Nr(D^\times)=F^\times_+$, so 
$\Nr(D^\times\whF^\times\wcO^\times)=F^\times_+\whF^{\times 2}
  \whO_F^\times=\whF^\times$.  
Thus for any $x\in \whD^\times$, there
exists $\alpha\in D^\times$ and $y\in \whF^\times\wcO^\times$ such
that $\Nr(\alpha x y)=1$, which shows that the map in \eqref{eq:C.5} is
surjective. 

If $D^1x_i\wcO^1$ with $x_i\in \whD^1$ for $i=1, 2$ are mapped to the same element, then there exists $\alpha\in D^\times$ and
$y\in \whF^\times\wcO^\times$ such that $x_1=\alpha x_2 y$. Taking the
reduced norm on both sides, we get
$\Nr(\alpha)=\Nr(y^{-1})\in F^\times_+\cap (\whF^{\times 2} \whO_F^\times)$.  It
follows from \eqref{eq:C.6} that there exists $a\in F^\times$ such
that $\beta:=\alpha a^{-1}\in D^1$.  
On the other hand,
\begin{equation}\label{eq:C.7}
  (\whF^\times\wcO^\times)\cap \whD^1=\wcO^1.
\end{equation}
If we put $u:=ay$, then $u\in \wcO^1$ and  $x_1=\beta x_2 u$. This 
 shows that the map in  \eqref{eq:C.5} is injective. 

 Lastly, we prove that $h^1(\calO)=h(\calO)/h(F)$.  The ideal class
 group $\Cl(O_F)$ acts naturally on the class set $\Cl(\calO)$ of locally
 principal right $\calO$-ideal classes by
\begin{equation}
\Cl(O_F)\times \Cl(\calO)\to \Cl(\calO),\qquad ([\gra],
[I])\mapsto [\gra I].   
\end{equation}
Let $\overline{\Cl}(\calO)$ be the set of orbits of this action, which can be described adelically as
\begin{equation}\label{eq:C.8}
  \overline{\Cl}(\calO)\simeq D^\times\backslash \whD^\times/\whF^\times\wcO^\times.
\end{equation}
Since $h(F)$ is odd, the action is free by
\cite[Corollary~2.5]{xue-yu:type_no}, and hence
$h^1(\calO)=\abs{\overline{\Cl}(\calO)}=h(\calO)/h(F)$.
\end{proof}


From now on we assume that $p\equiv 1\pmod{4}$ for the rest of this
section. We first show that the narrow class number of
$A=\Z[\sqrt{p}]$ is odd. 

\begin{lem}\label{lem:odd-class-no}
Put    $\varpi\coloneqq [O_F^\times:A^\times]$ . Then
  \begin{gather}
   \Pic_+(A)\cong\Pic(A),\qquad  A_+^\times=A^{\times2}\qquad \text{and} \label{eq:134}\\
    \label{eq:131}
    h_+(A)=h(A)=\left(2-\Lsymb{2}{p}\right)\frac{h(O_F)}{\varpi}.
 \end{gather}
In particular, $h_+(A)$ is odd. 
\end{lem}
\begin{proof}
According to \cite[\S 4.2]{xue-yang-yu:num_inv}, $\varpi\in \{1,3\}$,
and $\varpi=1$  if $p\equiv 1\pmod{8}$.  
  Thus,  $\varepsilon^3\in A^\times$. As
  $\Nm_{F/\Q}(\varepsilon^3)=(-1)^3=-1$, we find that
  $\Pic_+(A)\cong\Pic(A)$. The equality $A_+^\times=A^{\times 2}$ can
  be verified case by case according to $\varpi=1$ or $3$
  (See~\cite[\S4.3]{xue-yang-yu:num_inv}).  To compute the class
  number $h(A)$, we apply directly Dedekind's formula
  \cite[p.~95]{vigneras} for the class number of quadratic orders.
  Lastly, according to \cite[Corollary~18.4]{Conner-Hurrelbrink}, $h(O_F)$ is
  odd for every prime $p$. It is clear from (\ref{eq:131}) that
  $h_+(A)$ is odd as well.
  \end{proof}

Let us recall the definition of the $A$-oder $\calO_{16}$. Fix a
maximal $O_F$-order $\calO_1$ in $D$ and also an
identification $\calO_1\otimes \Z_2= \Mat_2(O_{F_2})$.  Then 
 $\calO_{16}$ is defined to be 
 the unique suborder of $\calO_1$ of index $16$ such
that  $\calO_{16}\otimes
\Z_2= \Mat_2(A_2)$.

\begin{lem}\label{lem:biject-O16}
The following canonical projection is a bijection
  \begin{equation}\label{eq:CC.5}
D^1\backslash \whD^1/\wcO_{16}^1\to D^\times\backslash
  \whD^\times/\whF^\times\wcO_{16}^\times. 
\end{equation}
In particular, $h^1(\calO_{16})$ coincides with the type number $t(\calO_{16})$.
\end{lem}

\begin{proof}
  Once we prove that the equalities  \eqref{eq:C.6}   and
  \eqref{eq:C.7} hold again when $O_F$ and $\calO$ are replaced by $A$ and
  $\calO_{16}$ respectively, the same proof as that  of
  Lemma~\ref{lem:bije-eichler-p1mod4} applies to the bijectivity of
  (\ref{eq:CC.5}) here.

  From Lemma~\ref{lem:odd-class-no}, $h_+(A)$ is
  odd, and $A_+^\times=A^{\times 2}$, so the equalities in 
\eqref{eq:C.6} still hold if $O_F$ is replaced by $A$. 
If $\Nr(au)=1$ for some $a\in \whF^\times$ and
$u\in \wcO_{16}^\times$, then
$a^2=\Nr(u^{-1})\in \Nr(\wcO_{16}^\times)=\whA^\times$, which implies
that $a\in \whO_F^\times$. Since
$[\whO_F^\times:
\whA^\times]=[O_{F_2}^\times:A_2^\times]=2-\Lsymb{2}{p}$ (which is odd),  we
get $a\in \whA^\times\subseteq \wcO_{16}^\times$, and hence
$au\in \wcO_{16}^1$.  This shows that
$ (\whF^\times\wcO_{16}^\times)\cap \whD^1=\wcO_{16}^1$.  The rest of
the proof of (\ref{eq:CC.5}) runs the same as that of
Lemma~\ref{lem:bije-eichler-p1mod4}.

Lastly, recall that $\calN(\calO_{16})=\whF^\times\wcO_{16}^\times$ as
in (\ref{eq:41}). It follows that the cardinality of the right hand
side of (\ref{eq:CC.5}) is precisely the type number $t(\calO_{16})$,
which is calculated in \cite[(4.12)]{xue-yu:type_no} and reproduced in
\eqref{eq:92} (resp.~\eqref{zeq:104}) for $p\geq 13$ (resp.~$p=5$). 
\end{proof}

Keep  the maximal order $\calO_1\subset D$ fixed as before.   We write
$\calO_8$ and $\calO_4$ for the unique $A$-suborders of $\calO_1$ of
index $8$ and $4$ respectively such that
\begin{equation}
  \label{eq:C.1}
   \calO_8\otimes \Z_2=
\begin{pmatrix}
  A_2 & 2 O_{F_2} \\  O_{F_2} & O_{F_2}\\
\end{pmatrix},\qquad    \calO_4\otimes \Z_2=
\begin{pmatrix}
  O_{F_2} & 2 O_{F_2} \\  O_{F_2} & O_{F_2}\\
\end{pmatrix}.
\end{equation}
By definition, 
$\calO_4$ is an Eichler order of level $2O_F$. Clearly, $A/2O_F\cong
A_2/2O_{F_2}\cong \F_2$.


\begin{prop}\label{prop:class-number-O8}
  Let $\zeta_F(s)$ be the Dedekind zeta function of $F=\Q(\sqrt{p})$,
  and   $h(-p)$ be the class number $\Q(\sqrt{-p})$. Then 
  \begin{equation}
    \label{eq:C.2}
    h^1(\calO_8)=\frac{3}{2}\left(4-\Lsymb{2}{p}\right)\zeta_F(-1)+\left(2-\Lsymb{2}{p}\right)\frac{h(-p)}{8}.
  \end{equation}
\end{prop}
To compute $h^1(\calO_8)$, we first study its relation with $h^1(\calO_4)$. 
\begin{lem}\label{lem:relating-h1}
We have 
  \begin{equation}
    \label{eq:64}
    h^1(\calO_8)=
    \begin{cases}
      h^1(\calO_4)\qquad &\text{if } p\equiv 1\pmod{8},\\
      3h^1(\calO_4)-2\abs{S}&\text{if } p\equiv 5\pmod{8}, 
    \end{cases}
  \end{equation}
where $S$ denotes the following set 
  \[\{[x]=D^1x\wcO_4^1\in  D^1\backslash\whD^1/\wcO_4^1\,\mid\, \exists \zeta\in D^1\cap x\wh\calO_4^1
    x^{-1} \text{ such that } \ord(\zeta)=3\}.\]
\end{lem}
\begin{proof}
  Let
  $\theta: D^1\backslash \whD^1/\wcO_8^1\to D^1\backslash
  \whD^1/\wcO_4^1$ be the canonical projection map.  If $p\equiv
  1\pmod{8}$, then $O_{F_2}=\Z_2\times \Z_2$, and hence
  $A_2^\times=O_{F_2}^\times$. It follows that $\wcO_8^1=\wcO_4^1$ in
  this case, so $\theta$ is a bijection. 

Assume that $p\equiv 5\pmod{8}$ for the rest of the proof. For each
$[x]=D^1x\wcO_4^1$, 
\[\theta^{-1}([x])=D^1\backslash
  (D^1x\wcO_4^1)/ \wcO_8^1. \]
 Multiplying $D^1x\wcO_4^1$ from the left by $x^{-1}$
 induces bijections
  \begin{equation}
    \label{eq:C.4}
    D^1\backslash
  (D^1x\wcO_4^1)/ \wcO_8^1\simeq (x^{-1}D^1 x)\backslash
  (x^{-1}D^1x\wcO_4^1)/\wcO_8^1\simeq (x^{-1}D^1 x\cap \wcO_4^1)\backslash
  \wcO_4^1/\wcO_8^1.
  \end{equation}
From \eqref{eq:C.1},
 $(\calO_8\otimes \Z_2)^1$ is normal in $(\calO_4\otimes \Z_2)^1$, and 
\begin{equation}
 (\calO_4\otimes \Z_2)^1/(\calO_8\otimes
  \Z_2)^1 \simeq \F_4^\times,\qquad   \begin{pmatrix}
    a & b \\ c & d 
  \end{pmatrix}\mapsto (a \bmod 2O_{F_2}).
\end{equation}
Clearly,  $\wcO_4^1/\wcO_8^1\simeq  (\calO_4\otimes \Z_2)^1/(\calO_8\otimes
  \Z_2)^1$, and the action of $(x^{-1}D^1 x\cap \wcO_4^1)$ on $
\wcO_4^1/\wcO_8^1$ factors through the quotient \[(x^{-1}D^1 x\cap
  \wcO_4^1)\to   \wcO_4^1/\wcO_8^1. \]
Hence $\abs{\theta^{-1}([x])}\in \{1, 3\}$, and it takes value $1$ if
and only if the above homomorphism is surjective. 
If  $ \zeta'$ is an element of  $ (x^{-1}D^1 x\cap
\wcO_4^1)$  with  $3\ddiv \ord(\zeta')$, 
then it follows from \cite[\S2.8]{xue-yang-yu:num_inv} that
$F(\zeta')\simeq F(\sqrt{-3})$ and $\zeta'^2\pm \zeta'+1=0$. This
automatically implies that $\zeta'\not \in \wcO_8^1$ since
$\Tr(\xi)\equiv 0\pmod{2\whO_F}$ for every
$\xi\in \wcO_8^1$.
Therefore, $\abs{\theta^{-1}([x])}=1$ if and only if $[x]\in S$. In other words,  $\theta$ is a $3:1$ cover ``ramified''
above the set $S$, and the lemma follows.  
\end{proof}




\begin{proof}[Proof of Proposition~\ref{prop:class-number-O8}]
  First, suppose that $p\equiv 1\pmod{8}$. Combining the results of 
  Lemma~\ref{lem:relating-h1} and
  Lemma~\ref{lem:bije-eichler-p1mod4}, we see that
  $h^1(\calO_8)=h(\calO_4)/h(F)$, which also coincides with $t(\calO_8)$ by
  (\ref{eq:43}). The formulas for  $h(\calO_4)$ and   $t(\calO_8)$ are computed in \cite[Lemma~4.2]{xue-yu:type_no} and
  \cite[(4.11)]{xue-yu:type_no} respectively, so we get
  \begin{equation}
    \label{eq:65}
    h^1(\calO_8)=\frac{9}{2}\zeta_F(-1)+\frac{h(-p)}{8}\qquad
    \text{if }  p\equiv 1\pmod{8}. 
  \end{equation}

Next, suppose that $p\equiv 5\pmod{8}$.  Similar to the previous case,
we combine
Lemma~\ref{lem:bije-eichler-p1mod4} with
\cite[Lemma~4.2]{xue-yu:type_no} to obtain
  \begin{equation}\label{eq:C.13}
    h^1(\calO_4)=\frac{5}{2}\zeta_F(-1)+h(-p)/8+h(-3p)/3,  
  \end{equation}
where  $h(-3p)$ denotes the class number of $\Q(\sqrt{-3p})$. 
 
It remains to compute the cardinality of the set $S$ 
   defined in Lemma~\ref{lem:relating-h1}. Put
  \[\wt{S}:=\{[I]\in \Cl(\calO_4)\mid \exists \zeta\in
    \calO_l(I)^\times\text{ such that } \ord(\zeta)=3\}.\]
Here $\calO_l(I)$ denotes the left  order of $I$ as in (\ref{eq:34}). If we write
$I=D\cap x\wcO$ for some $x\in \whD^\times$, then $\calO_l(I)=D\cap
x\wcO x^{-1}$.  Therefore,  $[I]$ belongs to $\wt S$ if
and only if the image of $[I]$ belongs to  $S$ under
the following composition of maps
\[\Cl(\calO_4)\to \overline{\Cl}(\calO_4)\simeq
  D^\times\backslash\whD^\times/\whF^\times\wcO_4^\times\simeq D^1\backslash\whD^1/\wcO_4^1.\]
Since the action of $\Cl(O_F)$ on $\Cl(\calO_4)$ is free and
$\calO_l(\gra I)= \calO_l(I)$ for every $[\gra]\in \Cl(O_F)$, we see
that 
\begin{equation}\label{eq:C.11}
  \abs{\wt S}=h(F)\abs{S}.
\end{equation}

For simplicity, put $K=F(\sqrt{-3})$. For each
$\zeta\in \calO_l(I)^\times$ with $\ord(\zeta)=3$, we have
$F(\zeta)\simeq K$ and $F(\zeta)\cap \calO_l(I)\simeq O_K$ by
\cite[Proposition~3.1]{xue-yang-yu:num_inv}
(cf.~\cite[Table~7.2]{li-xue-yu:unit-gp}).  Thus each such
$\zeta\in \calO_l(I)^\times$ gives rise to an optimal embedding
$O_K\to \calO_l(I)$ and vice versa. For each $O_F$-order $\calO$ in
$D$, let $m(O_K, \calO, \calO^\times)$ be the number of
$\calO^\times$-conjugacy classes of optimal embeddings from $O_K$ into
$\calO$ (cf.~\cite[Section~3.2]{li-xue-yu:unit-gp}).  We claim that
\begin{equation}\label{eq:C.9}
m(O_K, \calO_l(I), \calO_l(I)^\times)=
\begin{cases}
2 \quad&\text{if}\quad  [I]\in \wt
S,\\
0 \quad &\text{otherwise.}
\end{cases}
\end{equation}
The equality clearly holds if $[I]\not\in \wt S$.  Suppose that
$[I]\in \wt S$. If $p>5$, then  
$\calO_l(I)^\times/O_F^\times$ is a finite group isomorphic to
$\bfC_3, \bfD_3$ or $\bfA_4$ according to \cite[Tables~4.1 and 4.2]{li-xue-yu:unit-gp}, where $\bfC_3$ denotes the cyclic group
of order $3$, 
$\bfD_3$ the dihedral group of order $6$, and $\bfA_4$ the alternating
group on $4$-letters.  However if
$\calO_l(I)^\times/O_F^\times\simeq \bfD_3$, then by
\cite[Proposition~4.3.5]{li-xue-yu:unit-gp}, $\calO_l(I)$ necessarily
contains an $O_F$-suborder isomorphic to the order 
$\scrO_3^\dagger$ in \cite[Table~4.3]{li-xue-yu:unit-gp}.   This leads to a contradiction since
$\scrO_3^\dagger$ has reduced discriminant $3O_F$ while $\calO_l(I)$
is an Eichler order of level $2O_F$.  
If $p=5$, then $h(\calO_4)=1$ by
\cite[Lemma~4.2]{xue-yu:type_no} (which implies that $\calO_\ell(I)=\calO_4$), and
 $\Mass(\calO_4)=1/[\calO_4^\times: O_F^\times]=1/12$ by a direct
 calculation using the mass formula \cite[Corollaire~IV.2.3]{vigneras}.
We find that $\calO_4^\times/O_F^\times\simeq \bfA_4$ when $p=5$ from
the classification in \cite[\S8.1]{li-xue-yu:unit-gp}.
In conclusion, if $[I]\in \wt S$, then $\calO_l(I)^\times/O_F^\times$
is isomorphic to either $\bfC_3$ or $\bfA_4$, so our claim follows
directly from
\cite[Proposition~3.4]{li-xue-yu:unit-gp}.


Lastly, we apply the trace formula \cite[Theorem~III.5.11]{vigneras}
to $O_K$ and $\calO_4$ to get
\begin{equation}\label{eq:C.10}
  \sum_{[I]\in \Cl(\calO_4)}m(O_K, \calO_l(I), \calO_l(I)^\times)=h(O_K)\left(1+\Lsymb{O_K}{2O_F}\right)=2h(K).
\end{equation}
Here $\Lsymb{O_K}{\grp}$ is the Eichler symbol \cite[p.~94]{vigneras} for a prime ideal $\grp\subseteq O_F$,
and it takes value $1$ at $\grp=2O_F$, which splits
in $K$. Combining \eqref{eq:C.11}, \eqref{eq:C.9} and \eqref{eq:C.10}, we get
\begin{equation}\label{eq:C.12}
  \abs{S}=\frac{2h(K)}{2h(F)}=\frac{h(-3p)}{2}.
\end{equation}
Plugging \eqref{eq:C.13}  and \eqref{eq:C.12} into the
equality $h^1(\calO_8)=3h^1(\calO_4)-2\abs{S}$ in Lemma~\ref{lem:relating-h1}, we obtain the
following class number formula
\begin{equation}
  \label{eq:63}
h^1(\calO_8)=\frac{15}{2}\zeta_F(-1)+\frac{3}{8}h(-p) \qquad
    \text{if }  p\equiv 5\pmod{8}.  
\end{equation}
Proposition~\ref{prop:class-number-O8} follows by 
combining (\ref{eq:65}) with (\ref{eq:63}) and interpolating. 
\end{proof}

\section{Classification of self-dual local lattices}
\label{sec:LC}

We consider some variants of lattices in symplectic spaces over local
fields that are used in this paper.  As usual,
$\delta_{ij}$ denotes Kronecker's
symbol, namely $\delta_{ij}=1$ or $0$ according to whether $i=j$ or
not.


\begin{lem}\label{LC.2}
  Let $K_0$ be a field, and $O_{K_0}\subseteq K_0$ be a Dedekind domain with quotient field
  $K_0$. 
  Let $K=\prod_i K_i$ be a commutative separable $K_0$-algebra, where
  each $K_i$ is a finite separable field extension of $K_0$, 
  and let $O_K$ be its maximal $O_{K_0}$-order. Let $(V,\psi)$ be a
  non-degenerate alternating $K_0$-valued $K$-module 
  such that $\psi(ax,y)=\psi(x,ay)$ for all $a\in K$ and all $x,y\in V$. Then 
  \begin{enumerate}
  \item There exists a self-dual $O_{K_0}$-valued $O_K$-lattice.
  \item Any
  two self-dual $O_{K_0}$-valued $O_K$-lattices $L$ and $L'$ are
  isometric. 
  \end{enumerate}
\end{lem}
\begin{proof}
Let $L$ and $L'$ be $O_K$-lattices in $V$.   Decompose $L=\oplus L_i$ and $L'=\oplus L_i'$ into $O_{K_i}$-modules
  with respect to the decomposition $O_K=\prod_i O_{K_i}$ into the product
  of Dedekind
  domains. Clearly this decomposition
  is orthogonal with respect to the pairing $\psi$. The
  $O_K$-lattice $L$ 
  is self-dual if and only if each component $L_i$ is
  self-dual. Moreover, $(L,\psi)\simeq (L',\psi)$ if and only if
  $(L_i,\psi)\simeq (L'_i,\psi)$ for all $i$. Thus, without loss of
  generality, we may assume that $K$ is a finite separable field
  extension of $K_0$.

    Let $\psi_K:V\times V\to K$ be the unique $K$-bilinear form such that
  $\Tr_{K/K_0}\circ \psi_K=\psi$. The form $\psi_K$ remains
  non-degenerate and alternating. Denote the inverse different of
  $K/K_0$ by 
  $\calD^{-1}_{K/K_0}$.  
   The dual lattice of $L$ with
  respect to $\psi$ is defined as 
\begin{equation}
  \label{eq:dual1}
   L^\vee\coloneqq \{x\in V\, |\, \psi(x,L)\subseteq O_{K_0}\, \}.  
\end{equation}
Since $\psi(x,L)\subseteq O_{K_0}$ if and only if $\psi_K(x, L)\subseteq
\calD_{K/K_0}^{-1}$, we  have
\begin{equation}
  \label{eq:LC.19}
     L^\vee=\{x\in V\, |\, \psi_K(x,L)\subseteq \calD_{K/K_0}^{-1}\, \}.  
\end{equation}

According to \cite[Proposition~1.3]{Shimura1963-AltHermForms}, there
exist a $K$-basis $\{x_1, \cdots, x_n, y_1, \cdots, y_n\}$ of $V$ 
and fractional $O_K$-ideals $\gra_1, \cdots, \gra_n$ such that
\begin{gather}
  \psi_K(x_i, x_j)=\psi_K(y_i, y_j)=0, \qquad \psi_K(x_i,
  y_j)=\delta_{ij}, \quad\text{for all } 1\leq i, j\leq n, \label{eq:LC.21}\\
  L=\sum_{i=1}^n O_K x_i+ \sum_{j=1}^n \gra_jy_j, \qquad   \gra_1\supseteq\gra_2\supseteq  \cdots \supseteq \gra_n. \label{eq:LC.20}
\end{gather}
The fractional ideals $\gra_1\supseteq  \cdots \supseteq \gra_n$ are
called the \emph{invariant factors} of $L$ with respect to $\psi_K$. 
A direct calculation shows that
\begin{equation}
  L^\vee=\calD_{K/K_0}^{-1}(\sum_{i=1}^n \gra_i^{-1} x_i+ \sum_{j=1}^n O_K y_j).
  \end{equation}
It follows that $L$ is self-dual if and only if
$\gra_i=\calD_{K/K_0}^{-1}$ for all $1\leq i\leq n$. Therefore, 
 a self-dual $O_K$-lattice in $(V, \psi)$ exists and is
unique up to isometry. 
\end{proof}
\begin{rem}
The first invariant factor $\gra_1$  is called the
\emph{norm} of $L$ with respect to $\psi_K$ since it coincides with the $O_K$-ideal
generated by $\psi_K(x, y)$ for all $x,y\in L$. Following
\cite[\S1.4]{Shimura1963-AltHermForms}, an $O_K$-lattice $M$ in
$(V, \psi_K)$ is said to be \emph{maximal} if $M$ is a maximal one
among the $O_K$-lattices in $(V, \psi_K)$ with the same norm. 
We have shown in the proof of Lemma~\ref{LC.2}
  that an $O_K$-lattice $L$ in $(V, \psi)$ is self-dual if and only if
  it is a maximal lattice with norm $\calD_{K/K_0}^{-1}$ in $(V,
  \psi_K)$. 
\end{rem}

Let $\gra$ be an invertible $O_K$-ideal in $K$. 
An $O_K$-lattice $L$ is said to be
\emph{$\gra$-modular} in $(V,\psi)$ if $\gra (L^\vee)=L$.

\begin{cor}\label{LC.mod}
 Keep the notation and assumption of Lemma~\ref{LC.2}.  For any
  invertible ideal $\gra$ in $K$, there exists a unique $\gra$-modular $O_K$-lattices
  in $(V,\psi)$ up to isometry.
\end{cor}
\begin{proof}
Once again, we reduce to the case $K$ being a finite
separable field extension of $K_0$. From the proof of Lemma~\ref{LC.2}, an
  $O_K$-lattice $L$  is
  $\gra$-modular if and only if all its invariant factors
  $\gra_i=\gra\calD_{K/K_0}^{-1}$. The corollary follows easily. 
\end{proof}
  
  %

We would like a similar result as Lemma~\ref{LC.2}
with a weaker condition than $O_K$ being a product of
Dedekind domains. The best that we can manage at the moment is for
Gorenstein orders over complete discrete valuation rings. 
Notation as in Lemma~\ref{LC.2}, 
an $O_{K_0}$-order $R$  in $K$ is
called {\it
  Gorenstein} if every short exact sequence of $R$-lattices
\[ 0\to R \to M \to N \to 0 \] 
splits. It is known that $R$ is Gorenstein if and only if its dual
$R^\vee\coloneqq \Hom_{O_{K_0}}(R,O_{K_0})$ is an invertible $R$-module 
\cite[Prop.~6.1, p.~1363]{Drozd-Kirichenko-Roiter-1967}; see also \cite[Proposition~3.5]{Gorenstein-orders-JPAA-2015}. 
\begin{lem}\label{LC.3}
  Let $K_0$ be a complete discrete valuation field, and $O_{K_0}$ be
  its valuation ring.  Let
  $K$, $(V,\psi)$ be as in Lemma~\ref{LC.2}, and let
  $R$ be a $O_{K_0}$-order in $K$. Assume that $V$ is free over $K$ and
  $R$ is Gorenstein. Then
  \begin{enumerate}
  \item There exists a self-dual $O_{K_0}$-valued {\it free} $R$-lattice
    in $(V,\psi)$. 
  \item Any
  two self-dual $O_{K_0}$-valued {\it free} $R$-lattices $L$ and $L'$ in
  $(V,\psi)$ are 
  isometric. 
  \end{enumerate}
\end{lem}

\begin{proof}
  Let $\psi_K:V\times V\to K$ be as before and by an abuse of notation, put 
\[ \calD^{-1}_{R/O_{K_0}}\coloneqq \{x\in
  K | \Tr_{K/K_0}(xR)\subset O_{K_0}\}.\] 
  One has $\calD^{-1}_{R/O_{K_0}}\simeq
  R^\vee$, which is  $R$-invertible
  by our assumption. Since $O_{K_0}$ is a
  complete discrete valuation ring, $R$ is a finite product of local rings. It follows that
  $\calD^{-1}_{R/O_{K_0}}=\delta^{-1}R$ for some element $\delta\in
  K^\times$. 
  Put $\psi_K'\coloneqq \delta \psi_K$. Then for any $R$-lattice $L$ in
  $(V, \psi)$, we have 
   \begin{equation}
     \label{eq:LC.22}
     L^\vee=\{x\in V\, |\, \psi_K(x,L)\subseteq \delta^{-1}R\, \}=\{x\in V\, |\, \psi_K'(x,L)\subseteq R\, \}.  
   \end{equation}
 Therefore, $L$ is self-dual with respect to $\psi$ if and only if it
 is so with respect to $\psi_K'$.  In particular, the existence and uniqueness of self-dual free $R$-lattices
in $(V, \psi)$ can be reduced to those of self-dual lattices in $(V,
\psi_K')$. Since $R$ is a product of local
Noetherian rings, any finitely generated stably free $R$-module is
free.   The desired results now follow from Lemma~\ref{LC.1} below. 
\end{proof}




\begin{lem}\label{LC.1}
  Let $R$ be a Noetherian commutative ring such that any finitely
  generated stably free
  $R$-module is free, and 
  let $L$ be a free $R$-module of rank $2n$.
  Then for any perfect alternating pairing 
  $\psi:L\times L\to
  R$, there is a Lagrangian basis with respect to $\psi$, 
  that is, a basis $\{x_1,\dots,
  x_n, y_1, \dots, y_n\}$ such that 
\[ \psi(x_i,y_j)=\delta_{ij}\quad \text{and} \quad
  \psi(x_i,x_j)=\psi(y_i,y_j)=0,\quad \forall\, i,j. \]
  In particular, there is only one perfect alternating pairing up to
  isomorphism.   
\end{lem}
\begin{proof}
  We prove  by induction on $n$. Let $\{e_1, \cdots, e_n\}$ be an
  $R$-basis of $L$. Since $\psi$ is a perfect pairing, the map $x\mapsto \psi(x,
  \cdot)$ establishes an isomorphism of  $R$-modules $L\to \Hom_R(L,
  R)$. 
  In particular, there exists $y_1\in L$ such that $\psi(e_1,
  y_1)=1$. Put $x_1=e_1$. 
  We note that  the $R$-submodule $\<x_1, y_1\>_R\subseteq L$ spanned by $x_1$ and $y_1$ is  \emph{free}
  of rank $2$. Indeed, if $ry_1\in Rx_1$ for some $r\in R$, then $r=\psi(x_1,
  ry_1)=0$. Let $L'\coloneqq \{x\in L\mid \psi(x,
  x_1)=\psi(x, y_1)=0\}$. For every $z\in L$, put $a=\psi(z, y_1)$,  
  $b=\psi(z, x_1)$, and $z'=z-ax_1+by_1$. Then $z'\in L'$, and hence $L=\<x_1, y_1\>_R\oplus L'$. Clearly, $L'$ is finitely
  generated and stably free. Thus it is free of rank $2n-2$ by our assumption on
  $R$. The restriction of $\psi$ to $L'$ is necessarily 
  perfect. Now the lemma follows by induction. 
\end{proof}

\begin{rem}\label{rem:LC}
  The condition for $R$ in Lemma~\ref{LC.1} is satisfied if $R$ is one
  of the following: a local ring, a Dedekind domain, a commutative
  Bass order \cite[Section~37]{curtis-reiner:1} \cite{Levy-Wiegand-1985}, a polynomial ring
  over a field (Serre's conjecture, now the Quillen-Suslin theorem),
  or any finite product of them.  In particular, if in
  Lemma~\ref{LC.3} we assume that $O_{K_0}$ is discrete valuation ring
  (without the completeness condition), and suppose further that $R$
  is a Bass $O_{K_0}$-order, then the results still hold. Indeed, in
  this case $R$ is a semi-local ring, so the invertible $R$-module
  $\calD^{-1}_{R/O_{K_0}}$ is again free, and we can still apply
  Lemma~\ref{LC.1} thanks to the preceding Remark.
\end{rem}

\begin{ex}
  We provide an example of a \emph{non-projective} self-dual lattice
  $L$ over a Bass order $R$ in a
  symplectic space. Let $d\in \Z$ be a square-free integer coprime to
  $6$, $K=\Q(\sqrt{d})$ and $R=\Z[6\sqrt{d}]$. 
     Let $(V, \psi_K)$ be a
  symplectic space over $K$ of dimension $4$ with a canonical basis
  $\{x_1, y_1, x_2, y_2\}$ satisfying (\ref{eq:LC.21}).  We put
\[ L=R_1 x_1 +\grb_1 y_1+ R_2x_2+\grb_2y_2,\]
  where $R_1=\Z[2\sqrt{d}], R_2=\Z[3\sqrt{d}]$ and
  $\grb_i=(R:R_i)=\{x\in K\mid xR_i\subseteq R\}$.  More explicitly,
  $\grb_1=3R_1$ and $\grb_2=2R_2$, and both of them are integral ideals
  of $R$.  From
 \cite[2.6(ii)]{Gorenstein-orders-JPAA-2015}, 
 $(R:\grb_i)=R_i$ for each $i$, which is also immediate from
 calculations. This implies that
\[ L=\{x\in V\mid \psi_K(x, L) \subseteq R\}\simeq \Hom_R(L, R).\]
The fractional $R$-ideal  $\<\psi_K(x,y)\mid x, y\in L\>_R$ coincides
with $R$ itself, so $L$ is also a maximal $R$-lattice with norm
$R$. Clearly, $L$ is not projective over $R$, otherwise both $R_i$ are
projective $R$-modules, which is absurd. 
\end{ex}

\section*{Acknowledgments}
The authors would like to express their gratitude to Tomoyoshi
Ibukiyama,  and
Chao Zhang for stimulating discussions.  
They thank the organizers of the 2019 Nagoya supersingular conference held September 30--October 4 of 2019 for their kind invitation.  
They also thank Sushi Harashita for his kind invitation to the RIMS supersingular conference held October 13--15, 2020, where the present work was presented.  
Xue is partially supported by the National Natural Science Foundation of China grant No.~12271410 and No.~12331002. Yu is partially supported
by the grants NSTC 109-2115-M-001-002-MY3 and 112-2115-M-001-009. Part of the  
manuscript was prepared during the first author's 2024 visit to Institute
of Mathematics, Academia Sinica. He thanks the institute for the warm
hospitality and great working conditions.


\def\cprime{$'$}

\end{document}